\documentclass[12pt]{amsart}
\usepackage{fullpage, amsmath, amsthm,amsfonts,amssymb,stmaryrd, mathrsfs}
\usepackage{hyperref}

\usepackage[pdftex]{color}
\usepackage[all]{xy}

\usepackage{latexsym}
\usepackage{amssymb}
\usepackage{amsmath}
\usepackage{amsfonts}
\usepackage{textcomp}
\usepackage{amscd}
\usepackage{amsthm}
\usepackage{mathrsfs}
\usepackage{amsxtra}
\usepackage{stmaryrd}
\usepackage[all]{xy}

\newtheorem{theorem}{Theorem}[subsection]
\newtheorem{lemma}[theorem]{Lemma}
\newtheorem{cor}[theorem]{Corollary}
\newtheorem{conj}[theorem]{Conjecture}
\newtheorem{prop}[theorem]{Proposition}
\theoremstyle{definition}
\newtheorem{defn}[theorem]{Definition}

\newtheorem{remark}[theorem]{Remark}
\newtheorem{convention}[theorem]{Convention}

\numberwithin{equation}{theorem}

\def\limind{\mathop{\oalign{{\rm lim}\cr
\hidewidth$\longrightarrow$\hidewidth\cr}}}
\def\limproj{\mathop{\oalign{{\rm lim}\cr
\hidewidth$\longleftarrow$\hidewidth\cr}}}

\def\Z{{\mathbb {Z}}_p}
\def\Q{{\mathbb {Q}}_p}

\def\G{{\rm GL}_2(\Q)}

\def\H{{\mathrm{H}}}

\def\ra{\rightarrow}
\def\ZZ{{\mathbb Z}}

\def\Qp{{\mathbb{Q}}_p}

\def\Zp{{\mathbb{Z}}_p}

\def\OO{\mathcal{O}}

\def\aa{\textbf{A}}

\def\int{\mathrm{int}}
\def\bd{\mathrm{bd}}

\def\calE{\mathcal{E}}

\def\rig{\mathrm{rig}}

\def\dif{\mathrm{dif}}

\def\D{\mathrm{D}}

\DeclareMathOperator{\rank}{rank}
\DeclareMathOperator{\Ind}{Ind}

\def\Sen{\mathrm{Sen}}

\def\m{(\varphi,\Gamma)}

\def\r{\mathcal{R}}

\begin{document}

\title{Triangulation of refined families}
\author{Ruochuan Liu}
\email{liuruochuan@math.pku.edu.cn}
\address{Beijing International Center for Mathematical Research\\ Peking University\\
Beijing, 100871\\
China}

\begin{abstract}
We prove the global triangulation conjecture for families of refined $p$-adic representations under a mild condition. That is, for a refined  family, the associated family of $\m$-modules admits a global triangulation on a Zariski open and dense subspace of the base that contains all regular non-critical points.
We also determine a large class of points which belongs to the locus of global triangulation. Furthermore, we prove that all the specializations of a refined family are trianguline. In the case of the Coleman-Mazur eigencurve, our results provide the key ingredient for showing its properness in a subsequent work \cite{DL13}.  
\end{abstract}

\maketitle

\setcounter{tocdepth}{2}
\tableofcontents

\section*{Introduction}
In the seminal work \cite{Ki03}, Kisin proved the Fontaine-Mazur conjecture for Galois representations attached to finite slope, overconvergent cuspidal eigenforms. The most significant part of his proof is showing that the dual of these representations satisfy the property that their restrictions on a decomposition group of $p$ have nonzero crystalline periods on which the crystalline Frobenius acts via multiplication with the $U_p$-eigenvalue. Furthermore, he conjectured that this property should characterize the Galois representations coming from finite slope overconvergent $p$-adic modular forms. This beautiful result inspired many important subsequent developments. In $p$-adic Hodge theory, Colmez introduced the notion of \emph{trianguline representations} reformulating this property in the framework of $\m$-modules over the Robba ring \cite{C07}; the notion of trianguline representations plays a key role in his construction of the $p$-adic local Langlands correspondence for $\mathrm{GL}_2(\Q)$. In the direction of Bloch-Kato conjecture, Bellaiche-Chenevier \cite{BC06} and Skinner-Urban \cite{SU06} applied some (different) variants of Kisin's result to construct elements of Selmer groups by deforming certain $p$-adic representations on eigenvarieties. More recently, Emerton \cite{E11} established the local-global compatibility of $p$-adic Langlands for $\mathrm{GL}_2/\mathbb{Q}$. As an application,  he confirmed the conjecture of Kisin. Nowadays, it is widely assumed that the condition of being trianguline at the $p$-adic place characterizes the Galois representations coming from finite slope overconvergent $p$-adic automorphic forms. In addition, it is conjectured\footnote{See for instance \cite[\S4]{BC06}.} that for a family of $p$-adic representations arising on eigenvarieties, the associated family of $\m$-modules admits a \emph{global triangulation} on a Zariski open and dense subspace of the base that contains all non-critical points.

The main objects of this paper are \emph{families of refined $p$-adic representations}. This notion was first introduced by Bellaiche-Chenevier \cite{BC06} for $p$-adic representations of $\mathrm{Gal}({\overline{\mathbb{Q}}_p/\Q})$ to encode the properties of families of Galois representations carried by eigenvarieties. In this paper, we first generalize this notion to $\mathrm{Gal}({\overline{\mathbb{Q}}_p/K})$-representations where $K$ is a finite extension over $\Q$. For technical reasons, we will assume throughout that our refined families are arithmetic families of $p$-adic representations, not just pseudocharacters as in Bellaiche-Chenevier's original definition. The main goal of this paper is to prove the global triangulation conjecture for such families under a mild condition. Namely, we will prove that a family of refined $p$-adic representations admits a global triangulation on a Zariski open and dense subspace of the base that contains all regular non-critical points. We also determine a large class of points which belongs to the locus of global triangulation. Furthermore, we will show that the specializations of refined families are all trianguline. Finally, as an application, we explicitly determine the local behavior of the family of $p$-adic representations carried by the Coleman-Mazur eigencurve. 

Our approach is largely inspired by Kisin's method of interpolating crystalline periods \cite{Ki03}. The first major step is to show that for a weakly refined family, the de Rham periods always coincide with the crystalline periods on which the Frobenius acts via multiplication with the prescribed eigenvalue. Furthermore, both of them form coherent sheaves on the base. To this end, we significantly refine Kisin's construction of finite slope subspaces by removing the ``$Y$-small" assumption. The value of this refinement is that it allows us to interpolate periods over all affinoid subdomains of the base, not only, as in the case of Kisin's original construction,  over $Y$-small affinoid subdomains.

The second major step is to show that for a given refined family, the crystalline periods of its exterior powers give rise to the desired global triangulation. Firstly, it is not difficult to see that the set of points by  which the crystalline periods of exterior powers of the family give rise to a triangulation of the specialization, which is named as the \emph{triangulation locus} of the family, is an analytic subspace of a Zariski open subspace of the base. Therefore, to prove the global triangulation conjecture, it reduces to show that the triangulation locus contains all regular non-critical points. We achieve this by combining the results of the first step and a result of Bellaiche-Chenevier on descent \cite{BC06}.

Indeed, by the recent works of Bellovin \cite{Bel13} and Kedlaya-Porttharst-Xiao \cite{KPX},  one knows that 
for an arithmetic family of $p$-adic representations over a rigid analytic space, the de Rham periods and crystalline periods on which the Frobenius acts via multiplication with an invertible function always form coherent sheaves on the base respectively. On the other hand, as mentioned above, we prove in this paper that for a weakly refined family, the de Rham periods always coincide with the crystalline periods on which the Frobenius acts via multiplication with the prescribed eigenvalue. This fact is a special feature for the eigen-families of $p$-adic representations, and turns out to be important for applications to eigenvarieties. For example, it follows that for the dual of the family of Galois representations carried by the Coleman-Mazur eigencurve, the de Rham periods coincide with the crystalline periods on which the Frobenius acts via multiplication with the $U_p$-eigenvalue.  This fact plays a key role in our subsequent work with Hansheng Diao proving the properness of the Coleman-Mazur eigencurve \cite{DL13}. 

Last but not least,  we should point out that after the work of this paper was finished,  some of our results were also obtained by other authors. In \cite{H}, Hellmann proves that the families of $\mathrm{Gal}(\overline{\mathbb{Q}}_p/\Q)$-representations carried by eigenvarieties for definite unitary groups over imaginary quadratic fields admit global triangulations as in our case. His strategy is to construct a map from the eigenvariety to the moduli space of (rigidified) trianguline $\m$-modules. 
In \cite{KPX}, Kedlaya-Pottharst-Xiao establish the 
finiteness of cohomology for arithmetic families of $\m$-modules. As an application, they show that an arithmetic family of $p$-adic representations that is densely pointwise refined in the sense of Mazur admits a global triangulation over a large subspace of the base. More recently,  in his thesis \cite{Be}, John Bergdall applies the techniques of \cite{KPX} to refined families and gives a new proof of our result that the triangulation locus contains all regular non-critical points in the case of $\mathrm{Gal}(\overline{\mathbb{Q}}_p/\Q)$-representations. 

In the following we will explain the main results of the paper and the idea of proofs in more detail.
\subsection{Finite slope subspaces}

We fix a finite extension $K$ of $\Q$ in $\overline{\mathbb{Q}}_p$, and fix a uniformizer $\pi_K$ of $K$. Let 
$G_K=\mathrm{Gal}(\overline{\mathbb{Q}}_p/K)$. Let $K_0$ be the maximal unramified extension of $\Q$ contained in $K$, and let $f=[K_0:\Q]$. Let $E$ be the Galois closure of $K$ in $\overline{\mathbb{Q}}_p$,  and let $\mathrm{H}_K$ be the set of $\Q$-embeddings of $K$ into $\overline{\mathbb{Q}}_p$ (hence into $E$). For a $K\otimes_{\Q}E$-module $M$, set 
\[
M_{\tau}=M\otimes_{K\otimes_{\Q}E}(K\otimes_{K,\tau}E)
\] 
for any $\tau\in \mathrm{H}_K$.  We may identify $M$ with $\oplus_{\tau\in\mathrm{H}_K}M_\tau$.
For any $m\in M$, let $m_\tau$ denote the projection of $m$ onto $M_\tau$.

We now assume that $X$ is a rigid analytic space over $E$, and let $V_X$ be a family of $p$-adic representations of $G_K$ of dimension $d$ over $X$. Suppose the Sen polynomial \footnote{See Definition \ref{def:Sen-operator} for the definition of Sen polynomial.} for $V_X$ is $TQ(T)$ for some $Q(T)\in (K\otimes_{\Q}\OO(X))[T]$. For $\tau\in\H_K$, following the notation introduced above, $Q(T)_\tau$ denotes the projection of $Q(T)$ onto the $\tau$-isotypic component of $ (K\otimes_{\Q}\OO(X))[T]=K\otimes_{\Q}\OO(X)[T]$. Let $\alpha\in \OO(X)^\times$ be an invertible rigid analytic function on $X$.  We define finite slope subspaces of $X$ with respect to the pair $(\alpha, V_X)$ as follows.

\begin{defn}\label{def:fs-space}
For such a triple $(X,\alpha, V_X)$, we call an analytic subspace\footnote{In this paper, the terminology "analytic subspaces" refers to Zariski closed rigid analytic subspaces.} $X_{fs}\subset X$ a \emph{finite slope subspace of $X$ with respect to the pair $(\alpha,V_X)$} if it satisfies the following conditions.
\begin{enumerate}
\item[(1)]For every integer $j\leq0$ and $\tau\in \mathrm{H}_K$, the subspace $(X_{fs})_{Q(j)_{\tau}}$, which denotes the complement of the vanishing locus of  $Q(j)_\tau$ on $X_{fs}$, is Zariski open and dense in $X_{fs}$.
\item[(2)]For any morphism $g:M(R)\ra X$, where $R$ is an $E$-affinoid algebra, if $g$ factors through $X_{Q(j)_\tau}$ for every integer $j\leq0$ and $\tau\in\mathrm{H}_K$, then it factors through $X_{fs}$ if and only if the natural map
\begin{equation}\label{eq:cris-dR}
   \iota_{n,K}: K\otimes_{K_0}(\mathrm{D}^{\dag}_{\rig}(g^*(V_X))^{\varphi^f=g^*(\alpha),\Gamma=1}\ra \mathrm{D}_{\dif}^{+,fn}(g^*(V_X))^{\Gamma}
   \end{equation}
is an isomorphism for all sufficiently large $n$ (indeed for all $n\geq n(V_R)$)\footnote{See \S1.3 for the definition of the functor $\mathrm{D}_{\dif}^{+,n}$.}.
\end{enumerate}
\end{defn}

The above definition is a $\m$-module theoretical interpretation of Kisin's original definition of finite slope subspaces except that we relax the assumption on $g$. That is, we do not require that $g$ is $\alpha$-small in the sense of \cite[(5.2)]{Ki03}. It is not difficult to see that our finite slope subspace $X_{fs}$ (assuming its existence and uniqueness) coincides with the one introduced by Nakamura \cite{N11} which generalizes of Kisin's finite slope subspaces to finite extensions of $\Q$ (see Remark \ref{rem:fs-same}).

The idea for introducing finite slope subspaces is to cut out the maximal analytic subspace $X_{fs}$ of $X$ such that $Q(j)_\tau$ is not identically $0$ on any component of $X_{fs}$ for any $j\leq0$ and $\tau\in\mathrm{H}_K$, and for
any affinoid subdomain $M(S)$ of $X$, the natural maps
\begin{equation}\label{eq:intro-natural}
\qquad K\otimes_{K_0}\D_{\rig}^{\dag}(V_X|_{M(S)\cap X_{fs}})^{\varphi^f=\alpha,\Gamma=1}
\ra(\D_{\dif}^{+,fn}(V_X|_{M(S)\cap X_{fs}})/(t^k))^{\Gamma}
\end{equation}
are isomorphisms for all sufficiently large $k$.
As it was already pointed out by Kisin \cite[(5.5)(5)]{Ki03}, the ``$Y$-small" assumption in \cite[Proposition 5.4]{Ki03} is due to some technical obstacle to solve a certain Frobenius equation over the relative crystalline period ring. We get over this difficulty by using the relative extended Robba ring which is much bigger than  the relative crystalline period ring and sufficient to solve this equation. 

\begin{theorem}(Theorem \ref{thm:fs})\label{thm:intro-fs}
The rigid analytic space $X$ has a unique finite slope subspace with respect to the pair $(\alpha, V_X)$.
\end{theorem}

More importantly, we will prove that if $k$ is bigger than the valuation of $\alpha$ in $S$, then (\ref{eq:intro-natural}) is an isomorphism. This result is crucial for later applications to refined families. 

\begin{theorem}(Theorem \ref{thm:fs-sheaf})\label{thm:intro-fs-sheaf}
Let $M(S)$ be an affinoid subdomain of $X_{fs}$. Then
for any $n\geq n(V_S)$ and $k>\log_{|\pi_K^{-1}|}|\alpha^{-1}|_{\mathrm{sp}}$ where the spectral norm is taken in $S$, the natural map 
\[
K\otimes_{K_0}\D_{\rig}^{\dag}(V_X|_{M(S)})^{\varphi^f=\alpha,\Gamma=1}
\ra(\D_{\dif}^{+,fn}(V_X|_{M(S)})/(t^k))^{\Gamma}
\]
is an isomorphism. 
As a consequence, the presheaf $M(R)\mapsto \D_{\rig}^{\dag}(V_X|_{M(R)})^{\varphi^f=\alpha,\Gamma=1}$, where $M(R)$ runs through all affinoid subdomains of $M(S)$,
is indeed a coherent sheaf on $M(S)$.
\end{theorem}
We denote by $\mathscr{D}_{\rig}^{\dag}(V_S)^{\varphi^f=\alpha,\Gamma=1}$ this coherent sheaf, and by $\mathscr{D}_{\rig}^{\dag}(V_{X_{fs}})^{\varphi^f=\alpha,\Gamma=1}$ the coherent sheaf on $X_{fs}$ obtained by gluing the sheaves $\mathscr{D}_{\rig}^{\dag}(V_S)^{\varphi^f=\alpha,\Gamma=1}$ for all affinoid subdomains $M(S)$ of $X_{fs}$.

\subsection{Rank 1 $\m$-modules and trianguline representations} 
For a $K_0\otimes_{\Q}K_0$-module $M$, set 
\[
M_{\sigma}=M\otimes_{K_0\otimes_{\Q}K_0}(K_0\otimes_{K_0,\sigma}K_0)
\] 
for any $\sigma\in \mathrm{Gal}(K_0/\Q)$.  We may identify $M$ with $\oplus_{\sigma\in\mathrm{Gal}(K_0/\Q)}M_\sigma$.
Then for any $m\in M$, let $m_\sigma$ denote the $M_\sigma$-component of it.
Now let $S$ be an affinoid algebra over $K_0$, set $S_{K_0}=S\otimes_{\Q}K_0$.  Let $\phi\in \mathrm{Gal}(K_0/\Q)$ be the arithmetic Frobenius.  Using the canonical isomorphism 
\[
S_{K_0}\cong\prod_{0\leq i\leq f-1}S_{\phi^i},
\] 
for any $a\in S$ we equip $S_{K_0}$ with a $1\otimes\phi$-semilinear action $\varphi$ by setting
\[
\varphi(x_1,x_2,\dots, x_{f-1})=(ax_{f-1},x_1,\dots,x_{f-2}).
\]
Let $D_a$ denote this $\varphi$-module. The $\varphi$-action on $D_a$ satisfies $\varphi^f=a\otimes 1$.

Let $K'_0$ be the maximal unramified extension of $\Q$ contained in $K(\mu_{p^\infty})$. Recall that one may identify $\mathbf{B}_{\rig,K}^\dag$ with the Robba ring $\r_{K'_0}$ over $K_0'$ (see \cite[\S2.6]{LB02}). Equip $\r_{K_0'}$ with the induced actions of $\varphi$ and $\Gamma$. Let $\widehat{\mathscr{T}}(S)$ be the set of continuous characters $\delta:K^\times\ra S^\times$. For any $\delta\in\widehat{\mathscr{T}}(S)$, we attach to it a rank 1 $(\varphi,\Gamma)$-module $\r_S(\delta)$ over $\r_{K_0'}\widehat{\otimes}_{\Q}S$ as follows. If $\delta|_{\OO_K^\times}$ is trivial, we set $\r_S(\delta)=D_{\delta(\pi_K)}\otimes_{S_{K_0}}(\r_{K_0'}\widehat{\otimes}_{\Q}S)$; here $D_{\delta(\pi_K)}$ is equipped with the trivial $\Gamma$-action. For general $\delta$, we may write $\delta=\delta'\delta''$ so that $\delta'(\pi_K)=1$ and $\delta''|_{\OO_K^\times}$ is the trivial character. By local class field theory, $\delta'$ can be viewed as an $S^\times$-valued continuous character of the Weil group of $K$; it extends to a character of $G_K$ by continuity. We set $\r_S(\delta)=\D_\rig^\dagger(\delta')\otimes_{\r_{K'_0}\widehat{\otimes}_{\Q}S}\r_S(\delta'')$. For any $(\varphi,\Gamma)$-module $D_S$ over $\r_{K_0'}\widehat{\otimes}_{\Q}S$, set $D_S(\delta)=D_S\otimes_{\r_{K_0'}\widehat{\otimes}_{\Q}S}\r_S(\delta)$.

\begin{defn}\label{def:trianguline}
For $\delta\in\widehat{\mathscr{T}}(S)$, a rank 1 $\m$-module over $\r_{K_0'}\widehat{\otimes}_{\Q}S$ is called of \emph{type $\delta$} if it is isomorphic to $M\otimes_{S_{K_0}}\r_S(\delta)$ for some locally free rank 1 $S_{K_0}$-module $M$ equipped with trivial $\varphi$- and $\Gamma$-actions. We call a $\m$-module $D_S$ over $\r_{K_0'}\widehat{\otimes}_{\Q}S$ \emph{triangulable} if it admits a filtration
\[
0=\mathrm{Fil}_0(D_S)\subset \mathrm{Fil}_1(D_S)\subset\cdots\subset\mathrm{Fil}_{d-1}(D_S)
\subset \mathrm{Fil}_d(D_S)=D_S
\]
by $\m$-submodules over $\r_{K'_0}\widehat{\otimes}S$ such that each successive quotient $\mathrm{Fil}_i(D_S)/\mathrm{Fil}_{i-1}(D_S)$ is of type $\delta_i$ for some $\delta_i\in\widehat{\mathscr{T}}(S)$; any such a filtration $(\mathrm{Fil}_i(D_S))_{1\leq i\leq d}$ is called a \emph{triangulation} of $D_S$, and $(\delta_i)_{1\leq i\leq d}$ are called the \emph{parameters} of this triangulation. We call a locally free $S$-linear representation $V_S$ of $G_K$ \emph{trianguline} if the corresponding $\m$-module $\D_\rig^\dag(V_S)$ is triangulable.
\end{defn}

\subsection{Refined and weakly refined families}
From now on, let $X$ be reduced. For any $x\in X$, let $V_x$ denote the specialization of $V_X$ at $x$. The following definitions generalize Bellaiche-Chenevier's notions of refined and weakly refined families \cite[\S4.2.3]{BC06} to  $p$-adic representations of $G_K$\footnote{Note that the Hodge-Tate weight of $p$-adic cyclotomotic character is normalized to be $-1$ in \cite{BC06}; this is opposite to our normalization.}. As previously mentioned, we will define our refined and weakly refined families to be arithmetic families of $p$-adic representations, not just pseudocharacters as in Bellaiche-Chenevier's original definitions.

\begin{defn}\label{def:weak-refined-family}
A \emph{family of weakly refined $p$-adic representations} of $G_K$ of dimension $d$ over $X$ is a family of $p$-adic representations $V_X$ of $G_K$ of dimension $d$ over $X$ together with the following data
\begin{enumerate}
\item[(1)]$d$ analytic functions $\kappa_1,\dots,\kappa_d\in K\otimes_{\Q}\OO(X)$,
\item[(2)]an analytic function $F\in\OO(X)$,
\item[(3)]a Zariski-dense subset $Z$ of $X$,
\end{enumerate}
subject to the following requirements.
\begin{enumerate}
\item[(a)]For every $x\in X$, the generalized Hodge-Tate weights\footnote{We set the generalized Hodge-Tate weights to be the roots of the Sen polynomial.} of $V_x$ are, with multiplicity, $\kappa_1(x),\dots,\kappa_d(x)$.
\item[(b)]If $z\in Z$, $V_z$ is crystalline.
\item[(c)]If $z\in Z$, $\kappa_1(z)_\tau$ is the biggest Hodge-Tate weight of $D_{\mathrm{dR}}(V_z)_\tau$ for every $\tau\in\mathrm{H}_K$.
\item[(d)]For each $z\in Z$, $D_{\mathrm{crys}}(V_z)$ has a $\varphi$-submodule over $K_0\otimes_{\Q}k(x)$ which is isomorphic to $\displaystyle{D_{F(z)\prod_{\tau\in\H_K}\tau(\pi_K)^{-\kappa_1(z)_\tau}}}$.

\item[(e)]For any non-negative integer $C$, let $Z_C$ be the set
\[
\{z\in Z, \kappa_1(z)_\tau-\kappa_n(z)_\tau>C, \forall n\in\{2,\dots,d\}, \tau\in\H_K\}.
\]
Then $Z_C$ accumulates at any $z\in Z$ for all $C$.
\item[(f)]There exists a continuous character $\chi:\OO_K^\times\ra\OO(X)^\times$ whose derivative at $1$ is $-\kappa_1$ and whose evaluation at any $z\in Z$ is the character $x\mapsto\prod_{\tau\in\H_K}\tau(x)^{-\kappa_1(z)_\tau}$.
\end{enumerate}
\end{defn}

We may also view $\chi$ as a continuous character from $ K^\times$ to $\mathcal{O}(X)^\times$ by setting $\chi(\pi_K)=1$. By twisting $V_X$ with $\chi^{-1}$, we may suppose $\kappa_1=0$.  In this case, we have the following result. 

\begin{theorem}(Theorem \ref{thm:weak-refined-family})\label{thm:intro-weak-refined-family}
The finite slope subspace of $X$ with respect to the pair $(F, V_X)$ is $X$ itself.
\end{theorem}

\begin{defn}\label{def:refined-family}
A \emph{family of refined $p$-adic representations} of $G_K$ of dimension $d$ over $X$ is a family of $p$-adic representations $V_X$ of $G_K$ of dimension $d$ over $X$ together with the following data
\begin{enumerate}
\item[(1)]$d$ analytic functions $\kappa_1,\dots,\kappa_d\in K\otimes_{\Q}\OO(X)$,
\item[(2)]$d$ analytic functions $F_1,\dots,F_d\in\OO(X)$,
\item[(3)]a Zariski-dense subset $Z$ of $X$,
\end{enumerate}
subject to the following requirements.
\begin{enumerate}
\item[(a)] For every $x\in X$, the generalized Hodge-Tate weights of $V_x$ are, with multiplicity, $\kappa_1(x),\dots,\kappa_d(x)$.

\item[(b)]If $z\in Z$, $V_z$ is crystalline.
\item[(c)]If $z\in Z$, then $\kappa_1(z)_\tau>\kappa_2(z)_\tau>\cdots>\kappa_d(z)_\tau$ for any $\tau\in\H_K$.
\item[(d)]For each $z\in Z$, there exists a refinement of $V_z$ such that the associated ordering of the $\varphi^f$-eigenvalues are
\[
    (\prod_{\tau\in\H_K}\tau(\pi_K)^{-\kappa_1(z)_\tau}F_1(z),\dots,
    \prod_{\tau\in\H_K}\tau(\pi_K)^{-\kappa_d(z)_\tau}F_d(z)).
\]
\item[(e)]For any non-negative integer $C$, let $Z_C$ be the set
\[
\{z\in Z, |\kappa_I(z)_\tau-\kappa_J(z)_\tau|>C, \forall I,J\subseteq\{1,\dots,d\},|I|=|J|>0, I\neq J, \tau\in\H_K\},
\]
where $\kappa_I=\sum_{i\in I}\kappa_i$. Then $Z_C$ accumulates at any $z\in Z$ for all $C$.
\item[(f)]For each $1\leq i\leq d$, there exists a continuous character $\chi_i:\OO_K^\times\ra\OO(X)^\times$ whose derivative at $1$ is $-\kappa_i$ and whose evaluation at any $z\in Z$ is the character 
\[
x\mapsto\prod_{\tau\in\H_K}\tau(x)^{-\kappa_i(z)_\tau}.
\]
\end{enumerate}
\end{defn}

For each $1\leq i\leq d$, let $\alpha_i=\prod_{j=1}^iF_j$ and $\eta_i=\prod_{j=1}^i\chi_i$. Let $\delta_i$ be the continuous character $\delta_i: K^\times\ra \mathcal{O}(X)^\times$ defined by $\delta_i(\pi_K)=\alpha_i$ and
$\delta_i|_{\OO_K^\times}=\eta_i$. It is straightforward to see that $\wedge^iV_X$ is weakly refined with $F=\alpha_i$ and $\chi=\eta_i$.  Thus by Theorem \ref{thm:intro-weak-refined-family} and Theorem \ref{thm:intro-fs}, for each $1\leq i\leq d$, we get a coherent sheaf of crystalline periods $\mathscr{D}_\rig^\dag((\wedge^iV_X)(\eta_i^{-1}))^{\varphi^f=\alpha_i, \Gamma=1}$ on $X$. The main result of this paper is the following theorem (see Theorem \ref{thm:good-triangulation} for a more precise version).  Here for any $\OO(X)^\times$-valued character $\Psi$ and $x\in X$, we denote by $\Psi(x)$ the evaluation of $\Psi$ at $x$.

\begin{theorem}\label{thm:intro-good-triangulation}
The families of $\m$-modules 
\[
\mathscr{D}_\rig^\dag((\wedge^iV_X)(\eta_i^{-1}))^{\varphi^f=\alpha_i, \Gamma=1}\otimes_{K_0\otimes\OO_X}\mathscr{D}_\rig^\dag(\eta_i)
\] 
for all $1\leq i\leq d$ give rise to a global triangulation of $V_X$ on a Zariski open and dense subspace of $X$ with parameters $(\delta_i/\delta_{i-1})_{1\leq i\leq d}$. Furthermore, the locus of global triangulation contains all $x\in X$ such that $\D_\rig^\dag(V_x)$ admits a triangulation with parameters
$((\delta_{i}/\delta_{i-1})(x))_{1\leq i\leq d}$
and satisfies 
\[
\dim_{\Q}\D_\rig^\dag((\wedge^iV^{\mathrm{ss}}_x)(\eta_i^{-1}(x)))_\sigma^{\varphi^f=\alpha_i(x),\Gamma=1}=1
\]
for every $1\leq i\leq d-1$ and $\sigma\in\mathrm{Gal}(K_0/\Q)$. In particular, the locus of global triangulation contains all regular non-critical points. 
\end{theorem}

\subsection{Triangulation loci of refined families}
\begin{defn}\label{def:intro-good}
Let $V_X$ be a weakly refined family. For $x\in X$, we say $x$ is \emph{saturated} for the family $V_X$ if the following two conditions hold. 
\begin{enumerate}
\item[(1)]For any $\sigma\in\mathrm{Gal}(K_0/\Q)$, the coherent $\OO_X$-module $\mathscr{D}^{\dag}_{\rig}(V_X(\chi^{-1}))_\sigma^{\varphi^f=F,\Gamma=1}$ is locally free of rank 1 around $x$.
\item[(2)]The image of $\mathscr{D}^{\dag}_{\rig}(V_X(\chi^{-1}))^{\varphi^f=F,\Gamma=1}$ in $\D_\rig^\dag(V_x(\chi^{-1}(x)))$ generates a rank 1 saturated $\m$-submodule.
\end{enumerate}
We denote by $X_s$ the set of saturated points, and call it the \emph{saturated locus} of $V_X$. 
\end{defn}

Using Theorem \ref{thm:intro-fs-sheaf}, it is not difficult to show that $X_s$ is a Zariski open and dense subspace of $X$. 
For a refined family $V_X$,  the saturated locus $X_s$ is defined to the intersections of the saturated loci of the weakly refined families $\wedge^iV_X$ for all $1\leq i\leq d$. It follows that $X_s$ is a Zariski open and dense subspace of $X$. 

\begin{defn}
Let $V_X$ be a refined family. The \emph{triangulation locus} of  $V_X$ is defined to be the set of $x\in X_s$ such that  the $\m$-modules
\[
\mathscr{D}_\rig^\dag((\wedge^iV_X)(\eta_i^{-1}))^{\varphi^f=\alpha_i, \Gamma=1}\otimes_{K_0\otimes_{\Q}\OO_X}\D_\rig^\dag(\eta_i(x))
\] 
for all $1\leq i\leq d$ give rise to a triangulation of $\D_\rig^\dag(V_x)$. That is, there exists a triangulation $(\mathrm{Fil}_i(\D_\rig^\dag(V_x)))_{1\leq i\leq d}$ of $\D_\rig^\dag(V_x)$ such that
\[
\mathscr{D}_\rig^\dag((\wedge^iV_X)(\eta_i^{-1}))^{\varphi^f=\alpha_i,\Gamma=1}\otimes_{K_0\otimes_{\Q}\OO_X}\D_\rig^\dag(\eta_i(x))=\wedge^i(\mathrm{Fil}_i(\D_\rig^\dag(V_x)))
\]
for all $1\leq i\leq d$.
\end{defn}

It is obvious that the locus of global triangulation is contained in the triangulation locus. In fact, it turns out that they actually coincide. That is, a refined family admits a global triangulation on the triangulation locus. 

\begin{prop}\label{thm:intro-chain-Zariski}(Proposition \ref{prop:chain-Zariski})
The triangulation locus forms a reduced Zariski closed subspace of $X_s$. Furthermore, for any affinoid subdomain $M(S)$ of the triangulation locus, the sequence of $\m$-modules 
\[
(\D_\rig^\dag((\wedge^iV_X|_{M(S)})(\eta_i^{-1}))^{\varphi^f=\alpha_i, \Gamma=1}\otimes_{K_0\otimes_{\Q}S}
\D^\dag_{\rig}(\eta_i))_{1\leq i\leq d}
\] 
gives rise to  a triangulation of $\D_\rig^\dag(V_X|_{M(S)})$ with parameters $(\delta_i/\delta_{i-1})_{1\leq i\leq d}$.
\end{prop}

To cut out the triangulation locus, we view the associated family of $\m$-modules as a family of vector bundles over the relative half-open annulus $X\times \{0<v_p(T)\leq r\}$ for some $r>0$. For a general base $X$, it is difficult to deal with vector bundles over such a relative annulus.  We get over this difficulty by restricting the family of vector bundles on a closed annulus $v_p(T)\in[r/p^f,r]$, where $r$ is sufficiently small, over $X$.  We then cut out the triangulation locus and construct the global triangulation over this closed annulus. Finally, we use the Frobenius action to extend the global triangulation over the closed annulus to a global triangulation of the original family of $\m$-modules. 

Therefore,  to prove the global triangulation conjecture, it finally reduces to show that all regular non-critical points belong to the triangulation locus. As mentioned before, we prove this using a result of Bellaiche-Chenevier on descent \cite[\S3.2]{BC06}.  On the other hand, although the global triangulation can not be extended to the whole base (because of the existence of critical points as pointed out by Bellaiche-Chenevier \cite[Remark 2.5.9]{BC06}),  it turns out that the specializations of refined families are all trianguline. 
\begin{theorem}\label{thm:intro-trianguline} (Theorem \ref{thm:trianguline})
For any $x\in X$, $V_x$ is trianguline.
\end{theorem}

\subsection{Application to the eigencurve}

Let $\mathcal{C}$ be the eigencurve associated with an absolutely irreducible $2$-dimensional residual representation of $\mathrm{Gal}(\overline{\mathbb{Q}}/\mathbb{Q})$ which is  $p$-modular in the sense of \cite{CM}. Let $\alpha\in\mathcal{O}(\mathcal{C})^\times$ be the function of $U_p$-eigenvalues, and let 
$\kappa: \mathcal{C}\ra \mathcal{W}$
be the map to the weight space. We normalize $\kappa$ in such a way that if $x$ is a classical eigenform of weight $k$, then $\kappa(x)=k-1$. Let 
\[
\epsilon: (\mathbb{Z}/N\ZZ)^\times\times\Zp^\times\ra\OO(\mathcal{C})^\times
\] be the nebentypus-weight character (cf. \cite[\S3.1]{CM}). That is, the diamond operators act on overconvergent eigenforms parametrized by $\mathcal{C}$ through $\epsilon$. 

Following \cite{Ki03}, let $V_{\mathcal{C}}$ be the \emph{dual} of the family of $p$-adic representations of $G_{\mathbb{Q}}$ on $\mathcal{C}$ interpolating the Galois representations attached to classical eigenforms. That is, for any $x\in\mathcal{C}$ and prime $l$ not dividing $pN$, the characteristic polynomial of the geometric Frobenius at $l$ on $V_x$ is 
\[
X^2-a_l(x)X+\epsilon(x)=0,
\]
where $a_l$  denotes the $l$-th coefficient of the $q$-expansion. Let $Z$ be the set of classical points $z\in\mathcal{C}$ such that $V_z$ is crystalline with distinct crystalline Frobenius eigenvalues. Coleman's classicality theorem then ensures that $V_{\mathcal{C}}$ is a weakly refined family together with $\kappa_1=0, \kappa_2=-\kappa$, $F=\alpha$ and $Z$. The following theorem completely determines the local behavior of $V_{\mathcal{C}}$.

\begin{theorem} (Theorem \ref{thm:eigencurve})\label{thm:intro-eigencurve}
For any $x\in\mathcal{C}$,  the coherent sheaf $\mathscr{D}_\rig^\dag(V_{\mathcal{C}})^{\varphi=F,\Gamma=1}$ is locally free of rank 1 around $x$ unless $\kappa(x)=0$ and 
$\dim D_{\mathrm{crys}}(V^{\mathrm{ss}}_x)^{\varphi=F(x)}=2$. In particular, $V^{\mathrm{ss}}_x$ is crystalline in this case. 
If $x$ is not of this form, it is not saturated if and only if it satisfies one of the following two disjoint conditions.
\begin{enumerate}
\item[(1)]The weight $\kappa(x)$ is a positive integer and $v_p(F(x))>\kappa(x)$. As a consequence, $V_x$ belongs to $\mathscr{S}^{\mathrm{ng}}_*\cap\mathscr{S}_*^{\mathrm{HT}}$ in the sense of \cite{C12}; hence $V_x$ is irreducible, Hodge-Tate and non-de Rham. Furthermore, the image of $t^{-\kappa(x)}(\mathscr{D}_\rig^\dag(V_{\mathcal{C}}))^{\varphi=F,\Gamma=1}$ generates a rank 1 saturated $\m$-submodule in $\D_\rig^\dagger(V_x)$.
\item[(2)]The weight $\kappa(x)$ is a positive integer and $v_p(F(x))=\kappa(x)$, and $V_x$ has a rank 1 subrepresentation $V_x'$ which is crystalline with Hodge-Tate weight $-\kappa(x)$. Furthermore, in this case, the image of $\mathscr{D}_\rig^\dag(V_{\mathcal{C}})^{\varphi=F,\Gamma=1}$ in $\D_\rig^\dag(V_x)$ is $k(x)\cdot t^{\kappa(x)}e'$ where $e'$ is a canonical basis of $\D_\rig^\dag(V_x')$.
\end{enumerate}
In case (2), if $x\in Z$, then it is critical. Hence it is decomposable. Suppose $V_x=V_1\oplus V_2$ where $V_1$ has Hodge-Tate weight $0$ and $V_2$ has Hodge-Tate weight $-\kappa(x)$. Then the image of $\mathscr{D}_\rig^\dag(V_{\mathcal{C}})^{\varphi=F,\Gamma=1}$ in $\D_\rig^\dag(V_x)$ is $k(x)\cdot t^{\kappa(x)}e_2$ where $e_2$ is a canonical basis of $\D_\rig^\dag(V_2)$.

\end{theorem}


\section*{Acknowledgements}
The author would like to thank Joel Bellaiche, Rebeca Bellovin, Laurent Berger,  Ga\"{e}tan Chenevier, Pierre Colmez, Matthew Emerton, Kiran S. Kedlaya, Haruzo Hida, Mark Kisin, Benjamin Schraen, Fucheng Tan and Liang Xiao for helpful discussions and communications. The author would also like to express his gratitudes to Brain Conrad and Shanwen Wang for their very careful reading on earlier drafts of this paper.\\

\section*{Notation and conventions}
Let $v_p$ denote the $p$-adic valuation on $\mathbb{C}_p$ normalized as $v_p(p)=1$. Let $|\cdot|$ be the corresponding norm defined by $|x|=p^{-v_p(x)}$.  Fix a finite extension $K$ of $\Q$ in $\mathbb{C}_p$. Let $\OO_K$ be the ring of integers of $K$, and let $\pi_K$ be a fixed uniformizer. Let $v_K$ denote the $p$-adic valuation on $\mathbb{C}_p$ normalized as $v_K(\pi_K)=1$. For any valuation $v$ (norm $|\cdot|$) and a matrix $A = (A_{ij})$, we use $v(A)$ (resp. $|A|$) to denote the minimal valuation (resp. maximal norm) among the entries.

We may view any continuous character of $\OO_K^\times$ as a continuous character of $K^\times$ by pulling back via the projection $K^\times\ra \OO_K^\times$ determined by $\pi_K$. We may further view it as a continuous character of $W_K$, which denotes the Weil group of $K$, via the local reciprocity isomorphism
$W_K^{\mathrm{ab}}\cong K^\times$ where a geometric Frobenius element maps to $\pi_K$.

We choose a compatible sequence of primitive $p$-powers roots of unity $(\varepsilon_n)_{n\geq0}$, i.e. each $\varepsilon_n\in\overline{\mathbb{Q}}_p$ is a primitive $p^n$-th root of $1$, and they satisfy $\varepsilon^p_{n+1}=\varepsilon_{n}$ for all $n\geq0$. Fix $\varepsilon=(\varepsilon_0,\varepsilon_1,\dots)$ to be Fontaine's $p$-adic $\exp(2\pi i)$.  For a finite extension $L$ of $\Q$ in $\mathbb{C}_p$, let $L_n=L(\varepsilon_n)$ for $n\geq 1$, and let $L_\infty=\cup_{n\in\mathbb{N}}L_n$. Let $L_0'$ be the maximal unramified extension of $\Q$ in $L_\infty$. Let $H_L=\mathrm{Gal}(\overline{\mathbb{Q}}_p/L_\infty)$, and let $\Gamma_L=\mathrm{Gal}(L_\infty/L)$. Denote $\Gamma_K$ by $\Gamma$ for simplicity. 

We normalize the Hodge-Tate weight in a way that the $p$-adic cyclotomic character has Hodge-Tate weight 1.

 Let $K_0$ be the maximal unramified extension of $K$ in $\mathbb{C}_p$, and let $f=[K_0:\Q]$. 
For any $\sigma\in \mathrm{Gal}(K_0/\Q)$, let $\H_\sigma$ be the set of $\tau\in \H_K$ such that its restriction on $K_0$ is $\sigma$. 

For $r>0$, put $\rho(r)=\frac{p-1}{pr}$. For $n\geq0$, let $r_n=p^{n-1}(p-1)$. For $s>0$, let $n(s)$ be the maximal integer $n$ such that $r_n\leq s$. 

For an affinoid algebra $S$, we denote by $\mathcal{O}_S$ the unit ball of $S$. For a topological group $G$ and a rigid analytic space $X$ over $\Q$, by a \emph{family of $p$-adic representations of $G$} of dimension $d$ on $X$ we mean a locally free coherent $\mathcal{O}_X$-module $V_X$ of rank $d$ equipped with a continuous $\mathcal{O}_X$-linear $G$-action. When $X=M(S)$ is an affinoid space over $\Qp$, we also call a family of $p$-adic representations of $G$ on $X$ an \textit{$S$-linear $G$-representation}. If $M(R)\subset M(S)$ is an affinoid subdomain and $V_S$ is a family of representation on $M(S)$, write $V_R$ for the base change of $V_S$ from $S$ to $R$. Finally, for every $x\in M(S)$, we write $V_x$ to denote the specialization $V_S\otimes_S k(x)$.

\section{Preliminaries}

\subsection{The $\m$-module functor}
Let $S$ be an affinoid algebra over $\Q$, and let $V_S$ be a finite locally free $S$-linear representation of $G_K$.  The $(\varphi,\Gamma)$-module functors $\D^\dag_K(V_S)$ and $\D_{\rig,K}^\dag(V_S)$ are constructed in \cite{BC07} and \cite{KL}.  However, both of these works do not really verify that $\D^\dag_K(V_S)$ and $\D_{\rig,K}^\dag(V_S)$ are $\varphi$-modules in the sense that they are isomorphic to their $\varphi$-pullbacks respectively. This small gap will be filled in this subsection.  
We follow the notations of \cite{BC07} and \cite{KL}, and refer the reader to them for more details. Recall that Berger-Colmez show that the ring $\widetilde{\mathbf{A}}^{(0,1]}\widehat{\otimes}_{\Zp}\OO_S$ together with the cyclotomic character $\chi:G_K\ra\Zp^{\times}$ satisfy the Tate-Sen axioms (see \cite[APPENDIX D]{Bel13} for a detailed exposition about Tate-Sen axioms) for any $c_1>0, c_2>0$ and $c_3>\frac{1}{p-1}$ \cite[Proposition 4.2.1, Proposition 3.1.4]{BC07}.
For any finite extension $L$ of $\Qp$, $n\in\mathbb{Z}$ and $s>0$, let $\mathbf{A}^{\dag,s}_{L,n}$ denote the subring $\varphi^{-n}(\mathbf{A}^{\dag,p^ns}_L)$ of $\widetilde{\mathbf{A}}^{\dag,s}$ (see \cite[\S1.3]{LB02} for the definitions of $\mathbf{A}_L^{\dag,s}$ and $\widetilde{\mathbf{A}}^{\dag,s}$). The following result then follows from \cite[Proposition 3.3.1]{BC07}.

\begin{prop}
\label{prop:Tate-Sen}
Let $T_S$ be a free $\OO_S$-linear representation of $G_K$ of rank $d$. Let $L$ be a finite Galois extension of $K$ so that $G_L$ acts trivially on $T_S/p^kT_S$, where $k$ is an integer such that 
$\mathrm{val}^{(0,1]}(p^k)>c_1+2c_2+2c_3$. 
Then there exists an integer $n(L) \geq 0$ such that
for any $n\geq n(L)$,
$T_S\otimes_{\OO_S}(\widetilde{\mathbf{A}}^{\dag,\frac{p-1}{p}}\widehat{\otimes}_{\Zp}\OO_S)$ has a
 unique sub-$\mathbf{A}^{\dag,\frac{p-1}{p}}_{L,n}\widehat{\otimes}_{\Zp}\OO_S $-module
$\D^{\dagger,\frac{p-1}{p}}_{L,n}(T_S)$ which is free of rank $d$, fixed by $H_L$, stable under $G_K$,
and has a basis which is $c_3$-fixed by $\Gamma_L$
(that is,  for each $\gamma \in \Gamma_L$, the matrix $W_\gamma$ of $\gamma$ with respect to this basis satisfies $\mathrm{val}^{(0,1]}(W_\gamma-1)>c_3$), and satisfies
\begin{equation}
\label{E:berger-colmez}
\D^{\dagger,\frac{p-1}{p}}_{L,n}(T_S)\otimes_{\aa^{\dag,s}_{L,n}\widehat{\otimes}_{\Zp}\OO_S}
(\widetilde{\mathbf{A}}^{\dag,\frac{p-1}{p}}\widehat{\otimes}_{\Zp}\OO_S)=
T_S\otimes_{\OO_S}(\widetilde{\mathbf{A}}^{\dag,\frac{p-1}{p}}\widehat{\otimes}_{\Zp}\OO_S).
\end{equation}
\end{prop}


\begin{cor}\label{cor:Tate-Sen}
Keep notations as above.
\begin{enumerate}
\item[(1)]We have $\D^{\dagger,\frac{p-1}p}_{L,n+1}(T_S)=\D^{\dagger,s}_{L,n}(T_S)\otimes_{\mathbf{A}^{\dag,\frac{p-1}p}_{L,n}\widehat{\otimes}_{\Z}\OO_S}(\mathbf{A}^{\dag,\frac{p-1}p}_{L,n+1}\widehat{\otimes}_{\Z}\OO_S).$
\item[(2)]By enlarging $L$ and $n(L)$, we may have 
\[
\D^{\dagger,\frac{p-1}p}_{L,n+1}(T_S)=\varphi^{-1}(\D^{\dag,\frac{p-1}p}_{L,n}(T_S)\otimes_{\mathbf{A}^{\dag,\frac{p-1}p}_{L,n}\widehat{\otimes}_{\Z}\OO_S}(\mathbf{A}^{\dag,p-1}_{L,n}
\widehat{\otimes}_{\Z}\OO_S)).
\]
\end{enumerate}
\end{cor}
\begin{proof}
Let $D_1$ and $D_2$ denote the right hand sides of (1) and (2) respectively. By (\ref{E:berger-colmez}), we first see that 
\[
D_1\otimes_{\mathbf{A}^{\dag,\frac{p-1}p}_{L,n+1}\widehat{\otimes}_{\Z}\OO_S}(\widetilde{\mathbf{A}}^{\dag,\frac{p-1}p}
\widehat{\otimes}_{\Z}\OO_S)=
T_S\otimes_{\OO_S}(\widetilde{\mathbf{A}}^{\dag,\frac{p-1}p}\widehat{\otimes}_{\Zp}\OO_S).
\]
Moreover, it is straightforward to see that $D_1$ satisfies all the properties of $\D^{\dagger,\frac{p-1}{p}}_{L,n+1}(T_S)$ given by Proposition \ref{prop:Tate-Sen}. We therefore conclude
$D_1=\D^{\dagger,\frac{p-1}{p}}_{L,n+1}(T_S)$ by the uniqueness of $\D^{\dagger,s}_{L,n+1}(T_S)$. This proves (1).

To prove (2), we enlarge $L$ and $n(L)$ so that Proposition \ref{prop:Tate-Sen} holds both for $(c_1, c_2, c_3)$ and for $(c_1, c_2, c_3'=pc_3)$. We first see that
\[
\varphi(D_2)\otimes_{\mathbf{A}^{\dag,p-1}_{L,n}\widehat{\otimes}_{\Z}\OO_S}(\widetilde{\mathbf{A}}^{\dag,p-1}
\widehat{\otimes}_{\Z}\OO_S)=
T_S\otimes_{\OO_S}(\widetilde{\mathbf{A}}^{\dag,p-1}\widehat{\otimes}_{\Zp}\OO_S).
\]
It follows that
\[
D_2\otimes_{\mathbf{A}^{\dag,\frac{p-1}p}_{L,n+1}\widehat{\otimes}_{\Z}\OO_S}(\widetilde{\mathbf{A}}^{\dag,\frac{p-1}p}
\widehat{\otimes}_{\Z}\OO_S)=
T_S\otimes_{\OO_S}(\widetilde{\mathbf{A}}^{\dag,\frac{p-1}p}\widehat{\otimes}_{\Zp}\OO_S).
\]
Let $e$ be a $c_3'$-fixed basis of $\D^{\dag,\frac{p-1}p}_{L,n}(T_S)$. It is then straightforward to see that the basis $\varphi^{-1}(e)$ of $D_2$ is at least $c_3$-fixed. We therefore conclude that $D_2=\D^{\dag,\frac{p-1}p}_{L,n+1}(T_S)$ by the uniqueness of $\D^{\dag,\frac{p-1}p}_{L,n+1}(T_S)$. 
\end{proof}

In the rest of the paper, following the convention of \cite{BC07}, we fix some constants $c_1>0, c_2>0$ and $c_3>\frac{p-1}{p}$ such that $c_1+2c_2+2c_3<v_p(12p)$. Now let $V_S$ be a free $S$-linear $G_K$-representation of rank $d$. Choose a free $\OO_S$-lattice $T_S$ in $V_S$.
Since the $G_K$-action is continuous, there exists a finite Galois extension
$L$ of $K$ such that $G_L$ carries $T_S$ into itself; hence $T_S$ is $G_L$-stable. We may enlarge $L$ so that $G_L$ acts trivially on $T_S/12pT_S$. We also assume that Corollary \ref{cor:Tate-Sen}(2) holds by further enlarging $L$ and $n(L)$. 

For any $g\in G_K$, it follows that $gT_S$ is also a $G_L$-stable $\OO_S$-lattice of $V_S$. Moreover, $G_L$ acts trivially on $gT_S/12p(gT_S)$ as well. By the uniqueness of $\D^{\dagger, p-1/p}_{L,n}(gT_S)$, we get 
\[
\D^{\dagger, p-1/p}_{L,n}(gT_S)=g\D^{\dagger, p-1/p}_{L,n}(T_S).
\]
Using the fact that $T_S$ and $gT_S$  are commensurable, we therefore deduce that  $\D^{\dagger, p-1/p}_{L,n}(T_S)$ and $g\D^{\dagger, p-1/p}_{L,n}(T_S)$ are commensurable.
 This implies that the sub-$S$-module $\D^{\dagger, p-1/p}_{L,n}(T_S)\otimes_{\OO_S}S$ of
$V_S\otimes_{\OO_S}(\widetilde{\mathbf{A}}^{\dag,r_n}\widehat{\otimes}_{\Z}\OO_S)$ is independent of the choice of $T_S$ and $G_K$-stable for any $n\geq n(L)$.
For $s\geq r_{n(L)}$, we set
\begin{center}
$\D^{\dagger,s}_K(V_S)=(\varphi^{n(L)} (\D^{\dagger, p-1/p}_{L,n(L)}(T_S))\otimes_{\mathbf{A}^{\dag,r_{n(L)}}_L\widehat{\otimes}_{\Zp}\OO_S}\mathbf{B}^{\dag,s}_L\widehat{\otimes}_{\Qp} S)^{H_K}$
\end{center}
which is equipped with a $\Gamma_K$-action. By \cite[Proposition 2.2.1]{BC07} and \cite[Lemme 4.2.5]{BC07}, there exists an $s(L/K)>0$ such that if $s\geq s(L/K)$, then $\D^{\dagger,s}_K(V_S)$ is a locally free $\mathbf{B}^{\dag,s}_K\widehat{\otimes}_{\Q}S$-module of rank $d$. Let $n(V_S)=\max\{n(L),n(s(L/K))\}$, and put $s(V_S)=r_{n(V_S)}$.

\begin{remark}
\label{R:tate-sen}
By Corollary \ref{cor:Tate-Sen}(2), we see that for any integers $n_1,n_2$ such that $n(L)\leq n_1,n_2\leq n(s)$, 
\begin{equation}
\label{E:r-tate-sen}
\varphi^{n_1} (\D^{\dagger,p-1/p}_{L,n_1}(T_S))\otimes_{\mathbf{A}^{\dag, r_{n_1}}_L
\widehat{\otimes}_{\Zp}\OO_S}\mathbf{B}^{\dag,s}_L\widehat{\otimes}_{\Qp} S
=\varphi^{n_2} (\D^{\dagger,p-1/p}_{L,n_2}(T_S))\otimes_{\mathbf{A}^{\dag, r_{n_2}}_L
\widehat{\otimes}_{\Zp}\OO_S}\mathbf{B}^{\dag,s}_L\widehat{\otimes}_{\Qp} S.
\end{equation}
Thus one can replace $n(L)$ with any integer $n$ such that $n(L)\leq n\leq n(s)$ in the construction of $\D_K^{\dag,s}(V_S)$.
\end{remark}

If $S\ra R$ is a map of affinoid algebras over $\Q$, we set $V_{R}=V_S\otimes_SR$.  The following theorem slightly refines \cite[Th\'eor\`eme 4.2.9]{BC07} in the case of affinoid algebras.
\begin{theorem}\label{thm:BC-generalization}
For any $s\geq s(V_S)$, the locally free $\mathbf{B}_K^{\dagger,s}\widehat{\otimes}_{\Qp}S$-module
$\D^{\dagger,s}_K(V_S)$ is well-defined, i.e. its construction is independent of the choices of $T_S$ and $L$. Furthermore, it satisfies the following properties.
\begin{enumerate}
\item[(1)]
The natural map $\D^{\dagger,s}_K(V_S)\otimes_{\mathbf{B}_K^{\dagger,s}\widehat{\otimes}_{\Qp}S} \widetilde{\mathbf{B}}_K^{\dagger,s}\widehat{\otimes}_{\Qp}S \ra V_S\widehat{\otimes}_{\Q}\widetilde{\mathbf{B}}_K^{\dagger,s}$ is an isomorphism.
\item[(2)]
The construction is compatible with base change in $S$. 
\item[(3)]
The construction is compatible with passage
from $K$ to a finite extension $L$, i.e.
$\D_L^{\dagger,s}(V_S)=\D_K^{\dagger,s}(V_S)\otimes_{\mathbf{B}_K^{\dagger,s}\widehat{\otimes}_{\Qp}S}
\mathbf{B}_L^{\dagger,s}\widehat{\otimes}_{\Qp}S$.
\item[(4)]
For any $s'\geq s$, $\D^{\dag,s'}_K(V_S)=\D^{\dag,s}_K(V_S)\otimes_{\mathbf{B}_K^{\dag,s}\widehat{\otimes}_{\Q}S}
\mathbf{B}_K^{\dag,s'}\widehat{\otimes}_{\Q}S$.
\end{enumerate}
\end{theorem}
\begin{proof}
The statements $(1)$ and $(3)$ are already proved in \cite[Theorem 3.11]{KL} (which in turn is an easy consequence of \cite[Th\'eor\`eme 4.2.9]{BC07}). The assertion $(2)$ follows easily from the construction. For (4), let $T_S$ and $L$ be as above. It follows that
\[
\D^{\dagger,s}_L(V_S)=\varphi^{n(L)} (\D^{\dagger, p-1/p}_{L,n(L)}(T_S))\otimes_{\mathbf{A}^{\dag,r_{n(L)}}_L\widehat{\otimes}_{\Zp}\OO_S}\mathbf{B}^{\dag,s}_L\widehat{\otimes}_{\Qp} S
\]
for any $s\geq s(V_S)$. This implies 
\[
\D^{\dag,s'}_L(V_S)=\D^{\dag,s}_L(V_S)\otimes_{\mathbf{B}_L^{\dag,s}\widehat{\otimes}_{\Q}S}
\mathbf{B}_L^{\dag,s'}\widehat{\otimes}_{\Q}S.
\]
By (3), we get
\[
\D^{\dag,s'}_K(V_S)\otimes_{\mathbf{B}^{\dag,s'}_K\widehat{\otimes}_{\Q}S}\mathbf{B}^{\dag,s'}_L\widehat{\otimes}_{\Q}S
=\D^{\dag,s}_K(V_S)\otimes_{\mathbf{B}^{\dag,s}_K\widehat{\otimes}_{\Q}S}
\mathbf{B}_L^{\dag,s'}\widehat{\otimes}_{\Q}S.
\]
We conclude by taking the $H_K$-invariants on both sides.
\end{proof}

From now on, we assume $s\geq s(V_S)$ unless specified otherwise.

\begin{prop}\label{prop:dagger-phi-compare}
We have $\varphi(\D^{\dagger,s}_K(V_S))\subset\D^{\dag,ps}_K(V_S)$ and the natural map
\[
\varphi(\D^{\dagger,s}_K(V_S))\otimes_{\mathbf{B}^{\dag,s}_K\widehat{\otimes}_{\Q}S,\varphi}
\mathbf{B}^{\dag,ps}_K\widehat{\otimes}_{\Q}S\ra\D^{\dag,ps}_K(V_S)
\]
is an isomorphism.
\end{prop}
\begin{proof}
Let $T_S$ be a free $\OO_S$-lattice of $S$, and let $L$ be a finite Galois extension of $K$ such that $T_S$ is $G_L$-stable, and $G_L$ acts trivially on $T_S/12pT_S$. By Corollary (\ref{cor:Tate-Sen})(1), we get
\begin{equation*}
\begin{split}
&\varphi(\D^{\dagger,s}_L(V_S))\otimes_{\mathbf{B}^{\dag,s}_L\widehat{\otimes}_{\Q}S,\varphi}
(\mathbf{B}^{\dag,ps}_L\widehat{\otimes}_{\Q}S)\\
=&\varphi(\varphi^{n(L)}(\D^{\dagger,\frac{p-1}{p}}_{L,n(L)}(T_S))\otimes_{\mathbf{A}^{\dag,r_{n(L)}}_L\widehat{\otimes}_{\Z}\OO_S}(\mathbf{B}^{\dag,s}_L\widehat{\otimes}_{\Q}S))\otimes_{\mathbf{B}^{\dag,s}_L\widehat{\otimes}_{\Q}S,\varphi}
\mathbf{B}^{\dag,ps}_L\widehat{\otimes}_{\Q}S\\
=&\varphi^{n(L)+1}(\D^{\dagger,\frac{p-1}{p}}_{L,n(L)}(T_S))\otimes_{\mathbf{A}^{\dag,r_{n(L)}}_L\widehat{\otimes}_{\Z}\OO_S,\varphi}\mathbf{B}^{\dag,ps}_L\widehat{\otimes}_{\Q}S\\
=&\varphi^{n(L)+1}(\D^{\dagger,\frac{p-1}{p}}_{L,n(L)}(T_S))\otimes_{\mathbf{A}^{\dag,r_{n(L)}}_L\widehat{\otimes}_{\Z}\OO_S,\varphi}(\mathbf{A}^{\dag,r_{n(L)+1}}_L\widehat{\otimes}_{\Z}\OO_S)\otimes_{\mathbf{A}^{\dag,r_{n(L)+1}}_L\widehat{\otimes}_{\Z}\OO_S}\mathbf{B}^{\dag,ps}_L\widehat{\otimes}_{\Q}S\\
=&\varphi^{n(L)+1}(\D_{L,n(L)}^{\dag,\frac{p-1}p}(T_S)\otimes_{\mathbf{A}^{\dag,r_{n(L)}}_{L,n(L)}\widehat{\otimes}_{\Z}\OO_S}\mathbf{A}^{\dag,r_{n(L)+1}}_{L,n(L)+1}\widehat{\otimes}_{\Z}\OO_S)\otimes_{\mathbf{A}^{\dag,r_{n(L)+1}}_L\widehat{\otimes}_{\Z}\OO_S}\mathbf{B}^{\dag,ps}_L\widehat{\otimes}_{\Q}S\\
=&\varphi^{n(L)+1}(\D_{L,n(L)+1}^{\dag,\frac{p-1}p}(T_S))\otimes_{\mathbf{A}^{\dag,r_{n(L)+1}}_L\widehat{\otimes}_{\Z}\OO_S}\mathbf{B}^{\dag,ps}_L\widehat{\otimes}_{\Q}S\\
=&\D^{\dag,ps}_L(V_S).
\end{split}
\end{equation*}
The last step follows from Remark \ref{R:tate-sen}.
By Theorem \ref{thm:BC-generalization}(3), we may rewrite the above equality as
\[
\varphi(\D^{\dagger,s}_K(V_S))\otimes_{\mathbf{B}^{\dag,s}_K\widehat{\otimes}_{\Q}S,\varphi}
\mathbf{B}^{\dag,ps}_L\widehat{\otimes}_{\Q}S=\D^{\dag,ps}_K(V_S)\otimes_{\mathbf{B}^{\dag,ps}_K\widehat{\otimes}_{\Q}S}
\mathbf{B}^{\dag,ps}_L\widehat{\otimes}_{\Q}S.
\]
We conclude by taking $H_K$-invariants on both sides.
\end{proof}

 We set
$\D_{\rig,K}^{\dagger,s}(V_S)= \D_K^{\dagger,s}(V_S)\otimes_{\mathbf{B}_K^{\dagger,s}\widehat{\otimes}_{\Q}S}
(\mathbf{B}_{\rig,K}^{\dagger,s}\widehat{\otimes}_{\Q}S)$.
We put
\[
\mathbf{B}_K^{\dagger}\widehat{\otimes}_{\Qp}S=\cup_{s>0}\mathbf{B}_K^{\dagger,s}\widehat{\otimes}_{\Qp}S,\quad
\widetilde{\mathbf{B}}_K^{\dagger}\widehat{\otimes}_{\Qp}S=\cup_{s>0}\widetilde{\mathbf{B}}_K^{\dagger,s}\widehat{\otimes}_{\Qp}S
\]
and
\[
\mathbf{B}_{\rig,K}^{\dagger}\widehat{\otimes}_{\Qp}S=\cup_{s>0}\mathbf{B}_{\rig,K}^{\dagger,s}\widehat{\otimes}_{\Qp}S, \quad
\widetilde{\mathbf{B}}_{\rig,K}^{\dagger}\widehat{\otimes}_{\Qp}S=\cup_{s>0}\widetilde{\mathbf{B}}_{\rig,K}^{\dagger,s}\widehat{\otimes}_{\Qp}S.
\]
We then set
\[
\D_K^{\dagger}(V_S)=\D_K^{\dagger,s}(V_S)\otimes_{\mathbf{B}_K^{\dagger,s}\widehat{\otimes}_{\Qp}S}\mathbf{B}_K^{\dagger}\widehat{\otimes}_{\Qp}S
=\cup_{s\geq s(V_S)}\D_K^{\dagger,s}(V_S)
\]
and
\[
\D^{\dagger}_{\rig,K}(V_S)=\D_{\rig,K}^{\dagger,s}(V_S)\otimes_{\mathbf{B}_{\rig,K}^{\dagger,s}\widehat{\otimes}_{\Qp}S}
\mathbf{B}_{\rig,K}^{\dagger}\widehat{\otimes}_{\Qp}S=\cup_{s\geq s(V_S)}\D_{\rig,K}^{\dag,s}(V_S).
\]
By Proposition \ref{prop:dagger-phi-compare}, we see that $\D^\dag_K(V_S)$  and $\D_{\rig,K}^\dag(V_S)$ are stable under $\varphi$ and isomorphic to their $\varphi$-pullbacks respectively. That is, the natural morphisms $\varphi^*(\D^{\dag}_K(V_S))
\ra\D^{\dag}_K(V_S)$ and $\varphi^*(\D^{\dag}_{\rig,K}(V_S))\ra\D^{\dag}_{\rig,K}(V_S)$ are isomorphisms.
Thus $\D_K^\dagger(V_S)$ is a $(\varphi,\Gamma)$-module over $\mathbf{B}_K^{\dagger}\widehat{\otimes}_{\Qp}S$
in the sense of \cite{KL}. (See Remark \ref{rem:Drig-phi-Gamma} for the relevant discussion about $\D_{\rig,K}^\dag(V_S)$.)

\begin{remark}
In the case when $V_S$ admits a $G_K$-stable free $\OO_S$-lattice $T_S$, we further have that the $\m$-module $\D_K^\dag(V_S)$ is \emph{globally \'etale} in the sense of \cite{L08}. In fact, if $L$ is a finite Galois extension of $K$ so that $G_L$ acts trivially on $T_S/12pT_S$, 
\[
\mathbf{A}^{\dag}_K(T_S)=\cup_{s\geq s(V_S)}(\varphi^{n(L)}(\D^{\dag,\frac{p-1}{p}}_{L,n(L)}(T_S))\otimes_{\mathbf{A}^{\dag,r_{n(L)}}\widehat{\otimes}_{\Z}\OO_S}
\mathbf{A}_L^{\dag,s}\widehat{\otimes}_{\Z}\OO_S)^{H_K}
\]
is a locally free $\mathbf{A}_L^{\dag}\widehat{\otimes}_{\Z}\OO_S$-lattice of $\D_K^\dag(V_S)$ which satisfies
\[
\varphi^*(\mathbf{A}^{\dag}_K(T_S))=\mathbf{A}^{\dag}_K(T_S).
\]
\end{remark}

\begin{cor}\label{cor:dagger-phi-compare}
Let $a\in\D^{\dag,ps}_{\rig,K}(V_S)$. If $\varphi(a)\in\D^{\dag,ps}_{\rig,K}(V_S)$, then $a\in\D^{\dag,s}_{\rig,K}(V_S)$.
\end{cor}
\begin{proof}
Let $T_S$ be a free $\OO_S$-lattice of $V_S$, and let $L$ be a finite Galois extension of $K$ so that $G_L$ acts trivially on $T_S/12pT_S$. By its construction $\D^{\dag,s}_L(V_S)$ is a free $\mathbf{B}^{\dag,s}_L\widehat{\otimes}_{\Q}S$-module of rank $d$. Let $e_1,\dots,e_d$ be a basis, and write $a=\sum_{i=1}^da_ie_i$ with $a_i\in\mathbf{B}^{\dag,ps}_{\rig,L}\widehat{\otimes}_{\Q}S$. Since $\D^{\dag,s}_{\rig,K}(V_S)=(\D^{\dag,s}_{\rig,L}(V_S))^{H_K}$, it reduces to show that $a\in\D^{\dag,s}_{\rig,L}(V_S)$. By Proposition \ref{prop:dagger-phi-compare}, $\varphi(e_1),\dots,\varphi(e_d)$ form a $\mathbf{B}^{\dag,ps}_{\rig,L}\widehat{\otimes}_{\Q}S$-basis
of $\D^{\dag,s}_L(V_S)$. Hence $\varphi(a)=\sum_{i=1}^d\varphi(a_i)\varphi(e_i)$ belongs to $\D^{\dag,ps}_{\rig,L}(V_S)$ if and only if $\varphi(a_i)\in\mathbf{B}^{\dag,ps}_{\rig,L}\widehat{\otimes}_{\Q}S$ for all $i$. The latter is equivalent to $a_i\in\mathbf{B}^{\dag,s}_{\rig,L}\widehat{\otimes}_{\Q}S$ for all $i$. This yields the desired result.
\end{proof}


\subsection{Sheafification of the $\m$-module functor}
Following \cite{KL}, we extend the $\m$-module functors to finite locally free $S$-linear representations as follows. From now on, let $V_S$ be a locally free $S$-linear representation of $G_K$ of rank $d$. We choose a finite covering of $M(S)$ by affinoid subdomains $M(S_1),\dots, M(S_m)$ such that $V_{S_i}$ is free over $S_i$ for each $i$. Let $s_0=\max_{1\leq i\leq m}\{s(V_{S_i})\}$. By \cite[Lemma 3.3]{KL} and Theorem \ref{thm:BC-generalization}, for any $s\geq s_0$ and $1\leq i\leq m$, the presheaf $\mathscr{D}_K^{\dag,s}(V_{S_i}): R_i\mapsto \D_K^{\dag,s}(V_{R_i})$, where $M(R_i)$ runs through all affinoid subdomains of $M(S_i)$, is indeed a sheaf for the weak $G$-topology of $M(S_i)$ (hence extends uniquely to the strong $G$-topology).  We glue the sheaves  $\mathscr{D}_K^{\dag,s}(V_{S_i})$ for all $1\leq i\leq m$ to form a sheaf $\mathscr{D}^{\dag,s}_K(V_S)$ on $M(S)$, which is independent of the choice of the covering. It turns out that $\mathscr{D}_K^{\dag,s}(V_S)$ is the sheaf associated to a finite locally free $\mathbf{B}_K^{\dag,s}\widehat{\otimes}_{\Q}S $-module $\D^{\dag,s}_{K}(V_S)$ \cite[Proposition 3.6]{KL}. It is straightforward to see that $\D_K^{\dag,s}(V_S)$ is equipped with a natural $\Gamma$-action, and the construction $V_S\mapsto \D_K^{\dag,s}(V_S)$ satisfies the analogues of the assertions of Theorem \ref{thm:BC-generalization}. 

We then define the functors $\D_{\rig,K}^{\dag,s}(V_S)$, $\D_K^{\dag}(V_S)$
and $\D^{\dag}_{\rig,K}(V_S)$ as in \S1.1. The sheaf property for $\mathscr{D}_K^{\dag,s}(V_S)$ and Proposition \ref{prop:dagger-phi-compare} ensure that $\D_{K}^{\dag}(V_S)$ is isomorphic to its $\varphi$-pullback. Hence $\D_{K}^{\dag}(V_S)$ and $\D_{\rig,K}^{\dag}(V_S)$ are $\m$-modules over $\mathbf{B}_K^{\dag}\widehat{\otimes}_{\Q}S$ and $\mathbf{B}^{\dag}_{\rig,K}\widehat{\otimes}_{\Q}S$ respectively. 

Note that for any affinoid subdomain $M(S')$ of $M(S)$,  one may proceed the above constructions for $V_{S'}$ using the covering $M(S_1)\cap M(S'),\dots, M(S_m)\cap M(S')$. In particular, one can define $\D_{\rig,K}^{\dag,s}(V_{S'})$ for any $s\geq s_0$.

\begin{defn}
For any $s\geq s_0$, define the presheaves $\mathscr{D}_{\rig,K}^{\dag,s}(V_S)$ and $\mathscr{D}_{\rig,K}^{\dag}(V_S)$ on the weak $G$-topology of $M(S)$ by setting
\[
\mathscr{D}_{\rig,K}^{\dag,s}(V_S)(M(S'))=\D_{\rig,K}^{\dag,s}(V_{S'}),
\quad
\mathscr{D}_{\rig,K}^{\dag}(V_S)(M(S'))=\D_{\rig,K}^{\dag}(V_{S'})
\]
for any affinoid subdomain $M(S')$ of $M(S)$.
\end{defn}

\begin{prop}\label{prop:sheaf-rig}
Both $\mathscr{D}_{\rig,K}^{\dag,s}(V_S)$ and $\mathscr{D}_{\rig,K}^{\dag}(V_S)$ are sheaves for the weak $G$-topology of $M(S)$, and hence extend uniquely to the strong $G$-topology. 
\end{prop}
\begin{proof}
We first show that $\mathscr{D}_{\rig,K}^{\dag,s}(V_S)$ is a sheaf. This amounts to checking the sheaf condition for finite coverings of affinoid subdomains by affinoids. Recall that 
\[
\mathscr{D}_{\rig,K}^{\dag,s}(V_S)(M(S'))=\D_{\rig,K}^{\dag,s}(V_{S'})=\D_{\rig,K}^{\dag,s}(V_{S})\otimes_{ \mathbf{B}_{\rig,K}^{\dag,s}\widehat{\otimes}_{\Q}S} \mathbf{B}_{\rig,K}^{\dag,s}\widehat{\otimes}_{\Q}S'
\]
for any affinoid subdomain $M(S')$ of $M(S)$. Since $\D_{\rig,K}^{\dag,s}(V_S)$ is
a finite locally free $\mathbf{B}_{\rig,K}^{\dag,s}\widehat{\otimes}_{\Q}S$-module,  it reduces to show that the presheaf $M(S')\mapsto  \mathbf{B}_{\rig,K}^{\dag,s}\widehat{\otimes}_{\Q}S'$ is a sheaf on the weak $G$-topology of $M(S)$. By definition, $\mathbf{B}_{\rig,K}^{\dag,s}$ is the Fr\'echet completion of $\mathbf{B}_{K}^{\dag,s}$ with respect to the set of valuations $\{\mathrm{val}^{(0,r]}\}_{r\geq s}$. For $r\geq s$, let $\mathbf{B}_K^{[s,r]}$ be the completion of $\mathbf{B}_{K}^{\dag,s}$ with respect to $\max\{\mathrm{val}^{(0,r]}, \mathrm{val}^{(0,s]}\}$. It follows that 
\[
\mathbf{B}_{\rig,K}^{\dag,s}=\limproj_{r}\mathbf{B}_K^{[s,r]}.
\] 
Using a Schauder basis of $S$, we deduce
\begin{equation}\label{eq:intersection}
\mathbf{B}_{\rig,K}^{\dag,s}\widehat{\otimes}_{\Q}S=\limproj_{r}\mathbf{B}_K^{[s,r]}\widehat{\otimes}_{\Q}S.
\end{equation}
Therefore, it suffices to show that
the presheaf defined by $M(S')\mapsto\mathbf{B}^{[s,r]}_K\widehat{\otimes}_{\Q}S'$ is a sheaf on the weak $G$-topology of $M(S)$; this follows from \cite[Lemma 3.3]{KL}.

Note that the presheaf $\mathscr{D}_{\rig,K}^\dag(V_S)$ is the direct limit of the sheaves $\mathscr{D}_{\rig,K}^{\dag,s}(V_S)$ in the category of presheaves. In general, the direct limit of sheaves in the category of presheaves is not necessarily a sheaf. However, using the facts that the coverings are all finite and the connecting maps $\mathscr{D}_{\rig,K}^{\dag,s_1}(V_S)\ra\mathscr{D}_{\rig,K}^{\dag,s_2}(V_S)$ are injective for all $s_1<s_2$, it is straightforward to check that the direct limit of the sheaves $\mathscr{D}_{\rig,K}^{\dag,s}(V_S)$ in the category of presheaves is indeed a sheaf.
\end{proof}

\begin{theorem}\label{thm:rig-locally-free}
The constructions $\D_{\rig,K}^{\dag,s}(V_S)$ and $\D_{\rig,K}^{\dag}(V_S)$ for finite locally free $S$-linear representations $V_S$ have the same properties as for finite free $S$-linear representations given in \S1.1.
\end{theorem}
\begin{proof}
We choose a finite covering of $M(S)$ by affinoid subdomains such that the restriction of $V_S$ on each piece is free. The theorem then follows from Proposition \ref{prop:sheaf-rig}.
\end{proof}

The following lemma will be used in \S3.

\begin{lemma}\label{lem:phi-rig}
Let $a\in\D_{\rig,K}^\dagger(V_S)$ and $\alpha\in S$.  If $\varphi^m(a)-\alpha a\in\D_{\rig,K}^{\dag,p^ms}(V_S)$, then $a\in\D_{\rig,K}^{\dagger,s}(V_S)$.
\end{lemma}
\begin{proof}
Put $b=\varphi^m(a)-\alpha a$. Suppose that $a\in \D_{\rig,K}^{\dagger,s'}(V_{S})$ for some $s'$. If $s'>s$, we get $\varphi^m(a)=b+\alpha a\in \D_{\rig,K}^{\dagger,s'}(V_{S})$. It follows from Corollary \ref{cor:dagger-phi-compare} and Theorem \ref{thm:rig-locally-free} that $a\in \D_{\rig,K}^{\dagger,s''}(V_{S})$ for $s''=\max\{s'/p^m,s\}$. We then conclude $a\in \D_{\rig,K}^{\dagger,s}(V_{S})$
by iterating this argument.
\end{proof}

\begin{defn}
Let $X$ be a rigid analytic space over $\Q$, and let $V_X$ be a locally free coherent $\OO_X$-module equipped with a continuous $\OO_X$-linear $G_K$-action. We choose an admissible covering of $X$ by affinoid subdomains $\{M(S_i)\}_{i\in I}$. We then define the sheaf  $\mathscr{D}_{\rig,K}^{\dag}(V_X)$ by gluing the sheaves  $\mathscr{D}_{\rig,K}^{\dag}(V_{S_i})$
for all $i\in I$; this construction is independent of the choice of the covering $\{M(S_i)\}_{i\in I}$.
\end{defn}

\subsection{Localization maps}
Recall that Fontaine's $p$-adic $2\pi i$ is defined as $t=\log[\varepsilon]$. We equip $K_n[[t]]$ with the induced Fr\'echet topology via the natural identification $K_n[[t]]\cong K_n^{\mathbb{N}}$. We define $K_n((t))\widehat{\otimes}_{\Q}S$ as the inductive limit of  $(t^{-i}K_n[[t]])\widehat{\otimes}_{\Q}S$.
Recall that for any $n\geq n(s)$, there is a continuous $\Gamma$-equivariant injective map 
\[
\iota_n:\mathbf{B}_{K}^{\dagger,s}\rightarrow K_n[[t]],
\] 
which extends to a continuous $\Gamma$-equivariant injective map 
\[
\iota_n:\mathbf{B}_{\rig, K}^{\dagger,s}\ra K_n[[t]].
\]
It is defined as the composite
\[
\xymatrix{
\mathbf{B}_{K}^{\dagger,s}\subset\widetilde{\mathbf{B}}^{\dag,s}\stackrel{\varphi^{-n}}{\to}\widetilde{\mathbf{B}}^{\dag,p^{-n}s}\subset \widetilde{\mathbf{B}}^+\subset\mathbf{B}_{\mathrm{dR}}^+,}
\]
and it factors through $K_n[[t]]$  (see \cite[\S2]{LB02} for more details about $\iota_n$).  In particular we have $\iota_{n+1}\circ\varphi=\iota_n$. The map $\iota_n$ induces a continuous $\Gamma$-equivariant map 
\[
\iota_n: \mathbf{B}_{\rig,K}^{\dagger,s}\widehat{\otimes}_{\Q}S\ra K_n[[t]]\widehat{\otimes}_{\Q}S.
\] We define
\[
\D_{\dif}^{+,K_n}(V_S)=\D_{\rig,K}^{\dagger,s}(V_S)
\otimes_{\mathbf{B}_{\rig,K}^{\dagger,s}\widehat{\otimes}_{\Q}S,\iota_n}(K_n[[t]]\widehat{\otimes}_{\Q}S)
\]
and
\[
\D_{\dif}^{K_n}(V_S)=\D_{\rig,K}^{\dagger,s}(V_S)\otimes_{\mathbf{B}_{\rig,K}^{\dagger,s}\widehat{\otimes}_{\Q}S,\iota_n}
(K_n((t))\widehat{\otimes}_{\Q}S);
\]
it is clear that $\D_{\dif}^{K_n}(V_S)=\D_{\dif}^{+,K_n}(V_S)[1/t]$. We denote by $\iota_n$ the natural map
\[
\D_{\rig,K}^{\dag,s}(V_S)\ra\D_{\dif}^{+,K_n}(V_S),
\]
and call it the \emph{localization map}. It is straightforward to see that $\varphi: \mathbf{B}_{\rig,K}^{\dagger,s}\ra \mathbf{B}_{\rig,K}^{\dagger,ps}$ induces a $K_n[[t]]$-linear morphism $\D_{\dif}^{+,K_n}(V_S)\ra \D_{\dif}^{+,K_{n+1}}(V_S)$ which induces an isomorphism  
\[
\D_{\dif}^{+,K_n}(V_S)\otimes _{K_n[[t]]}K_{n+1}[[t]]\cong \D_{\dif}^{+,K_{n+1}}(V_S).
\]
We define $\D_{\mathrm{Sen}}^{K_n}(V_S)=\D_{\dif}^{+,K_n}(V_S)/(t)$. Finally, we define
\[
\D_\dif^{+,K_n}(V_S)=\cup_{n\geq n(s)}\D_\dif^{+,K_n}(V_S),\quad \D_\dif^{K}(V_S)=\cup_{n\geq n(s)}\D_\dif^{K_n}(V_S),
\quad\D_{\mathrm{Sen}}^{K}(V_S)=\cup_{n\geq n(s)}\D_{\mathrm{Sen}}^{K_n}(V_S).
\]

\begin{convention}
When the base field $K$ is clear, we omit $K$ in all of these functors for simplicity.
\end{convention}

By their constructions and the base change properties of $\m$-module functors, the following proposition is obvious.

\begin{prop}\label{prop:base-change}
The functors $\D^{+,n}_\mathrm{dif}$ and $\D_{\mathrm{Sen}}^n$ are compatible with base change.
\end{prop}


Let $q=\varphi([\varepsilon]-1)/([\varepsilon]-1)$.
The following proposition is a generalization of \cite[Theorem 4.3]{L07}.
\begin{prop}\label{prop:brig-mod-phi-t}
Let $k$ be a positive integer. The following are true.
\begin{enumerate}
\item[(1)]
The localization map $\iota_n:\D^{\dag,s}_{\rig}(V_S)\ra\D_{\dif}^{+,n}(V_S)$ induces an isomorphism
\[
\D^{\dag,s}_{\rig}(V_S)/(\varphi^{n-1}(q))^k\cong \D_{\dif}^{+,n}(V_S)/(t^k).
\]
\item[(2)]
The natural map $\prod_{n\geq n(s)}\iota_n:\D^{\dag,s}_{\rig}(V_S)\ra\prod_{n\geq n(s)}\D_\dif^{+,n}(V_S)$ induces an isomorphism
\[
\D^{\dag,s}_{\rig}(V_S)/(t^k)\cong\prod_{n\geq n(s)}\D_{\dif}^{+,n}(V_S)/(t^k).
\]
\item[(3)]
The natural map $\varphi:\D^{\dag,s}_{\rig}(V_S)/(t^k)\ra\D^{\dag,ps}_{\rig}(V_S)/(t^k)$ is given by $((a_n)_{n\geq n(s)})\ra((a_{n-1})_{n\geq n(s)+1})$ under the isomorphism of $(2)$,
\end{enumerate}
\end{prop}
\begin{proof}
For (1) and (2), since $\D_{\rig}^{\dag,s}(V_S)$ is a finite locally free $\mathbf{B}^{\dag,s}_{\rig,K}\widehat{\otimes}_{\Q}S$-module, it reduces to show that
\begin{equation}\label{eq:brig-mod-phi}
(\mathbf{B}^{\dag,s}_{\rig,K}\widehat{\otimes}_{\Q}S)/(\varphi^{n-1}(q))^k\cong (K_n[[t]]\widehat{\otimes}_{\Q}S)/(t^k)
\end{equation}
and
\begin{equation}\label{eq:brig-mod-t}
(\mathbf{B}_{\rig,K}^{\dag,s}\widehat{\otimes}_{\Q}S)/(t^k)\cong \prod_{n\geq n(s)}(K_n[[t]]\widehat{\otimes}_{\Q}S)/(t^k).
\end{equation}
We first show them for $S=\Q$. By \cite[Proposition 4.8]{LB02}, for $f\in \mathbf{B}^{\dag,s}_{\rig,K}$, $t|\iota_n(f)$ if and only if $\varphi^{n-1}(q)|f$. Note that $t||\iota_n(\varphi^{n-1}(q))$. We thus deduce that the map 
\[
\mathbf{B}^{\dag,s}_{\rig,K}/(\varphi^{n-1}(q))^k\ra K_n[[t]]/(t^k)
\] 
is injective. Furthermore, it is an isomorphism for $k=1$ by \cite[lemme 4.9]{LB02}. It follows that it is an isomorphism for any $k$. Since $t=\prod_{n\geq n(s)}(\varphi^{n-1}(q)/p)$ in $\mathbf{B}_{\rig,K}^{\dag,s}$, we further get
\[
\mathbf{B}_{\rig,K}^{\dag,s}/(t^k)\cong\prod_{n\geq n(s)}\mathbf{B}^{\dag,s}_{\rig,K}/\varphi^{n-1}(q)\cong\prod_{n\geq n(s)} K_n[[t]]/(t^k).
\]

We claim that the exact sequence
\[
0\ra (\varphi^{n-1}(q))^k\mathbf{B}_{\rig,K}^{\dag,s}\ra\mathbf{B}^{\dag,s}_{\rig,K}
\ra\mathbf{B}^{\dag,s}_{\rig,K}/(\varphi^{n-1}(q))^k\ra0
\]
splits as complete Fr\'echet spaces over $\Q$. Note that the quotient  
$\mathbf{B}^{\dag,s}_{\rig,K}/(\varphi^{n-1}(q))^k$ is a finite dimensional $\Q$-vector space. We choose a section of the $\Qp$-linear map $\mathbf{B}^{\dag,s}_{\rig,K}
\ra\mathbf{B}^{\dag,s}_{\rig,K}/(\varphi^{n-1}(q))^k$, and denote its image by $U$. Note that for any $r\geq s$, 
$\mathrm{val}^{(0,r]}$ is indeed a norm on $\mathbf{B}^{\dag,s}_{\rig,K}$.  Since every finite dimensional normed $\Q$-vector space is complete,  we deduce that $U$ is a closed Fr\'echet subspace of $\mathbf{B}^{\dag,s}_{\rig,K}
$. By open mapping theorem for Fr\'echet spaces \cite[Proposition 8.6]{S02},  the map $(\varphi^{n-1}(q))^k\mathbf{B}_{\rig,K}^{\dag,s}\oplus U\ra\mathbf{B}^{\dag,s}_{\rig,K}
$ is an isomorphism of $\Q$-Fr\'echet spaces. This proves the claim.

The claim yields the following exact sequence
\[
0\ra (\varphi^{n-1}(q))^k\mathbf{B}_{\rig,K}^{\dag,s}\widehat{\otimes}_{\Q}S\ra
\mathbf{B}^{\dag,s}_{\rig,K}\widehat{\otimes}_{\Q}S
\ra(\mathbf{B}^{\dag,s}_{\rig,K}/(\varphi^{n-1}(q))^k)\widehat{\otimes}_{\Q}S\ra0;
\]
hence
\begin{equation*}
\begin{split}
(\mathbf{B}^{\dag,s}_{\rig,K}\widehat{\otimes}_{\Q}S)/(\varphi^{n-1}(q))^k&\cong
(\mathbf{B}^{\dag,s}_{\rig,K}/(\varphi^{n-1}(q))^k)\widehat{\otimes}_{\Q}S\\
&\cong K_n[[t]]/(t^k)\otimes_{\Q} S\cong (K_n[[t]]\widehat{\otimes}_{\Q}S)/(t^k),
\end{split}
\end{equation*}
yielding (\ref{eq:brig-mod-phi}). We get (\ref{eq:brig-mod-t}) by a similar argument. We get (3) immediately from the fact that $\iota_{n+1}\circ\varphi=\iota_n$ for all $n\geq n(s)$.
\end{proof}

Note that $\varphi^f$ acts $K_0$-linearly on $\D^{\dag}_{\rig}(V_S)$. We extend the $\varphi^f$-action to  $K\otimes_{K_0}\D^{\dag}_{\rig}(V_S)$ $K$-linearly. 
For $s\geq s(V_S)$ and $n\geq\max\{n(V_S), n(s)/f\}$, we set 
\[
\iota_{n,K}:K\otimes_{K_0}\D^{\dag,s}_{\rig}(V_S)\ra \D_{\dif}^{+,fn}(V_S)
\]
as the $K$-linear extension of $\iota_{fn}$.
Recall that every closed ideal of  $\mathbf{B}^{\dag,s}_{\rig,K}$ is principal \cite[Theorem 2.9.6]{Ke05}. It follows that every closed ideal of $K\otimes_{K_0}\mathbf{B}^{\dag,s}_{\rig,K}$ is principal. Thus the closed ideal $\iota_{n,K}^{-1}((t))$ of $K\otimes_{K_0}\mathbf{B}^{\dag,s}_{\rig,K}$ is principal; we fix a generator $q_{n,K}$ of it. It follows that $\iota_{n,K}$ induces an isomorphism $(K\otimes_{K_0}\mathbf{B}^{\dag,s}_{\rig,K})/(q^k_{n,K})\cong K_n[[t]]/(t^k)$. Again, the closed ideal $\displaystyle\cap_n(q_{n,K})$ is principal; we fix a generator $t_K$ of it.

\begin{prop}\label{prop:localization-ramified}
The following are true.
\begin{enumerate}
\item[(1)]The ideal $(q_{n,K})$ is a prime factor of $(\varphi^{fn-1}(q))$.
\item[(2)]The map $\iota_{n,K}$ induces an isomorphism $(K\otimes_{K_0}\D^{\dag,s}_{\rig}(V_S))/(q^k_{n,K})\cong \D_{\dif}^{+,fn}(V_S)/(t^k)$ for any $k\geq 1$ and $s\leq r_{fn}$.
\item[(3)]We have $\varphi^{f}((q_{n,K}))=(q_{n+1,K})$.
\item[(4)]For any $k\geq 1$, the natural map $\prod_{fn\geq n(s)}\iota_{n,K}:K\otimes_{K_0}\D^{\dag,s}_{\rig}(V_S)\ra\prod_{fn\geq n(s)}\D_\dif^{+,fn}(V_S)$ induces an isomorphism
\[
(K\otimes_{K_0}\D^{\dag,s}_{\rig}(V_S))/(t_K^k)\cong\prod_{fn\geq n(s)}\D_{\dif}^{+,fn}(V_S)/(t^k).
\]
\item[(5)]The valuation of $\varphi^f(t_K)/t_K$, which is viewed as an element of $K\otimes_{K_0}\mathbf{B}^{\dag}_{K}$, is $1$; here we put the valuation of $\pi_K$ in $K\otimes_{K_0}\mathbf{B}^{\dag}_{K}$ to be 1.
\end{enumerate}
\end{prop}
\begin{proof}
By Proposition \ref{prop:brig-mod-phi-t}(1), $\iota_{fn}$ induces an isomorphism
\[
(K\otimes_{K_0}\D^{\dag,s}_{\rig}(V_S))/(\varphi^{fn-1}(q)^k)\cong K\otimes_{K_0}\D_{\dif}^{+,fn}(V_S)/(t^k).
\]
The map $K\otimes_{K_0}\D^{\dag,s}_{\rig}(V_S)\ra\D_{\dif}^{+,fn}(V_S)/(t^k)$ is then just the composite 
\[
K\otimes_{K_0}\D^{\dag,s}_{\rig}(V_S)\ra (K\otimes_{K_0}\D^{\dag,s}_{\rig}(V_S))/(\varphi^{fn-1}(q)^k)
\cong K\otimes_{K_0}\D_{\dif}^{+,fn}(V_S)/(t^k)\ra\D_{\dif}^{+,fn}(V_S)/(t^k).
\]
This implies the first two statements. We deduce (3) from Proposition \ref{prop:brig-mod-phi-t}(3). Note that the ideal $(t_K)$ is the product of all $(q_{n,K})$ which are mutually prime by (1). We then deduce (4) using a similar argument as in the proof of Proposition \ref{prop:brig-mod-phi-t}(2). For (5), we choose a generator $t_f\in \mathbf{B}^{\dag,s}_{\rig,K}$ of the closed ideal $\cap_{n}(\varphi^{fn-1}(q))$. By Proposition \ref{prop:brig-mod-phi-t}(1), (2), it is straightforward to see that 
\[
\prod_{i=0}^{f-1}\varphi^i((t_f))=(t).
\] 
This yields that $\varphi^f(t_f)/t_f$ belongs to $(\mathbf{B}^{\dag,s}_{\rig,K})^\times\subset\mathbf{B}^{\dag}_{K}$ and has valuation $1$ as $\varphi(t)=pt$; hence its valuation in $K\otimes_{K_0}\mathbf{B}_K^\dag$ is equal to the ramification index $e=[K:\Q]/f$. 
After a suitable base change, we may assume that $K$ is Galois over $K_0$. A short computation shows that 
\[
\prod_{\sigma\in\mathrm{Gal}(K/K_0)}\sigma((t_K))=(t_f)
\] 
as ideals of $K\otimes_{K_0}\mathbf{B}^{\dag,s}_{\rig,K}$. It follows that 
\[
\varphi^f(\sigma(t_K))/\sigma(t_K)\in (K\otimes_{K_0}\mathbf{B}^\dag_{\rig,K})^\times\subset K\otimes_{K_0}\mathbf{B}^\dag_{K}
\]
for each $\sigma\in\mathrm{Gal}(K/K_0)$, and their product is $\varphi^f(t_f)/t_f$ up to a unit of $K\otimes_{K_0}\mathbf{B}_K^\dag$. Since they all have the same valuation, we conclude (5) by the fact that $\varphi^f(t_f)/t_f$ has valuation $e$ in $K\otimes_{K_0}\mathbf{B}_K^\dag$. 
\end{proof}

\subsection{The sheaf $(\mathscr{D}_{\mathrm{dif}}^{+,n}(V_S)/(t^k))^{\Gamma}$}
\begin{defn}
Define the presheaves $\mathscr{D}_{\mathrm{dif}}^{+,n}(V_S)$ and $\mathscr{D}_{\mathrm{dif}}^{+,n}(V_S)/(t^k)$ on the weak $G$-topology of $M(S)$ by setting
\[
(\mathscr{D}_{\mathrm{dif}}^{+,n}(V_S))(M(S'))=\mathrm{D}_{\mathrm{dif}}^{+,n}(V_{S'}), \quad (\mathscr{D}_{\mathrm{dif}}^{+,n}(V_S)/(t^k))(M(S'))=\mathrm{D}_{\mathrm{dif}}^{+,n}(V_{S'})/(t^k)
\]
for any affinoid subdomain $M(S')$ of $M(S)$. Define the presheaves
\[
\mathscr{D}_{\mathrm{dif}}^{+}(V_S)=\limind_{n\ra\infty}\mathscr{D}_{\mathrm{dif}}^{+,n}(V_S)\quad\text{and}\quad
\mathscr{D}_{\mathrm{dif}}^{+}(V_S)/(t^k)=\limind_{n\ra\infty}\mathscr{D}_{\mathrm{dif}}^{+,n}(V_S)/(t^k),
\]
where the transition maps on any affinoid subdomain $M(S')$ of $M(S)$ are the natural $K_n[[t]]$-linear morphisms 
$\mathrm{D}_{\mathrm{dif}}^{+,n}(V_{S'})\ra\mathrm{D}_{\mathrm{dif}}^{+,n+1}(V_{S'})$ and $\mathrm{D}_{\mathrm{dif}}^{+,n}(V_{S'})/(t^k)\ra\mathrm{D}_{\mathrm{dif}}^{+,n+1}(V_{S'})/(t^k)$ introduced in \S1.3. 
\end{defn}

\begin{prop}\label{prop:sheaf-dif}
The presheaves $\mathscr{D}_\dif^{+,n}(V_S)$ and $\mathscr{D}_\dif^{+}(V_S)$ are sheaves for the weak $G$-topology of $M(S)$, and hence extend uniquely to the strong $G$-topology.
\end{prop}
\begin{proof}
As in the proof of Proposition \ref{prop:sheaf-rig}, by the base change property of the functor $\mathrm{D}^{+,n}_{\mathrm{dif}}$, it reduces to show that the presheaf
\[
M(S')\mapsto K_n[[t]]\widehat{\otimes}_{\Q}S'
\] 
is a sheaf. The latter is clear. 
\end{proof}

\begin{lemma}\label{lem:flat-invariant}
Let $G$ be a topologically finitely generated group. Let $A$ be a commutative Hausdoff topological ring, and let $M$ be a finite $A$-module equipped with a continuous $A$-linear action of $G$. Suppose $B$ is a commutative Hausdoff topological ring with a continuous flat morphism $A\ra B$. Then $(M\otimes_AB)^G=M^G\otimes_AB$.
\end{lemma}
\begin{proof}
Choose a finite set of topological generators $g_1,\dots,g_n$ of $G$. Consider the exact sequence
\[
0\longrightarrow M^{G}\longrightarrow M
\longrightarrow\oplus_{i=1}^nM
\]
where the last map is $m\mapsto\oplus_{i=1}^n(g_i-1)m$. Since $B$ is flat over $A$, tensoring up with $B$, we get
\[
0\longrightarrow M^G\otimes_AB
\longrightarrow M\otimes_AB
\longrightarrow\oplus_{i=1}^mM\otimes_AB.
\]
This yields the lemma.
\end{proof}

\begin{convention}
Let $X$ be a rigid analytic space over $\Q$. Let $G$ be a group, and let $M$ be a presheaf on $X$ equipped with a $G$-action. We denote by $M^{G}$ the presheaf  on $X$ defined by
$M^{G}(U)=M(U)^G$ for any admissible open subset $U$ of $X$.
\end{convention}

\begin{prop}\label{prop:dif-torsion-sheaf}
The following are true.
\begin{enumerate}
\item[(1)] The presheaf $\mathscr{D}_{\mathrm{dif}}^{+,n}(V_S)/(t^k)$ is a locally free coherent sheaf.
\item[(2)]The presheaf $\mathscr{D}_{\mathrm{dif}}^{+}(V_S)/(t^k)$ is a sheaf. 
\item[(3)] The presheaf $(\mathscr{D}_{\mathrm{dif}}^{+,n}(V_S)/(t^k))^{\Gamma}$ is a coherent sheaf.
\item[(4)] The presheaf $(\mathscr{D}_{\mathrm{dif}}^{+}(V_S)/(t^k))^{\Gamma}$ is a sheaf.
\end{enumerate}
\end{prop}

\begin{proof}
By the previous lemma, (1) implies (3).  By the same argument as in the proof of Proposition \ref{prop:sheaf-rig}, (1) implies (2) and (3) implies (4) respectively. Thus it suffices to prove (1). Note that $\mathrm{D}^{+,n}_{\mathrm{dif}}(V_{S'})/(t^k)$ is a locally free $S'$-module of finite rank for any affinoid subdomain $M(S')$ of $M(S)$.  We only need to show that $\mathscr{D}_{\mathrm{dif}}^{+,n}(V_S)/(t^k)$ satisfies the sheaf properties. The latter follows from the fact that the presheaf
\[
M(S')\mapsto (K_n[[t]]\widehat{\otimes}_{\Q}S')/(t^k)
\] 
is a sheaf as in the proof of Proposition \ref{prop:sheaf-rig}.
\end{proof}

We also denote the sheaf $\mathscr{D}_{\mathrm{dif}}^+(V_S)/(t)$ by $\mathscr{D}_{\mathrm{Sen}}(V_S)$.

\begin{defn}
Let $X$ be a rigid analytic space over $\Q$, and let $V_X$ be a locally free coherent $\OO_X$-module equipped with a continuous $\OO_X$-linear $G_K$-action. We choose an admissible covering of $X$ by affinoid subdomains $\{M(S_i)\}_{i\in I}$. We then define the sheaf $\mathscr{D}_{\dif}^{+}(V_X)$ (resp. $\mathscr{D}_{\dif}^{+}(V_X)/(t^k)$, $\mathscr{D}_{\Sen}(V_X)$) by gluing the sheaves  $\mathscr{D}_{\dif}^{+}(V_{S_i})$ (resp. $\mathscr{D}_{\dif}^{+}(V_{S_i})/(t^k)$, $\mathscr{D}_{\Sen}(V_{S_i})$)
for all $i\in I$; this construction is independent of the choice of the covering $\{M(S_i)\}_{i\in I}$.
\end{defn}

We need the following result in \S4.
\begin{prop}
\label{prop:torsion-free}
If $S$ is torsion-free, then both 
\[
(\D_{\mathrm{dif}}^{+,n}(V_S)/(t^k))^{\Gamma}\quad\text{and} \quad
(\D_{\mathrm{dif}}^{+,n}(V_S)/(t^k))/(\D_{\mathrm{dif}}^{+,n}(V_S)/(t^k))^{\Gamma}
\] 
are torsion-free $S$-modules.
\end{prop}
\begin{proof}
Since $\D_{\mathrm{dif}}^{+,n}(V_S)/(t^k)$ is a finitely generated locally free $S$-module, it is torsion-free by the assumption on $S$. So $(\D_{\mathrm{dif}}^{+,n}(V_S)/(t^k))^\Gamma$ is torsion-free as well. On the other hand, for $s\in S, a\in \D_{\mathrm{dif}}^{+,n}(V_S)/(t^k)$ and $\gamma\in \Gamma$, if $\gamma(sa)=sa$, then $\gamma(a)=a$ because $\gamma$ acts trivially on $S$ and $\D_{\mathrm{dif}}^{+,n}(V_S)/(t^k)$ is torsion-free. This yields that $(\D_{\mathrm{dif}}^{+,n}(V_S)/(t^k))/(\D_{\mathrm{dif}}^{+,n}(V_S)/(t^k))^{\Gamma}$ is a torison-free $S$-module.
\end{proof}

\subsection{Sen operator}
Let $V_S$ be a free $S$-linear representation of $G_K$ of rank $d$. Let $T_S$ and $L$ be as in the construction of $\D^{\dagger}_K(V_S)$, and let $n\geq n(V_S)$. By its construction, the module $\D_{\Sen}^{L_n}(V_S)$ is free of rank $d$ over $L_n\otimes_{\Q}S$. Furthermore, for any $\gamma\in\Gamma_L$ satisfying $n(\gamma)\geq n$, we may choose some $L_n\otimes_{\Q}S$-basis of $\D_{\Sen}^{L_n}(V_S)$ so that the matrix $M_\gamma$ of $\gamma$ under this basis satisfies $|M_\gamma-1|<1$. We then define $\log\gamma\in\mathrm{End}_{L_n\otimes_{\Q}S}(\D_{\Sen}^{L_n}(V_S))$ by setting
\[
\log\gamma=-\sum_{m\geq1}\frac{(1-\gamma)^m}{m}.
\]
The convergence of the right hand side follows from the condition $|M_\gamma-1|<1$.  Since $\Gamma_L$ is a 1-dimensional $p$-adic Lie group, the operator 
\[
\Theta=\log\gamma/\log_p\chi(\gamma)\in\mathrm{End}_{L_n\otimes_{\Q}S}(\D_{\Sen}^{L_n}(V_S))
\]
is independent of the choice of $\gamma$; hence it is well-defined. Note that 
\[
\D_{\mathrm{Sen}}^{L_n}(V_S)=\D_{\mathrm{Sen}}^{K_n}(V_S)\otimes_{K_n}L_n,
\] 
and $\gamma$ carries $\D_{\mathrm{Sen}}^{K_n}(V_S)$ into itself. Hence we may view $\Theta$  as an element of $\mathrm{End}_{K_n\otimes_{\Q}S}(\D_{\Sen}^{K_n}(V_S))$.  Furthermore, since $\Gamma$ is commutative,  $\Theta$ commutes with $\Gamma$; hence its characteristic polynomial has coefficients in $(K_n\otimes_{\Q}S)^\Gamma=K\otimes_{\Q}S$.

\begin{defn}\label{def:Sen-operator}
Let $X$ be a rigid analytic space over $\Q$, and let $V_X$ be a locally free coherent $\OO_X$-module equipped with a continuous $\OO_X$-linear $G_K$-action. We choose an admissible covering of $X$ by affinoid subdomains $\{M(S_i)\}_{i\in I}$ such that $V_{S_i}$ is free for each $i\in I$. We glue the operators $\Theta\in \mathrm{End}_{K_\infty\otimes_{\Q}S_i}(\mathscr{D}_{\Sen}(V_{S_i}))$ for all $i\in I$ to obtain an operator  $\Theta\in\mathrm{End}_{K_\infty\otimes_{\Q}\OO_X}(\mathscr{D}_{\Sen}(V_X))$; this is independent of the choice of the covering.  We call $\Theta$ the \emph{Sen operator} for $V_X$. 

We also glue the characteristic polynomials of $\Theta\in \mathrm{End}_{K_\infty\otimes_{\Q}S_i}(\mathscr{D}_{\Sen}(V_{S_i}))$ for all $i\in I$ to get an element of $(K\otimes_{\Q}\OO(X))[T]$; this is independent of the choice of the covering, and it is called the \emph{Sen polynomial} for $V_X$. 

\end{defn}

\begin{remark}
By their constructions and base change property of $\D_{\mathrm{Sen}}$, it is clear that the notions of Sen operator and Sen polynomial compatible with base change. That is, given a morphism $f:X'\ra X$ of $\Q$-rigid analytic spaces, the Sen operator and Sen polynomial of $f^*V_X$ are naturally isomorphic to the pullback of the 
Sen operator and Sen polynomial of $V_X$ via $f$ respectively.
\end{remark}
The rest of this subsection is a $\m$-module theoretical interpretation of \cite[(2.3)-(2.6)]{Ki03}.

\begin{prop}\label{prop:Theta-kill}
Let $V_S$ be a finite free $S$-linear representation. Then for any $n\geq n(V_S)$, both $H^0(\Gamma, \D^n_{\mathrm{Sen}}(V_S))$ and $H^1(\Gamma, \D^n_{\mathrm{Sen}}(V_S))$ are killed by $\det(\Theta)$.
\end{prop}
\begin{proof}
Let $L, \gamma$ be as above. Since $H^0(\Gamma_{L}, \D^{n}_{\mathrm{Sen}}(V_S))$ and $H^1(\Gamma_{L}, \D^{n}_{\mathrm{Sen}}(V_S))$ are computed by the complex 
\[
0\ra\D^{n}_{\mathrm{Sen}}(V_S)\stackrel{\gamma-1}\rightarrow\D^{n}_{\mathrm{Sen}}(V_S)\ra0,
\]
both of them are killed by $\gamma-1$. Thus both of them are killed by $\Theta$; hence both of them are killed by $\det(\Theta)$. This yields the desired result since $H^0(\Gamma, \D^{n}_{\mathrm{Sen}}(V_S))\subseteq H^0(\Gamma_L, \D^{n}_{\mathrm{Sen}}(V_S))$ and $H^1(\Gamma, \D^{n}_{\mathrm{Sen}}(V_S))$ is a quotient of $H^1(\Gamma_L, \D^{n}_{\mathrm{Sen}}(V_S))$.
\end{proof}

From now on, let $V_S$ be only locally free over $S$.

\begin{cor}\label{cor:t-descent}
For any $k\geq1$ and $n\geq n(V_S)$,
the natural map
\[
(\mathrm{D}_{\mathrm{dif}}^{+,n}(V_S)/(t^{k}))^{\Gamma}\ra(\mathrm{D}^n_{\mathrm{Sen}}(V_S))^{\Gamma}
\]
has kernel and cokernel killed by $\prod_{i=1}^{k-1}\det(\Theta+iI)$.
\end{cor}
\begin{proof}
Since $(\mathscr{D}_{\mathrm{dif}}^{+,n}(V_S)/(t^{k}))^{\Gamma}$ and $(\mathscr{D}^n_{\mathrm{Sen}}(V_S))^{\Gamma}$
 are coherent sheaves, by restricting on a finite covering of $M(S)$, it suffices to treat the case that $V_S$ is free over $S$. It then suffices to show that
the natural map
\[
(\mathrm{D}_{\mathrm{dif}}^{+,n}(V_S)/(t^{i+1}))^{\Gamma}\ra(\mathrm{D}_{\mathrm{dif}}^{+,n}(V_S)/(t^i))^{\Gamma}
\]
has kernel and cokernel killed by $\det(\Theta+iI)$ for each $i\geq1$. By the short exact sequence
\[
0\longrightarrow\D^n_{\Sen}(V_S(i))\longrightarrow\mathrm{D}_{\mathrm{dif}}^{+,n}(V_S)/(t^{i+1})\longrightarrow
\mathrm{D}_{\mathrm{dif}}^{+,n}(V_S)/(t^{i})\longrightarrow0,
\]
we get the exact sequence
\[
0\ra(\D^n_{\Sen}(V_S(i)))^{\Gamma}\ra(\mathrm{D}_{\mathrm{dif}}^{+,n}(V_S)/(t^{i+1}))^{\Gamma}
\ra(\mathrm{D}_{\mathrm{dif}}^{+,n}(V_S)/(t^i))^{\Gamma}
\ra H^1(\Gamma,\D^n_{\Sen}(V_S(i))).
\]
We thus conclude from Proposition \ref{prop:Theta-kill} and the fact that Sen operator for $V_S(i)$ is $\Theta+iI$.
\end{proof}

\begin{prop}\label{prop:Theta-gamma}
Keep notations as above. Then there exists a finite Galois extension $L'$ of $K$ containing $L$ such that $\Theta/(\gamma'-1)$ is invertible on $\D^{n}_{\Sen}(V_S)$ for any $\gamma'\in\Gamma_{L'}$.
\end{prop}
\begin{proof}
It suffices to treat the case that $V_S$ is free. Let $\gamma\in \Gamma_L$ such that $\chi(\gamma)\in 1+p^n\mathbb{Z}_p$. It follows that $\gamma$ acts $L_n\otimes_{\Q}S$-linearly on $\D^{L_n}_{\Sen}(V_S)$. Thus for any positive integer $k$, the matrix $M_{\gamma^{p^k}}$ of $\gamma^{p^k}$ is just $M_\gamma^{p^k}$. Therefore we may choose a sufficiently large $k$ so that $|M_{\gamma^{p^k}}-1|<p^{-1}$. Let $L'$ be a finite Galois extension of $K$ so that $\Gamma_{L'}\subseteq\langle\gamma^{p^k}\rangle$. Then for any $\gamma'\in \Gamma_{L'}$, we have $|M_{\gamma'}-1|\leq |M_{\gamma^{p^k}}-1|<p^{-1}$. It follows that 
\[
|(M_{\gamma'}-1)^m/(m+1)|<\frac{p^{-m}}{|m+1|}\leq p^{-1}
\] 
for any $m\geq1$. Let $u=\sum_{m\geq0}(1-\gamma')^m/(m+1)$. It follows that the matrix of $u-1$, which is 
\[
\sum_{m\geq1}(1-M_{\gamma'})^m/(m+1),
\]
has positive valuation. This yields that $u$ is invertible. Hence
$\Theta/(\gamma'-1)=\chi(\gamma')^{-1}u$
is invertible.
\end{proof}

In the following, we further suppose $\det(\Theta)=0$, and write $\det(TI-\Theta)=TQ(T)$ for some $Q(T)\in (K\otimes_{\Q}S)[T]$. Put $P(i)=\prod_{j=0}^{i-1}Q(-j)$ for every integer $i\geq1$.

\begin{prop}\label{prop:Gamma-base-change}
If $f:S\ra R$ is a map of affinoid algebras over $\Q$, for each $n\geq n(V_S)$, the natural map
\begin{equation}
(\D^n_{\mathrm{Sen}}(V_S))^{\Gamma}\otimes_SR\ra(\D^n_{\mathrm{Sen}}(V_R))^{\Gamma}
\end{equation}
has kernel and cokernel killed by a power of $f(Q(0))$. In particular, if $f(Q(0))$ is a unit, this map is an isomorphism.
\end{prop}
\begin{proof}
Write $Q(T)=\sum_{i=0}^{d-1}a_iT^i$. First note that $Q(\Theta)\Theta=0$ in $\mathrm{End}(\D_{\mathrm{Sen}}(V_S))$ by Cayley's theorem.
Hence
\[
\Theta(\D^n_{\mathrm{Sen}}(V_S))\subseteq \ker(Q(\Theta)|\D^n_\Sen(V_S))
\quad\text{and}\quad Q(\Theta)(\D^n_{\mathrm{Sen}}(V_S))\subseteq \ker(\Theta|\D^n_\Sen(V_S)).
\]
By the equality $a_0=Q(\Theta)-\Theta(\sum_{i=1}^{d-1}a_i\Theta^{i-1})$, we deduce that both the kernel and cokernel of the natural map
\[
\ker(\Theta|\D^n_{\Sen}(V_S))\oplus\ker(Q(\Theta)|\D^n_{\Sen}(V_S))\longrightarrow \D^n_{\mathrm{Sen}}(V_S)
\]
are killed by $a_0=Q(0)$. Hence the natural map
\begin{equation}\label{eq:ker+cokernel}
(\ker(\Theta|\D^n_{\Sen}(V_S)))_{a_0}\oplus(\ker(Q(\Theta)|\D^n_{\Sen}(V_S)))_{a_0}\longrightarrow (\D^n_{\mathrm{Sen}}(V_S))_{a_0}
\end{equation}
is an isomorphism. By the same reasoning, the natural map
\[
(\ker(\Theta|\D^n_{\Sen}(V_R)))_{f(a_0)}\oplus(\ker(Q(\Theta)|\D^n_{\Sen}(V_R)))_{f(a_0)}\longrightarrow (\D^n_{\mathrm{Sen}}(V_R))_{f(a_0)}
\]
is also an isomorphism. Consider the following commutative diagram
\[
\xymatrix{
\ar[d]\ker(\Theta|\D^n_{\Sen}(V_S))\otimes_SR_{f(a_0)}\oplus \ker(Q(\Theta)|\D^n_{\Sen}(V_S))\otimes_SR_{f(a_0)}
\ar[r] & \D^n_{\mathrm{Sen}}(V_S)\otimes_SR_{f(a_0)}\ar[d] \\
 (\ker(\Theta|\D^n_{\Sen}(V_R)))_{f(a_0)}
\oplus(\ker(Q(\Theta)|\D^n_{\Sen}(V_R)))_{f(a_0)}\ar[r] & (\D^n_{\mathrm{Sen}}(V_R))_{f(a_0)}
}
\]
where the upper map, which is obtained by tensoring up (\ref{eq:ker+cokernel}) with $R$ over $S$, is an isomorphism. Note that the right map is an isomorphism because $\D^n_\Sen(\cdot)$ is functorial in $V_S$. We thus deduce that both the natural maps
\[
\ker(\Theta|\D^n_{\Sen}(V_S))\otimes_SR_{f(a_0)}\ra(\ker(\Theta|\D^n_{\Sen}(V_R)))_{f(a_0)}
\]
and
\[
\ker(Q(\Theta)|\D^n_{\Sen}(V_S))\otimes_SR_{f(a_0)}\ra(\ker(Q(\Theta)|\D^n_{\Sen}(V_R)))_{f(a_0)}
\]
are isomorphisms.
Let $L'$ be a finite Galois extension of $K$ given by Proposition \ref{prop:Theta-gamma}. It then follows from Proposition \ref{prop:Theta-gamma} that 
\[
\ker(\Theta|\D^{n}_{\Sen}(V_S))=(\D^{n}_{\Sen}(V_S))^{\Gamma_{L'_n}}
\]
and 
\[
\ker(\Theta|\D^{n}_{\Sen}(V_R))=(\D^{n}_{\Sen}(V_R))^{\Gamma_{L'_n}}.
\] 
Note that $(\D^{n}_{\Sen}(V_S))^{\Gamma}$ (resp. $(\D^{n}_{\Sen}(V_R))^{\Gamma}$) is the image of the endomorphism 
\[
a\mapsto \sum_{g\in \Gamma/\Gamma_{L'_n}}ga
\] 
on $(\D^{n}_{\Sen}(V_S))^{\Gamma_{L'_n}}$ (resp. $(\D^{n}_{\Sen}(V_R))^{\Gamma_{L'_n}}$).
We therefore conclude immediately that the natural map
\[
(\D^n_{\Sen}(V_S))^{\Gamma_{K}}\otimes_SR_{f(a_0)}\longrightarrow(\D^n_{\Sen}(V_R))^{\Gamma_{K}}_{f(a_0)}
\]
is an isomorphism.
\end{proof}

\begin{cor}\label{cor:base-change}
If $f:S\ra R$ is a map of affinoid algebras over $\Q$, for each $n\geq n(V_S)$, the natural map
\[
(\mathrm{D}_{\mathrm{dif}}^{+,n}(V_S)/(t^k))^{\Gamma}\otimes_SR\ra
(\mathrm{D}_{\mathrm{dif}}^{+,n}(V_R)/(t^k))^{\Gamma}
\]
has kernel and cokernel killed by a power of $f(P(k))$. In particular, if $f(P(k))$ is a unit, this map is an isomorphism.
\end{cor}
\begin{proof}
Consider the following commutative diagram
\[
\xymatrix{
\ar[d](\mathrm{D}_{\mathrm{dif}}^{+,n}(V_S)/(t^k))^{\Gamma}\otimes_SR_{f(P(k))}
\ar[r]&(\mathrm{D}_{\mathrm{dif}}^{+,n}(V_R)/(t^k))^{\Gamma}_{f(P(k))}
\ar[d]\\
(\mathrm{D}^n_{\Sen}(V_S))^{\Gamma}\otimes_SR_{f(P(k))}\ar[r]&(\mathrm{D}^n_{\Sen}(V_R))_{f(P(k))}^{\Gamma}.
}
\]
The bottom map is an isomorphism by Proposition \ref{prop:Gamma-base-change}. The left map and right map are isomorphisms by Corollary \ref{cor:t-descent}. Hence the upper map is an isomorphism; this yields the desired result.
\end{proof}



\section{The extended Robba ring}
\subsection{Definitions}
Let $B$ be a $\Q$-Banach algebra with $|B|$ discrete. Set $v(x)=-\log_p(|x|)$ for any $x\in B$.
\begin{defn}\label{def:Robba}
For any interval
$I\subseteq(0,\infty)$, let $\r_B^{I}$ be the ring of Laurent series
\[
f=\sum_{i\in\mathbb{Z}}a_i T^i
\]
for which $a_i\in B$
and $v(a_i)+si\ra\infty$ as $i\ra\pm\infty$ for all $s\in I$.
For any $s\in I$,
define $w_s:\r_B^I\ra\mathbb{R}$ as
$$w_s(f)=\min_{i\in\mathbb{Z}}\{v(a_i)+si\}$$
and the norm $|\cdot|_s$ on $\r_B^I$ as
\[
|f|_s=\max_{i\in\mathbb{Z}}\{|a_i|p^{-si}\}=p^{-w_s(f)}.
\]
We denote $\r_B^{(0,r]}$ by $\r_B^r$ for simplicity. Let $\r_B^{\bd,r}$ be the subring of $\r_B^r$
consisting of elements with $\{v(a_i)\}_{i\in\mathbb{Z}}$ bounded below. Define $w:\r^{\bd,r}_B\ra\mathbb{R}$ as
\[
w(f)=\min_{i\in\mathbb{Z}}\{v(a_i)\}.
\]
Let $\r_B^{\int,r}$ be the subring of $\r_S^{\bd,r}$ consisting of $f$ with $w(f)\geq0$. We call $\r_B=\cup_{r>0}\r_B^r$ the \emph{Robba ring over $B$}, and call $\r^{\bd}_B=\cup_{r>0}\r_B^r$ the \emph{bounded Robba ring over $B$}.
\end{defn}

\begin{defn}\label{def:extended-Robba}
For any interval $I\subseteq (0,\infty)$, let $\widetilde{\r}_B^I$ be the set of formal sums
\[
f=\sum_{i\in\mathbb{Q}}a_iu^i
\]
with $a_i\in B$ satisfying the following conditions.
\begin{enumerate}
\item[(1)]For any $c>0$, the set of $i\in\mathbb{Q}$ so that $|a_i|\geq c$ is well-ordered (i.e. has no infinite decreasing subsequence).
\item[(2)]For all $s\in I$, $v(a_i)+si\ra\infty$ as $i\ra\pm\infty$, and $\inf_{i\in\mathbb{Q}}\{v(a_i)+si\}>-\infty$.
\end{enumerate}
These series form a ring under formal series addition and multiplication. For any $s\in I$, set $w_s:\widetilde{\r}^I_B\ra\mathbb{R}$ as
\[
w_s(f)=\inf_{i\in\mathbb{Q}}\{v(a_i)+si\}
\]
and the norm $|f|_s$ on $\widetilde{\r}_B^I$ as
\[
|f|_s=\sup_{i\in\mathbb{Q}}\{|a_i|p^{-si}\}=p^{-w_s(f)}.
\]
We denote $\widetilde{\r}_B^{(0,r]}$ by $\widetilde{\r}_B^r$ for simplicity.
Let $\widetilde{\r}_B^{\bd,r}$ be the subring of $\widetilde{\r}_B^r$
consisting of elements $f$ with $\{v(a_i)\}_{i\in\mathbb{Q}}$ bounded below. Define $w:\widetilde{\r}^{\bd,r}_B\ra\mathbb{R}$ as
\[
w(f)=\min_{i\in\mathbb{Q}}\{v(a_i)\}.
\]
 We call $\widetilde{\r}_B=\cup_{r>0}\widetilde{\r}_B^r$ the \emph{extended Robba ring over $B$},
and call $\widetilde{\r}^{\bd}_B=\cup_{r>0}\widetilde{\r}_B^r$ the \emph{extended bounded Robba ring over $B$}.
\end{defn}

We refer the reader to \S2.3 for more discussion of the extended Robba ring. 

\begin{remark}
Since $|B|$ is discrete, it follows from condition $(1)$ that $\inf_{i\in\mathbb{Q}}\{v(a_i)+si\}$ (hence also $\sup_{i\in\mathbb{Q}}\{|a_i|p^{-si}\}$) is attained at some $i\in\mathbb{Q}$.
\end{remark}
We equip $\r_B^I$ (resp. $\widetilde{\r}_B^I$) with the Fr\'echet topology defined by $\{w_s\}_{s\in I}$; then $\r_B^I$ (resp. $\widetilde{\r}_B^I$) is a complete Fr\'echet algebra over $\Q$. Furthermore, in the case that $I=[a,b]$ is a closed interval, $\r_B^I$ (resp. $\widetilde{\r}_B^I$) is a Banach algebra over $\Q$ with the norm $\max\{w_a,w_b\}$. We equip $\r_B^{\bd,r}$ (resp. $\widetilde{\r}_B^{\bd,r}$) with the norm $\max\{w,w_r\}$; then $\r_B^{\bd,r}$ (resp. $\widetilde{\r}_B^{\bd,r}$) is a Banach algebra over $\Q$.

\begin{defn}\label{def:extended-calE}
Let $\widetilde{\calE}_B$ be the ring of
formal sums $f=\sum_{i\in\mathbb{Q}}a_iu^i$ with
$a_i\in B$ satisfying the following conditions.
\begin{enumerate}
\item[(1)] For each $c>0$, the set
of $i\in\mathbb{Q}$ such that $|a_i|\geq c$ is well-ordered.
\item[(2)]The set $\{v(a_i)\}_{i\in\mathbb{Q}}$ is bounded below and
$v(a_i)\ra\infty$ as $i\ra-\infty$.
\end{enumerate}
Set $w:\widetilde{\calE}_B\ra\mathbb{R}$ as
\[
w(f)=\min_{i\in\mathbb{Q}}\{v(a_i)\}.
\]
\end{defn}
We equip $\widetilde{\calE}_B$ with the topology defined by $w$; then $\widetilde{\calE}_B$ is complete for this topoology.

In the following, let $L$ be a $p$-adic field in the sense that it is a complete discretely valued field  equipped with the structure of an $\Q$-algebra in such a way that the map $\Q\ra L$ is continuous. Put $B_L=L\widehat{\otimes}_{\Q}B$.

\begin{prop}\label{prop:comparison}
For $R\in\{\r^{\bd,r},\r^I\}$ or  $R=\r^r$,  where $I\subset(0,\infty)$ is a closed interval, the natural map
$R_L\otimes_{\Q}B\rightarrow R_{B_L}$
induces an isometric isomorphism 
\[
R_L\widehat{\otimes}_{\Q}B\cong R_{B_L}
\] 
of $L$-Banach algebras or Fr\'echet algebras. For $\widetilde{R}\in \{\widetilde{\calE},\widetilde{\r}^{\bd,r},\widetilde{\r}^I\}$ or $\widetilde{R}=\widetilde{\r}^{r}$, the natural map $\widetilde{R}_L\otimes_{\Q}B\rightarrow\widetilde{R}_{B_L}$ induces an isometric embedding 
\[
\widetilde{R}_L\widehat{\otimes}_{\Q}B\hookrightarrow \widetilde{R}_{B_L}
\]
of $L$-Banach algebras or Fr\'echet algebras.
\end{prop}
\begin{proof}
This follows from \cite[Lemma 2.1.6]{L08}.
\end{proof}

\begin{prop}\label{prop:intersection}
If $B$ is of countable type, then
\[
(\widetilde{\calE}_{L}\widehat{\otimes}_{\Q}B)\cap\widetilde{\r}^{\bd,r}_{B_L}=\widetilde{\r}^{\bd,r}_L\widehat{\otimes}_{\Q}B.
\]
\end{prop}
\begin{proof}
This follows from \cite[Lemma 2.1.8]{L08} by taking $S=\widetilde{\calE}_{L}$.
\end{proof}

\begin{lemma}\label{lem:projection}
Let $S$ be an affinoid algebra over $\Q$. Then for any $x\in M(S)$,
the natural map $\r^r_S\otimes_Sk(x)\ra\r^r_{k(x)}$ is an isomorphism.
\end{lemma}
\begin{proof}
It reduces to show that the natural map $\rho_x:\r_S^r\ra\r^r_{k(x)}$ is surjective and its kernel is $\mathfrak{m}_x\r_S^r$. By \cite[Proposition A.2.2]{Bel13}, the exact sequence
\[
0\longrightarrow\mathfrak{m}_x\longrightarrow S\longrightarrow k(x)\longrightarrow0
\]
induces the exact sequence
\[
0\longrightarrow\r^r_{\Q}\widehat{\otimes}_{\Q}\mathfrak{m}_x\longrightarrow \r_{\Q}^r\widehat{\otimes}_{\Q}S
\longrightarrow \r_{\Q}^r\widehat{\otimes}_{\Q}k(x)\longrightarrow0.
\]
Using Proposition \ref{prop:comparison}, we get that $\rho_x$ is surjective. Choose a finite set of generators $b_1,\dots,b_m$ of $\mathfrak{m}_x$. By the open mapping theorem for Banach spaces over discretely valued fields, the surjective map of $\Q$-Banach spaces $S^m\rightarrow \mathfrak{m}_x$ defined by $(a_1,\dots,a_m)\mapsto\sum_{i=1}^ma_ib_i$ is open. Hence there exists $c>0$ such that for any $a\in\mathfrak{m}_x$, there exist $a_1,\dots,a_m\in S$ with $|a_i|\leq c|a|$ such that $a=\sum_{i=1}^ma_ib_i$. Now let $f=\sum_{i\in\mathbb{Q}}a_iu^i$ belongs to kernel of $\rho_x$; so $a_i\in\mathfrak{m}_x$ for all $i$. For each $i\in\mathbb{Q}$, choose $a_{ij}\in S$ with $|a_{ij}|\leq c|a_i|$ for $1\leq j\leq m$  such that $a_i=\sum_{j=1}^ma_{ij}b_j$. Let $f_j=\sum_{i\in\mathbb{Q}}a_{ij}u^i$ for $1\leq j\leq m$. It is then clear that $f_j\in\r_S^r$ and $f=\sum_{j=1}^mb_jf_j$; hence $f\in \mathfrak{m}_x\r_S^r$.
\end{proof}

\subsection{Key lemma}
From now on, suppose that $L$ is equipped with an isometric automorphism $\varphi_L$ such that its restriction on $\Q$ is the identity. Let $S$ be an affinoid algebra over $\Q$, and let $\varphi$ be the continuous extension of $\varphi_L\otimes{\mathrm{id}}$ to $S_L$. We fix a positive integer $q>1$, and we extend $\varphi$ to automorphisms on $\widetilde{\r}_{S_L}$ and $\widetilde{\calE}_{S_L}$ by setting
\[
\varphi(\sum_{i\in\mathbb{Q}}a_iu^i)=\sum_{i\in\mathbb{Q}}\varphi(a_i)u^{qi}.
\]
It is obvious that $\varphi$ restricts to automorphisms on $\widetilde{\r}_L\widehat{\otimes}_{\Q}S$ and $\widetilde{\calE}_L\widehat{\otimes}_{\Q}S$.

Let $\alpha\in S^\times$. Consider the following Frobenius equation
\begin{equation}\label{eq:Frobenius-equation-family}
\varphi(b)-\alpha b=a.
\end{equation}
The following is a variant of \cite[Lemma 2.3.5(3)]{L08}.
\begin{lemma}\label{lemma:Frobenius-equation}
Suppose $|\alpha^{-1}|<1$. Then for $a=\sum_{i\in\mathbb{Q}}a_iu^i\in\widetilde{\r}^{r}_{S_L}$, the following are true.
\begin{enumerate}
\item[(1)](\ref{eq:Frobenius-equation-family}) admits at most one solution $b\in\widetilde{\r}_{S_L}$.
\item[(2)](\ref{eq:Frobenius-equation-family}) has a solution $b\in\widetilde{\r}_{S_L}$ if and only if
\begin{equation}\label{eq:criterion}
\sum_{m\in\mathbb{Z}}\alpha^{-(m+1)}\varphi^m(a_{iq^{-m}})=0
\end{equation}
for all $i<0$. Furthermore, in this case the unique solution $b$ is given by
\[
b=-\sum_{i\in\mathbb{Q}}(\sum_{m\in\mathbb{N}} \alpha^{-(m+1)}\varphi^m(a_{iq^{-m}}))u^i
\]
and belongs to $\widetilde{\r}^{qr}_{S_L}$, and it satisfies $w_r(b)\geq w_r(a)-C(r,\alpha)$ where $C(r,\alpha)$ is some constant only depending on $r,\alpha$.
\item[(3)]
Suppose $a\in \widetilde{\r}^r_L\widehat{\otimes}_{\Q}S$. If $b\in\widetilde{\r}_{S_L}$ is a solution of  (\ref{lemma:Frobenius-equation}), then $b\in \widetilde{\r}^r_L\widehat{\otimes}_{\Q}S$.
\end{enumerate}
\end{lemma}

\begin{proof}
Suppose that $b=\sum_{i\in\mathbb{Q}}b_iu^i\in \widetilde{\r}^{r'}_{S_L}$ is a solution of (\ref{eq:Frobenius-equation-family}). By comparing coefficients, we get
\[
\varphi(b_{i/q})-\alpha b_i=a_i,
\]
yielding
\begin{equation}\label{eq:iteration}
b_i=\alpha^{-1}\varphi(b_{i/q})-\alpha^{-1}a_i
\end{equation}
for every $i\in\mathbb{Q}$. Since $|\alpha^{-1}|_{\mathrm{sp}}<1$ and $\{|a_{iq^{-m}}|\}_{m\in\mathbb{N}}$ are bounded, we get
\[
b_i=-\sum_{m\in\mathbb{N}} \alpha^{-(m+1)}\varphi^m(a_{iq^{-m}})
\]
by iterating (\ref{eq:iteration}). Thus $b_i$ is uniquely determined by $\alpha$ and $a$. This proves (1). Furthermore, for any $k\in\mathbb{N}$,
\[
\varphi^k(\sum_{m=-k}^\infty\alpha^{-(m+1)}\varphi^m(a_{iq^{-m}}))
=\alpha^{k}\sum_{m=0}^\infty\alpha^{-(m+1)}\varphi^m(a_{(iq^k)p^{-m}})=-\alpha^{k}b_{iq^k}.
\]
Hence
\[
v(\sum_{m=-k}^\infty\alpha^{-(m+1)}\varphi^m(a_{iq^{-m}}))=v(\alpha^{k}b_{iq^k})\geq kv(\alpha)+v(b_{iq^k})\geq kv(\alpha)+w_{r'}(b)-r'iq^k.
\]
It follows that if $i<0$, then $v(\sum_{m=-k}^\infty \alpha^{-(m+1)}\varphi^m(a_{iq^{-m}}))\ra \infty$ as $k\ra\infty$; this yields (\ref{eq:criterion}), proving the ``only if'' part of (2).

To prove the ``if'' part of (2), for $f=\sum_{i\in\mathbb{Q}}a_iu^i\in\widetilde{\r}^r_{S_L}$ and $c\in\mathbb{R}$,
we set
\[
w_r^{c,-}(f)=\min_{i\leq c}\{v(a_i)+ri\}.
\]
It is clear that $w_r^{c,-}(f)\ra\infty$ as $c\ra-\infty$. Now suppose that (\ref{eq:criterion}) holds for all $i<0$. If $i\leq-1$, then for each $m\leq-1$,
\begin{equation*}
\begin{split}
v(\alpha^{-(m+1)}\varphi^m(a_{iq^{-m}}))&\geq v(a_{iq^{-m}})-(m+1)v(\alpha)=(v(a_{iq^{-m}})+riq^{-m})-riq^{-m}-(m+1)v(\alpha)\\
&\geq(w_r^{i,-}(a)-ri)+ri(1-q^{-m})-(m+1)v(\alpha)\\
&\geq (w_r^{i,-}(a)-ri)+r(q^{-m}-1)-(m+1)v(\alpha)\\
&\geq w_r^{i,-}(a)-ri-C_1(r,\alpha),
\end{split}
\end{equation*}
where $C_1(r,\alpha)$ is some constant depending on $r,\alpha$. Hence
\begin{equation}\label{eq:i<-1}
\begin{split}
w_r((\sum_{m=0}^\infty \alpha^{-(m+1)}\varphi^m(a_{iq^{-m}}))u^i)&=v(-\sum_{m=-1}^{-\infty}\alpha^{-(m+1)}\varphi^m(a_{iq^{-m}}))+ri\\
&\geq  w_r^{i,-}(a)-C_1(r,\alpha)
\end{split}
\end{equation}
for each $i\leq-1$. Note that $v(\alpha^{-1})>0$ by assumption. Thus if $-1<i<0$, for any $m\geq0$,
\begin{equation}
\begin{split}
v(\alpha^{-(m+1)}\varphi^m(a_{iq^{-m}}))&\geq w_r(a)-riq^{-m}+v(\alpha^{-1})(m+1)\\
&>w_r(a)-riq^{-m}\\
&>w_r(a)-ri.\\
\end{split}
\end{equation}
Hence
\begin{equation}\label{eq:i>-1}
w_r((\sum_{m=0}^\infty \alpha^{-(m+1)}\varphi^m(a_{iq^{-m}}))u^i)\geq w_r(a)
\end{equation}
for all $i>-1$. 

Now put
\[
a^+=\sum_{i\geq0}a_iu^i, \qquad a^-=\sum_{i<0}a_iu^i
\]
and
\[
b^+=-\sum_{i\geq0}(\sum_{m=0}^\infty\alpha^{-(m+1)}\varphi^m(a_{iq^{-m}}))u^i, \quad
b^-=-\sum_{i< 0}(\sum_{m=0}^\infty\alpha^{-(m+1)}\varphi^m(a_{iq^{-m}}))u^i.
\]
Since $|\alpha^{-1}|<1$, it is straightforward to see that the series $\sum_{m=0}^\infty\alpha^{-(m+1)}\varphi^m(a^+)$ is convergent in $\widetilde{\r}^{r}_{S_L}$, and
\[
w_r(\sum_{m=0}^\infty\alpha^{-(m+1)}\varphi^m(a^+))\geq w_r(a^+)\geq w_r(a).
\]
We then deduce 
\[
\sum_{m=0}^\infty\alpha^{-(m+1)}\varphi^m(a^+)=-b^+
\] 
by comparing the coefficients. We claim that $b^-$ also belongs to $\widetilde{\r}_{S_L}^r$. We first deduce from (\ref{eq:i<-1}) and (\ref{eq:i>-1}) that $b^-$ satisfies Definition \ref{def:extended-Robba}(2). On the other hand, since $|\alpha^{-1}|<1$, the series $\sum_{m=0}^\infty\alpha^{-(m+1)}\varphi^m(a^-)$ is convergent to $-b^{-}$ in $\widetilde{\calE}_{S_L}$. Hence $b^-$ also satisfies Definition \ref{def:extended-Robba}(1), yielding the claim. 

Now put $b=b^++b^-\in\widetilde{\r}_{S_L}^r$. It is then clear that $b$ is the solution of (\ref{eq:Frobenius-equation-family}). By (\ref{eq:i<-1}) and (\ref{eq:i>-1}), we get 
\[
w_r(b^-)\geq w_r(a)-C(r,\alpha).
\]
where $C(r,\alpha)=\max\{0, C_1(r,\alpha)\}$. It then follows $w_r(b)\geq w_r(a)-C(r,\alpha)$. Furthermore, since $\varphi(b)=a-\alpha b\in \widetilde{\r}^{r}_{S_L}$, we get $b\in \widetilde{\r}^{qr}_{S_L}$.

It remains to prove (3). If $a\in \widetilde{\r}^r_L\widehat{\otimes}_{\Q}S$, then $a^+\in\widetilde{\r}^r_L\widehat{\otimes}_{\Q}S$ and $a^-\in \widetilde{\r}^{\mathrm{bd},r}_L\widehat{\otimes}_{\Q}S$. We thus deduce $b^+\in \widetilde{\r}_L^r\widehat{\otimes}_{\Q}S$ and $b^-\in \widetilde{\calE}_L\widehat{\otimes}_{\Q}S$. Since $b^-\in\widetilde{\r}_L^{\bd,r}\widehat{\otimes}_{\Q}S$, by Proposition \ref{prop:intersection}, we conclude
that
\[
b^-\in (\widetilde{\calE}_L\widehat{\otimes}_{\Q}S)\cap \widetilde{\r}^{\bd,r}_{S_L}=\widetilde{\r}_L^{\bd,r}\widehat{\otimes}_{\Q}S.
\]
Hence $b=b^++b^-\in \widetilde{\r}^r_L\widehat{\otimes}_{\Q}S$.
\end{proof}
\begin{remark}
One can reformulate the above lemma using the notion of cohomology of $\varphi$-modules. For any $\alpha\in S^\times$, we define the rank 1 $\varphi$-module $\widetilde{\r}_{S_L}(\alpha)$ over $\widetilde{\r}_{S_L}$ by setting $\varphi(v)=\alpha^{-1}v$ for a generator $v$; we set $H^1(\widetilde{\r}_{S_L}(\alpha))=\widetilde{\r}_{S_L}(\alpha)/(\varphi-1)$. Then Lemma \ref{lemma:Frobenius-equation} says that if $|\alpha^{-1}|<1$, then $av$ is a coboundary if and only if $a$ satisfies (\ref{eq:criterion}).
\end{remark}
\subsection{Relations between different rings}
Recall that there exists a natural identification $\widetilde{\mathbf{B}}_\rig^\dag\cong\Gamma^{\mathrm{alg}}_{\mathrm{an}}$ which identifies $\widetilde{\mathbf{B}}_{\rig}^{\dagger,\rho(r)}$ with $\Gamma^{\mathrm{alg}}_{\mathrm{an},r}$ for any $r>0$ (see for instance \cite[\S1.1]{LB08}); here $\Gamma^{\mathrm{alg}}_{\mathrm{an}}$ and $\Gamma^{\mathrm{alg}}_{\mathrm{an},r}$ are relative extended Robba rings\footnote{See \cite[\S5]{KL15} for a uniform treatment about relative extended Robba rings associated to analytic fields of characteristic $p$.} associated to the residue field $\widehat{\mathbb{F}_p((u))^{\mathrm{alg}}}$\footnote{The completion is taken with respect to the $u$-adic topology.} introduced by Kedlaya (see \cite[\S2]{Ke05} for more details). On the other hand, $\widetilde{\r}_{\widehat{\Q^{\mathrm{ur}}}}$ and $\widetilde{\r}_{\widehat{\Q^{\mathrm{ur}}}}^r$ (together with the $\varphi$-action for $q=p$) are relative extended Robba rings with residue field $\mathbb{F}_p^{\mathrm{alg}}((u^{\mathbb{Q}}))$; here $\mathbb{F}_p^{\mathrm{alg}}((u^{\mathbb{Q}}))$ is the Hahn-Mal'cev-Neumann algebra with coefficients in $\mathbb{F}_p^{\mathrm{alg}}$ (see for instance \cite[Definition 4.5.4]{Ke05}). By \cite[Theorem 8]{Ke01}, $\widehat{\mathbb{F}_p((u))^{\mathrm{alg}}}$ is a closed subfield of $\mathbb{F}^{\mathrm{alg}}_p((u^{\mathbb{Q}}))$. This leads to natural closed embeddings of $\Qp$-Fr\'echet algebras
\[
\Gamma^{\mathrm{alg}}_{\mathrm{an,r}}\hookrightarrow\widetilde{\r}^r_{\widehat{\Q^{\mathrm{ur}}}}
\]
for all $r>0$. By the above identifications, we therefore get closed embeddings
\begin{equation}\label{eq:embedding}
\widetilde{\mathbf{B}}_{\rig}^{\dagger,\rho(r)}\hookrightarrow\widetilde{\r}^r_{\widehat{\Q^{\mathrm{ur}}}}
\end{equation}
which respect the $\varphi$-action. Henceforth we regard $\widetilde{\mathbf{B}}_{\rig}^{\dagger,\rho(r)}$ as a subring of $\widetilde{\r}^r_{\widehat{\Q^{\mathrm{ur}}}}$; we therefore regard $K\otimes_{K_0}\widetilde{\mathbf{B}}_{\rig}^{\dagger,\rho(r)}$ as a subring  of $\widetilde{\r}^r_{\widehat{K^{\mathrm{ur}}}}$

We will need the following results in \S3.

\begin{lemma}\label{lem:coefficient-vanishing}
For any $a\in S_L$, there exists an analytic subspace $M(S(a))$ of $M(S)$  such that for any map $g:S\ra R$ of affinoid algebras over $\Q$,  $g_L(a)=0$ if and only if the map $M(R)\ra M(S)$ factors through $M(S(a))$.
\end{lemma}
\begin{proof}
Choose an orthonormal basis $\{e_j\}_{j\in J}$ of $L$ over $\Q$; then it is also an orthonormal basis of $S_L$ as an $S$-Banach module. Let $I(a)$ be the ideal of $S$ generated by the coefficients of $a$. It is then clear that one can take $S(a)=S/I(a)$.
\end{proof}

\begin{lemma}\label{lem:solution-descent}
Let $a\in \widetilde{\r}^r_{\widehat{K^{\mathrm{ur}}}}\widehat{\otimes}_{\Q}S$. Then there exists an analytic subspace $M(S(a,r))$ of $M(S)$ such that for any map $g:S\ra R$ of $\Qp$-affinoid algebras,
$g(a)\in (K\otimes_{K_0}\mathbf{B}_{\rig,K}^{\dagger,\rho(r)})\widehat{\otimes}_{\Q}R$ if and only if the map $M(R)\ra M(S)$ factors through $M(S(a,r))$.
\end{lemma}
\begin{proof}
Note that $\mathbf{B}_{\rig,K}^{\dagger,\rho(r)}$ may be identified with the intersections of all $\mathbf{B}^{[\rho(r),\rho(s)]}_K$'s with $0<s<r$ (one way to see this fact is to use the identification of $\mathbf{B}_{\rig,K}^{\dagger,\rho(r)}$ with the Robba ring $\r^r_{K_0'}$). 
To prove the lemma, it then suffices to show that for any $0<s<r$, there exists an analytic subspace $M(S(a,s,r))$ of $M(S)$ such that for any map $g:S\ra R$ of $\Qp$-affinoid algebras,
$g(a)\in (K\otimes_{K_0}\mathbf{B}_{K}^{[\rho(r),\rho(s)]})\widehat{\otimes}_{\Q}R$ if and only if the map $M(R)\ra M(S)$ factors through $M(S(a,s,r))$; we then take $M(S(a,r))$ to be the intersections of all $M(S(a,s,r))$'s. 

Let $\widetilde{\mathbf{B}}^{[\rho(r),\rho(s)]}$ be the completion of $\widetilde{\mathbf{B}}_{\rig}^{\dagger,\rho(r)}$  with respect to $\max\{\mathrm{val}^{(0,\rho(r)]},\mathrm{val}^{(0,\rho(s)]}\}$. One may therefore identify $\mathbf{B}^{[\rho(r),\rho(s)]}_K$ as a closed subspace of $\widetilde{\mathbf{B}}^{[\rho(r),\rho(s)]}$. On the other hand, the closed embedding $\widetilde{\mathbf{B}}_{\rig}^{\dagger,\rho(r)}\hookrightarrow\widetilde{\r}^r_{\widehat{\Q^{\mathrm{ur}}}}$ of $\Qp$-Fr\'echet algebras induces a closed embedding $\widetilde{\mathbf{B}}^{[\rho(r),\rho(s)]}\hookrightarrow\widetilde{\r}^{[s,r]}_{\widehat{\Q^{\mathrm{ur}}}}$ of $\Qp$-Banach spaces. Therefore, by Hahn-Banach theorem for Banach spaces over discretely valued fields, we deduce that there exists a closed $\Qp$-subspace $V$ of $\widetilde{\r}^{[s,r]}_{\widehat{K^{\mathrm{ur}}}}$ so that
$\widetilde{\r}^{[s,r]}_{\widehat{K^{\mathrm{ur}}}}\cong K\otimes_{K_0}\mathbf{B}_K^{[\rho(r),\rho(s)]}\oplus V$. Hence we have
\[
\widetilde{\r}^{[s,r]}_{\widehat{K^{\mathrm{ur}}}}\widehat{\otimes}_{\Q}S
\cong (K\otimes_{K_0}\mathbf{B}_K^{[\rho(r),\rho(s)]})\widehat{\otimes}_{\Q}S\oplus V\widehat{\otimes}_{\Q}S
\]
and
\[
\widetilde{\r}^{[s,r]}_{\widehat{K^{\mathrm{ur}}}}\widehat{\otimes}_{\Q}R
\cong (K\otimes_{K_0}\mathbf{B}_K^{[\rho(r),\rho(s)]})\widehat{\otimes}_{\Q}R\oplus V\widehat{\otimes}_{\Q}R.
\]

Write $a=a_1+a_2$ with $a_1\in (K\otimes_{K_0}\mathbf{B}_{K}^{[\rho(r),\rho(s)]})\widehat{\otimes}_{\Q}S$ and $a_2\in V\widehat{\otimes}_{\Q}S$. It is then obvious that $g(a)\in (K\otimes_{K_0}\mathbf{B}_{K}^{[\rho(r),\rho(s)]})\widehat{\otimes}_{\Q}R$ if and only if $g(a_2)=0$. By Proposition \ref{prop:comparison}, we may regard $a_2$ as an element of $\tilde{\r}^{[s,r]}_{S_{\widehat{K^{\mathrm{ur}}}}}$; then $g(a_2)=0$ in $\widetilde{\r}_{\widehat{K^{\mathrm{ur}}}}^{[s,r]}\widehat{\otimes}_{\Q}R$ if and only if $g(a_2)=0$ in $\widetilde{\r}^{[s,r]}_{R_{\widehat{K^{\mathrm{ur}}}}}$. Write $a_2=\sum_{i\in\mathbb{Q}}c_iu^i$ with $c_i\in S_{\widehat{K^{\mathrm{ur}}}}$. Let
\[
I(a,s,r)=\sum_{i\in\mathbb{Q}}I(c_i)
\]
where $I(c_i)$ is the ideal defined in the proof of Lemma \ref{lem:coefficient-vanishing}. It is then clear that one can take $S(a,s,r)=S/I(a,s,r)$.
\end{proof}

\section{Construction of finite slope subspaces}


Throughout this section, let $X$, $V_X$ and $\alpha$ be as in \S0.1.  For any morphism  $X'\ra X$ of rigid analytic spaces over $\Q$, we denote by $V_{X'}$ the pullback of $V_X$ on $X'$ which is a locally free coherent $\OO_{X'}$-module of rank $d$ with a continuous $\OO_{X'}$-linear $G_K$-action. In the case when $X=M(S)$ is an affinoid space, we denote $V_X$ by $V_S$ instead. We have defined finite slope subspaces of $X$ with respect to $(\alpha,V_X)$ in Definition \ref{def:fs-space}. The goal of this section is to prove that $X$ has a unique finite slope subspace (which may well be empty).

\subsection{Prelude}

\begin{prop}\label{prop:flat-base-change}
The formation of $X_{fs}$ commutes with flat base change. Namely, if $h:X'\ra X$ is a flat morphism of separated and reduced rigid analytic spaces over $\Q$, and if $X_{fs}$ is a finite slope subspace of $X$ with respect to $(\alpha,V_X)$, then the base change $X_{fs}'$ of $X_{fs}$ via $h$ is a finite slope subspace of $X'$ with respect to $(h^*(\alpha), V_{X'})$.
\end{prop}
\begin{proof}
Note that the Sen polynomial for $V_{X'}$ is $Th^{*}(Q(T))$. By Definition \ref{def:fs-space}(1), we have that $Q(j)_{\tau}$ is a nonzero divisor in $X_{fs}$ for every integer $j\leq 0$ and $\tau\in\mathrm{H}_K$. The flatness of $h$ then implies that $h^*(Q(j)_\tau)$ is a nonzero divisor in $X'_{fs}$. Hence $X'_{fs}$ satisfies (1) of Definition \ref{def:fs-space}. Now let $g:M(R)\ra X'$ be a map of rigid spaces over $\Q$ which factors through $X'_{Q(j)_\tau}$ for every integer $j\leq0$ and $\tau\in\mathrm{H}_K$. Then $h\circ g$ factors through $X_{Q(j)_\tau}$ for every integer $j\leq0$ and $\tau\in\mathrm{H}_K$. By the universal property of $X_{fs}$, we know that for $n$ sufficiently large,
the natural map
\[
(K\otimes_{K_0}\mathrm{D}^{\dag}_{\rig}(V_R))^{\varphi^f=g^*(h^*(\alpha)),\Gamma=1}\ra (\mathrm{D}_{\dif}^{+,fn}(V_R))^{\Gamma}
\]
is an isomorphism if and only if $h\circ g$ factors through $X_{fs}$, i.e. if and only if $g$ factors through $X_{fs}'$. This implies that $X_{fs}'$ satisfies (2) of Definition \ref{def:fs-space}.
\end{proof}

\begin{prop}\label{prop:fs-unique}
There exists at most one finite slope subspace of $X$.
\end{prop}
\begin{proof}
Suppose that $X_1,X_2$ are two finite slope subspaces of $X$. Let $\{U_j\}_{j\in J}$ be an admissible affinoid covering of $X$ by affinoid subdomains. It suffices to show that for any $j\in J$, the restrictions of $X_1,X_2$ on $U_j$ coincide. By Proposition \ref{prop:flat-base-change}, we see that the restrictions of $X_1,X_2$ on $U_j$ are finite slope subspaces of $U_j$. Thus it reduces to the case that $X=M(S)$ is an affinoid space. We prove this by using Kisin's argument (\cite[(5.8)]{Ki03}). Let $I_1, I_2\subset S$ be the ideals corresponding to $X_1,X_2$ respectively. Let $W$ be the support of $(I_1+I_2)/I_1$ in $X_1$ (with its reduced structure). Let $x\in X_1$ be a closed point. If $x\in X_{Q(j)_\tau}$ for every integer $j\leq 0$ and $\tau\in \mathrm{H}_K$, applying (2) of Definition \ref{def:fs-space} to any finite length quotient $R$ of $\OO_{X_1,x}$, we get that $x\in X_2$ and $\widehat{\OO_{X_1,x}}=\widehat{\OO_{X_2,x}}$. This implies that $x\notin W$. Hence, for any $w\in W$ there exists integer $j\leq0$ and $\tau\in \mathrm{H}_{K}$ such that $Q(j)_\tau(w)=0$. If $W_0$ is an irreducible component of $W$, by  \cite[Lemma (5.7)]{Ki03}, we deduce that there exists $j_{W_0}\leq 0$ and $\tau_{W_0}\in\mathrm{H}_K$ such that $Q(j_{W_0})_{\tau_{W_0}}$ vanishes in $W_0$. It follows that
$X_1\backslash W$ contains $\cap_{W_0\subseteq W}(X_1)_{Q(j_{W_0})_{\tau_{W_0}}}$. The latter is Zariski open and dense in $X_1$ since $W$ has only finitely many components. A fortiori we see that $X_1\backslash W$, which is contained in $X_2$, is Zariski open and dense in $X_1$, yielding $X_1\subset X_2$. Thus $X_2=X_1$.
\end{proof}

\begin{remark}\label{rem:fs-artinian}
The proof of Proposition \ref{prop:fs-unique} actually implies that there exists at most one analytic subspace of $X$ which satisfies Definition \ref{def:fs-space}(1) and Definition \ref{def:fs-space}(2) for all finite $\Q$-algebras $R$.\end{remark}

\begin{prop}\label{prop:fs-glue}
Let $\{U_j\}_{j\in J}$ be an admissible covering of $X$ by affinoid subdomains. Suppose that each $U_j$ has the finite slope subspace $(U_j)_{fs}$. Then $\{(U_j)_{fs}\}_{j\in J}$ glues to form the finite slope subspace of $X$.
\end{prop}
\begin{proof}
By the uniqueness of finite slope subspaces, we see that $\{(U_j)_{fs}\}_{j\in J}$ glues to form an analytic subspace $X_{fs}$ of $X$. It is then clear that $X_{fs}$ satisfies (1) of Definition \ref{def:fs-space}. Now let $g:M(R)\ra X$ be a morphism of rigid analytic spaces over $\Q$ which factors through $X_{Q(j)}$ for each integer $j\leq0$. The pullback  $\{g^{-1}(U_j)\}$ forms an admissible covering of $M(R)$. We choose a finite covering $\{M(R_i)\}_{i\in I}$ of $M(R)$  by affinoid subdomains which refines $\{g^{-1}(U_j)\}$. It then follows that for each $i\in I$, the natural map 
\[
(K\otimes_{K_0}\D_\rig^{\dag}(V_{R_i}))^{\varphi^f=g^{*}(\alpha),\Gamma=1}\ra \D^{+,fn}_{\dif}(V_{R_i})^\Gamma
\] 
is an isomorphism for all sufficiently large $n$. We deduce from Propositions \ref{prop:sheaf-dif} and \ref{prop:sheaf-rig} that the natural map
\[
(K\otimes_{K_0}\D_\rig^{\dag}(V_{R}))^{\varphi^f=g^{*}(\alpha),\Gamma=1}\ra\D^{+,fn}_{\dif}(V_{R})^\Gamma
\] 
is an isomorphism for all sufficiently large $n$. This yields that $X_{fs}$ is the finite slope subspace of $X$.
\end{proof}

\subsection{Techniques}
We start by introducing some notations. For an affinoid algebra $S$, $\alpha\in S^\times$, and $a$ as in Lemma \ref{lemma:Frobenius-equation}, using Lemma \ref{lem:coefficient-vanishing}, 
we denote by $M(S(\alpha,a))$ the intersection of
\[
M(S(\sum_{m\in\mathbb{Z}}\alpha^{-(m+1)}\varphi^m(a_{iq^{-m}})))
\]
for all rational numbers $i<0$. From now on, let $V_S$ be a locally free $S$-linear representations of $G_K$ of rank $d$. 

\begin{prop}\label{prop:solution-extended-robba}
Let $\alpha\in S^\times$, and let $\beta\in(K\otimes_{K_0}\mathbf{B}^{\dag,s}_{K})^\times$ satisfying $|\beta|>|\alpha^{-1}|$ (here $|\beta|$ denotes the $p$-adic norm of $\beta$ in $K\otimes_{K_0}\mathbf{B}^{\dag,s}_{K}$). Then for any $a\in K\otimes_{K_0}\D^{\dagger,s}_\rig(V_S)$ there exists an $E$-analytic subspace $M(S(\alpha,\beta, a))$ of $M(S)$ such that for any morphism $g :S\ra R$ of affinoid algebras over $E$, the equation
\begin{equation}\label{eq:solution-extended-robba}
\varphi^f(b)-\beta g(\alpha)b=g(a)
\end{equation}
has a solution $b\in K\otimes_{K_0}\D^{\dagger,s}_\rig(V_R)$ if and only if the map $M(R)\ra M(S)$ factors through $M(S(\alpha,\beta, a))$. Furthermore, the solution $b$ is unique in this case.
\end{prop}
\begin{proof}
Granting the assertion of the proposition, it is then clear that the construction of  $M(S(\alpha,\beta, a))$ is compatible with base change. Thus it suffices to prove the proposition for each affinoid subdomain of an affinoid  covering of $M(S)$. Therefore it reduces to the case that $V_S$ is free over $S$. Choose an $S$-basis $e_1,\dots,e_d$ of $V_S$, and write $a=\sum_{i=1}^da_ie_i$ with $a_i\in K\otimes_{K_0}(\widetilde{\mathbf{B}}^{\dagger,s}_{\rig}\widehat{\otimes}_{\Q}S)$. Since $\varphi$ acts trivially on $V_R$, (\ref{eq:solution-extended-robba}) admits a solution in $V_R\otimes_R(\widetilde{\r}^{\rho(s)}_{\widehat{K^{\mathrm{ur}}}}\widehat{\otimes}_{\Q}R)$ if and only if each Frobenius equation
\begin{equation}\label{eq:solution-extended-robba-i}
\varphi^f(b_i)-\beta g(\alpha)b_i=g(a_i)
\end{equation}
admits a solution $b_i$ in $\widetilde{\r}^{\rho(s)}_{\widehat{K^{\mathrm{ur}}}}\widehat{\otimes}_{\Q}R$.  By \cite[Proposition 3.3.2]{Ke05}\footnote{By the definition of $\Gamma_{\mathrm{con}}^{\mathrm{alg}}$, one may identify it with $K\otimes_{K_0}\widetilde{\mathbf{B}}^{\dagger}$.}, we may choose some 
\[
x\in (K\otimes_{K_0}\widetilde{\mathbf{B}}^{\dagger})^\times
\] 
such that $y=\beta\varphi^f(x)/x$ belongs to $K$. Using Frobenius, we see that $x$ actually lies in $(K\otimes_{K_0}\widetilde{\mathbf{B}}^{\dagger,s})^\times$. We thus rewrite (\ref{eq:solution-extended-robba-i}) as
\[
\varphi^f(xb_i)-yg(\alpha )xb_i=\varphi^f(x)g(a_i).
\]
Note that $|y|=|\beta|$. Thus $|y^{-1}g(\alpha)^{-1}|\leq|\beta^{-1}||\alpha^{-1}|<1$. We deduce from Lemma \ref{lemma:Frobenius-equation} that (\ref{eq:solution-extended-robba}) admits a solution in $V_R\otimes_R(\widetilde{\r}^{\rho(s)}_{\widehat{K^{\mathrm{ur}}}}\widehat{\otimes}_{\Q}R)$ if and only if $M(R)\ra M(S)$ factors through 
\[
M(S')=\bigcap _{1\leq i\leq d}M(y\alpha, \varphi^f(x)a_i).
\] 
Furthermore, in this case, the solution is unique. Let $b$ be the solution of (\ref{eq:solution-extended-robba})  in $V_{S'}\otimes_{S'}(\widetilde{\r}^{\rho(s)}_{\widehat{K^{\mathrm{ur}}}}\widehat{\otimes}_{\Q}S')$.
Let $L$ be a finite extension of $K$ so that $\D^{\dag,s}_{\rig,L}(V_S)$ is free over $\mathbf{B}^{\dag,s}_{\rig,L}\widehat{\otimes}_{\Q}S$.
Choose a $\mathbf{B}^{\dag,s}_{\rig,L}\widehat{\otimes}_{\Q}S'$-basis $\{f_1,\dots,f_d\}$ of $\D^{\dag,s}_{\rig,L}(V_{S'})$. Since
\[
(K\otimes_{K_0}\D_{\rig,L}^{\dag,s}(V_{S'}))\otimes_{(K\otimes_{K_0}\mathbf{B}_{\rig,L}^{\dag,s})\widehat{\otimes}_{\Q}S'}
(\widetilde{\r}^{\rho(s)}_{\widehat{K^{\mathrm{ur}}}}\widehat{\otimes}_{\Q}S')
=V_{S'}\otimes_{S'}(\widetilde{\r}^{\rho(s)}_{\widehat{K^{\mathrm{ur}}}}\widehat{\otimes}_{\Q}S'),
\]
we may write $b=\sum_{i=1}^db_if_i$ with $b_i\in \widetilde{\r}^{\rho(s)}_{\widehat{K^{\mathrm{ur}}}}\widehat{\otimes}_{\Q}S'$.
Let $S''$ be the $\Qp$-affinoid algebra defined by
\[
M(S'')=\bigcap_{1\leq i\leq d}M(S'(c_i,s))
\] 
(see Lemma \ref{lem:solution-descent} for the definition of $M(S'(c_i,s))$). By Lemma \ref{lem:solution-descent}, $g(b)$ belongs to $K\otimes_{K_0}\D^{\dag,s}_{\rig,L}(V_R)$ if and only if the map $M(R)\ra M(S')$ factors through $M(S'')$.
Furthermore, by the uniqueness of the solution of (\ref{eq:solution-extended-robba}), the image of $b$ in $K\otimes_{K_0}\D^{\dag,s}_{\rig,L}(V_{S''})$ is $H_K$-invariant; hence it is in $K\otimes_{K_0}\D^{\dag,s}_{\rig,K}(V_{S''})$
by Theorem \ref{thm:BC-generalization}(4). Therefore we can take $S(\alpha,\beta, a)=S''$.
\end{proof}


By Lemma \ref{lem:phi-rig}, $(\D^\dagger_\rig(V_S))^{\varphi^f=\alpha}$ is contained in $\D_\rig^{\dag,s}(V_S)$ for any $\alpha\in S$ and $s\geq s(V_S)$. Thus for any $n\geq n(V_S)$, we have a natural map 
\[
\D^\dagger_\rig(V_S)^{\varphi^f=\alpha}\ra\D_{\dif}^{+,n}(V_S)
\] 
via the localization map $\iota_n$. 

\begin{prop}\label{prop:cris-dif-injective}
Let $\alpha\in S^\times$. Then for any $k>\log_{|\pi_K^{-1}|}|\alpha^{-1}|$ and $n\geq n(V_S)$, the natural map
\[
\iota_{n,K}: K\otimes_{K_0}\D_\rig^\dagger(V_S)^{\varphi^f=\alpha}\ra\D^{+,fn}_\dif(V_S)/(t^k)
\]
is injective.
\end{prop}
\begin{proof}
Let $a\in (K\otimes_{K_0}\D_\rig^\dagger(V_S))^{\varphi^f=\alpha}$, and let $a_m$ be its image in $\D_\dif^{+,fm}(V_S)/(t^k)$ via $\iota_{m,k}$
for any $m\geq n(V_S)$. The relation $\varphi^f(a)=\alpha a$ and Proposition \ref{prop:brig-mod-phi-t}(3) imply
\[
a_m=\alpha^{n-m}a_n.
\] 
Thus if $a_n=0$, then $a_m=0$ for all $m\geq n(V_S)$. This implies $t^k_K|a$ by Proposition \ref{prop:localization-ramified}. Now suppose that $a$ lies in the kernel of the map, and write $a=t_K^ka'$
for some $a'\in \D_\rig^\dagger(V_S)$.
It follows that 
\begin{equation}\label{eq:prop322}
\varphi^f(a')=(t_K/\varphi^f(t_K))^k\alpha a'.
\end{equation}
By Proposition \ref{prop:localization-ramified} and assumption, 
\[
|(t_K/\varphi^f(t_K))^{k}|=|\pi_K^{-1}|^k>|\alpha^{-1}|.
\] 
Hence
$a'=0$ is the unique solution of (\ref{eq:prop322}) by Proposition \ref{prop:solution-extended-robba}.
\end{proof}

\begin{prop}\label{prop:a-solution-locus}
For any $n\geq n(V_S)$, $k>\log_{|\pi_K^{-1}|}|\alpha^{-1}|$ and $a\in \D^{+,fn}_{\dif}(V_S)/(t^k)$,
there exists an $E$-analytic subspace $M(S(k,\alpha,a))$ of $M(S)$
such that for any map $g:S\ra R$ of affinoid algebras over $E$, $g(a)\in \D^{+,fn}_{\dif}(V_R)/(t^k)$ is contained in the image of 
\[
\iota_{n,K}: K\otimes_{K_0}\D_{\rig}^{\dag}(V_R)^{\varphi^f=g(\alpha)}\ra \D^{+,fn}_{\dif}(V_R)/(t^k)
\]
if and only if the map $M(R)\ra M(S)$ factors through $M(S(k,\alpha,a))$.
\end{prop}
\begin{proof}
As in the proof of Proposition \ref{prop:solution-extended-robba}, it suffices to treat the case that $V_S$ is free. Using Proposition \ref{prop:localization-ramified}, we choose $\tilde{a}\in\D_{\rig}^{\dag,r_{fn}}(V_S)$ such that the image of $\iota_{m,K}(\tilde{a})$ in $\D^{+,fm}_{\dif}(V_S)/(t^k)$ is $\alpha^{m-n}a$ for each $m\geq n$.
 If $g(a)$ can be lifted to $\tilde{b}\in (K\otimes_{K_0}\D_{\rig}^{\dag}(V_R))^{\varphi^f=g(\alpha)}$, it then follows that the image of $\iota_{m,K}(\tilde{b})$ in $\D^{+,fm}_{\dif}(V_R)/(t^k)$ is $g(\alpha)^{m-n}g(a)$. Then by Proposition \ref{prop:localization-ramified}, we see that $t_K^k$ divides $\tilde{b}-g(\tilde{a})$ in $K\otimes_{K_0}\D_{\rig}^{\dag,r_{fn}}(V_R)$. Hence 
 \[
 b=(\tilde{b}-g(\tilde{a}))/t^k_K
 \] 
 is a solution of the equation
\begin{equation}\label{eq:Q(varphi)-solution-1}
(\varphi^f-g(\alpha))(g(\widetilde{a})+t_K^kb)=0. 
\end{equation}
Conversely, any solution $b\in K\otimes_{K_0}\D_\rig^{\dag,r_{fn}}(V_R)$ of (\ref{eq:Q(varphi)-solution-1}) gives rise to the desired lift $g(\widetilde{a})+t_K^kb$ of $g(a)$. 
Therefore, we conclude that $g(a)$ can be lifted to $(K\otimes_{K_0}\D_{\rig}^{\dag}(V_R))^{\varphi^f=g(\alpha)}$ if and only if  (\ref{eq:Q(varphi)-solution-1}) has a solution $b\in K\otimes_{K_0}\D_\rig^{\dagger,r_{fn}}(V_R)$. 

A short computation shows that (\ref{eq:Q(varphi)-solution-1}) can be rewritten as
\begin{equation*}
\varphi^f(t_K)^k(\varphi^f-(t_K/\varphi^f(t_K))^kg(\alpha))b
=(g(\alpha)-\varphi^f)(g(\widetilde{a})).
\end{equation*}
By the construction of $\tilde{a}$, $t_K^k$ divides $\varphi^f(\tilde{a})-\alpha\tilde{a}$ in $K\otimes_{K_0}\D^{\dag,r_{fn}}_\rig(V_S)$. Note that 
\[
(t_K)=(\varphi^f(t_K)) 
\]
in $K\otimes_{K_0}\mathbf{B}^{\dagger,r_{f(n+1)}}_{\rig,K}$
by Proposition \ref{prop:localization-ramified}(3). Hence $\varphi^f(t_K)^k$ divides $(g(\alpha)-\varphi^f(\alpha))(g(\widetilde{a}))$ in $K\otimes_{K_0}\D_\rig^{\dagger,r_{f(n+1)}}(V_R)$. We therefore deduce that (\ref{eq:Q(varphi)-solution-1}) has a solution in $\D_\rig^{\dagger,r_{fn}}(V_R)$ if and only if
the equation
\begin{equation}\label{eq:Q(varphi)-solution-2}
(\varphi^f-(t_K/\varphi^f(t_K))^kg(\alpha))b
=\varphi^f(t_K)^{-k}(g(\alpha)-\varphi^f)(g(\widetilde{a})).
\end{equation}
has a solution in $K\otimes_{K_0}\D_{\rig}^{\dagger,r_{f(n+1)}}(V_R)$.
In fact, if $b$ is such a solution, we have 
\[
b\in K\otimes_{K_0}\D^{\dag,r_{fn}}_{\rig}(V_R)
\] 
by Lemma \ref{lem:phi-rig}.
The assumption implies that $|(t_K/\varphi^f(t_K))^k|=|\pi_K^{-1}|^k>|\alpha^{-1}|$. Using Proposition \ref{prop:solution-extended-robba}, $g(a)$ can be lifted to $(K\otimes_{K_0}\D^\dag_{\rig}(V_R))^{\varphi^f=g(\alpha)}$ if and only if the map $M(R)\ra M(S)$ factors through $M(S(\alpha, (t_K/\varphi^f(t_K))^k, \varphi^f(t_K)^{-k}(\alpha-\varphi^f)(\widetilde{a}))$. Thus we can take $S(k,\alpha, a)$ to be $S(\alpha, (t_K/\varphi^f(t_K))^k, \varphi^f(t_K)^{-k}(\alpha-\varphi^f)(\widetilde{a}))$.
\end{proof}

\begin{cor}\label{cor:solution-locus-1}
For any integer $n\geq n(V_S)$ and positive integer $k>\log_{|\pi_K^{-1}|}|\alpha^{-1}|$,
there exists an $E$-analytic subspace $M(S(k,\alpha,n))$ of $M(S)$
such that for any map $g:S\ra R$ of affinoid algebras over $E$, the $R$-submodule $g((\D^{+,fn}_{\dif}(V_S)/(t^k))^{\Gamma})$ of $(\D^{+,fn}_{\dif}(V_R)/(t^k))^{\Gamma}$ is contained in the image of 
\[
(K\otimes_{K_0}\D_{\rig}^{\dag}(V_R))^{\varphi^f=g(\alpha),\Gamma=1}\ra (\D^{+,fn}_{\dif}(V_R)/(t^k))^{\Gamma}
\]
if and only if the map $M(R)\ra M(S)$ factors through $M(S(k,\alpha,n))$.
\end{cor}
\begin{proof}
Recall that by Proposition \ref{prop:cris-dif-injective}, the map
\[
K\otimes_{K_0}\D_{\rig}^{\dag}(V_R)^{\varphi^f=g(\alpha)}\ra \D^{+,fn}_{\dif}(V_R)/(t^k)
\]
is injective. Since it is also $\Gamma$-equivariant, we deduce that for any $a\in(\D^{+,fn}_{\dif}(V_S)/(t^k))^{\Gamma}$, if $g(a)$ can be lifted to $K\otimes_{K_0}\D_{\rig}^{\dag}(V_R)^{\varphi^f=g(\alpha)}$, then the lift is also $\Gamma$-invariant. Thus we can take $M(S(k,\alpha,n))$ to be the intersection of $M(S(k,\alpha,a))$ for all $a\in(\D^{+,fn}_{\dif}(V_S)/(t^k))^{\Gamma}$.
\end{proof}

\begin{cor}\label{cor:solution-locus}
Keep notations as in Corollary \ref{cor:solution-locus-1}. Then
there exists an $E$-analytic subspace $M(S(k,\alpha))$ of $M(S)$
such that for any map $g:S\ra R$ of affinoid algebras over $E$, the $R$-submodule $g((\D^{+,fn}_{\dif}(V_S)/(t^k))^{\Gamma})$ of $(\D^{+,fn}_{\dif}(V_R)/(t^k))^{\Gamma}$ is contained in the image of \begin{equation}
(K\otimes_{K_0}\D^\dag_{\rig}(V_R))^{\varphi^f=g(\alpha),\Gamma=1}\ra (\D^{+,fn}_{\dif}(V_R)/(t^k))^{\Gamma}
\end{equation}
for all sufficiently large $n$ if and only if the map $M(R)\ra M(S)$ factors through $M(S(k,\alpha))$.
\end{cor}
\begin{proof}
It is clear that we can take
$M(S(k,\alpha))$ to be the intersection of $M(S(k,\alpha, n))$ for all $n\geq n(V_S)$.
\end{proof}

\subsection{Finite slope subspaces}

\begin{theorem}\label{thm:fs}
The rigid analytic space $X$ has a unique finite slope subspace $X_{fs}$.
\end{theorem}
\begin{proof}
By Proposition \ref{prop:fs-glue}, it suffices to treat the case that $X=M(S)$ is an affinoid space.
Let
\[
X'=\bigcap_{k>\log_{|\pi_K^{-1}|}|\alpha^{-1}|}M(S(k,\alpha)).
\]
Now for each $i\geq1$ and $\tau\in\mathrm{H}_K$, let $X'_{i}$ be the Zariski closure of 
\[
\bigcap_{\tau\in\mathrm{H}_K}X'_{P(i)_{\tau}}=\bigcap_{0\leq j\leq i-1,\tau\in\mathrm{H}_K}X'_{Q(-j)_{\tau}}.
\]
We claim that 
\[
X_{fs}=\bigcap_{i\geq 1}X'_{i}
\] 
is the finite slope subspace of $X$. First note that the decreasing sequence of closed subspaces $X'_{1}\supseteq X'_{2}\supseteq\dots$ becomes constant eventually because $S$ is Noetherian. We fix an $i_0$ such that $X'_{i}=X'_{i_0}$ for all $i\geq i_0$. Hence $X_{fs}=X'_{i}$ for all $i\geq i_0$.  Thus for any $i\geq i_0$ and $\tau\in\mathrm{H}_K$, we have
\[
(X_{fs})_{P(i)_{\tau}}\supseteq X_{fs}\cap X'_{P(i)_{\tau}}=X_{i_0}\cap X'_{P(i)_{\tau}}\supseteq \bigcap_{\tau\in\mathrm{H}_K}X'_{P(i)_{\tau}}.
\]
 Therefore the Zariski closure of $(X_{fs})_{P(i)_{\tau}}$ contains the Zariski closure of $\cap_{\tau\in\mathrm{H}_K}X'_{P(i)_{\tau}}$, which is $X'_i=X_{fs}$; this yields that $X_{fs}$ satisfies (1) of Definition \ref{def:fs-space}.

Now suppose that $g:M(R)\ra M(S)$ is a map of affinoid spaces over $E$ which factors through $X_{Q(j)}$ for every $j\leq0$. It follows from Corollary \ref{cor:t-descent} that for each $k\geq1$ and $n\geq n(V_S)$, the natural map
\[
(\mathrm{D}_{\dif}^{+,fn}(V_R))^{\Gamma}\ra(\mathrm{D}_{\dif}^{+,fn}(V_R)/(t^k))^{\Gamma}
\]
is an isomorphism. Hence (\ref{eq:cris-dR}) is an isomorphism if and only if the natural map
\begin{equation}\label{eq:cris-t^k}
(K\otimes_{K_0}\mathrm{D}^{\dag}_{\rig}(V_R))^{\varphi^f=g^*(\alpha),\Gamma=1}
\ra(\mathrm{D}_{\dif}^{+,fn}(V_R)/(t^k))^{\Gamma}
\end{equation}
is surjective for some (hence any) $k\geq1$. By Corollary \ref{cor:t-descent}, the natural map
\[
(\mathrm{D}_{\mathrm{dif}}^{+,fn}(V_S)/(t^k))^{\Gamma}\otimes_SR\ra
(\mathrm{D}_{\mathrm{dif}}^{+,fn}(V_R)/(t^k))^{\Gamma}
\]
is an isomorphism. Hence by Corollary \ref{cor:solution-locus-1}, the map (\ref{eq:cris-t^k}) is surjective if and only if the map $g:M(R)\ra M(S)$ factors through $M(S(k,\alpha,n))$ for each $k>\log_{|\pi_K^{-1}|}|\alpha^{-1}|$ by Corollary \ref{cor:solution-locus}. We thus conclude that (\ref{eq:cris-dR}) is an isomorphism for all sufficiently large $n$ if and only if $g:M(R)\ra M(S)$ factors through $X_{fs}$. This yields that $X_{fs}$ satisfies $(2)$ of Definition \ref{def:fs-space}.
\end{proof}

\begin{prop}\label{prop:fs-subdomain}
For any affinoid subdomain $M(S)$ of $X_{fs}$ and $k>\log_{|\pi_K^{-1}|}|\alpha^{-1}|$, we have $S(k,\alpha)=S$. As a consequence, for such $k$, the natural map
\[
(K\otimes_{K_0}\D^{\dag}_{\rig}(V_S))^{\varphi^f=\alpha,\Gamma=1}\ra(\D_{\dif}^{+,fn}(V_S)/(t^k))^{\Gamma}
\]
is an isomorphism for all $n\geq n(V_S)$.
\end{prop}
\begin{proof}
It is obvious that the finite slope subspace of $X_{fs}$ is $X_{fs}$ itself. We then deduce that $(M(S))_{fs}=M(S)$ since the formation of finite slope subspaces commutes with flat base change by Proposition \ref{prop:flat-base-change}. This yields that  $M(S)\subseteq M(S(k,\alpha))$ following the construction of the finite slope subspace in Theorem \ref{thm:fs}; hence $M(S)=M(S(k,\alpha))$. This yields the surjectivity of the map. The injectivity follows from Proposition \ref{prop:cris-dif-injective}.
\end{proof}

\begin{theorem}\label{thm:fs-sheaf}
Let $M(S)$ be an affinoid subdomain of $X_{fs}$. Then
for any $n\geq n(V_S)$ and $k>\log_{|\pi_K^{-1}|}|\alpha^{-1}|_{\mathrm{sp}}$ where the norm is taken in $S$, the natural map of sheaves
\[
(K\otimes_{K_0}\mathscr{D}_{\rig}^{\dag}(V_S))^{\varphi^f=\alpha,\Gamma=1}
\ra(\mathscr{D}_{\dif}^{+,fn}(V_S)/(t^k))^{\Gamma}
\]
is an isomorphism. As a consequence,
 $(\mathscr{D}_{\rig}^{\dag}(V_{X_{fs}}))^{\varphi^f=\alpha,\Gamma=1}$
is a coherent sheaf on $X_{fs}$.
\end{theorem}
\begin{proof}
It follows from Proposition \ref{prop:fs-subdomain} that the map of sheaves
\[
(K\otimes_{K_0}\mathscr{D}_{\rig}^{\dag}(V_S))^{\varphi^f=\alpha,\Gamma=1}\ra(\mathscr{D}_{\dif}^{+,fn}(V_S)/(t^k))^{\Gamma}
\]
is an isomorphism.
By Proposition \ref{prop:dif-torsion-sheaf}, we therefore deduce that $(\mathscr{D}_{\rig}^{\dag}(V_S))^{\varphi^f=\alpha,\Gamma=1}$ is a coherent sheaf. 
\end{proof}

\begin{theorem}\label{thm:fs-sheaf-general}
For any $E$-affinoid algebra $R$ and morphism $g:M(R)\ra X_{fs}$ which factors through $X_{Q(j)}$ for every integer $j\leq0$, the natural map
\begin{equation*}
    (K\otimes_{K_0}\mathscr{D}^{\dag}_{\rig}(V_R))^{\varphi^f=g^*(\alpha),\Gamma=1}\ra (\mathscr{D}_{\dif}^{+,fn}(V_R)/(t^k))^{\Gamma}
\end{equation*}
is an isomorphism for all sufficiently large $k$. As a consequence, $(\mathscr{D}^{\dag}_{\rig}(V_R))^{\varphi^f=g^*(\alpha),\Gamma=1}$ is a coherent sheaf.
\end{theorem}
\begin{proof}
We choose an admissible affinoid covering $\{M(S_i)\}_{i\in I}$ of $X_{fs}$ by affinoid subdomains. Let $\{M(R_j)\}_{j\in J}$ be a finite covering of $M(R)$ which refines the pullback of the covering $\{M(S_i)\}_{i\in I}$ on $M(R)$. Suppose that $M(R_j)$ maps to $M(S_{i_j})$ for each $j\in J$. Let $k$ be a positive integer such that 
\[
k>\log_{|\pi_K^{-1}|}\max_{i\in I}\{|\alpha^{-1}|\hspace{2mm}\text{in $S_i$}\}.
\] 
Now for any affinoid subdomain $M(R')$ of some $M(R_j)$, 
\[
(\D^{+,fn}_{\dif}(V_{R'})/(t^k))^{\Gamma=1}=(\D^{+,fn}_{\dif}(V_{S_{i_j}})/(t^k))^{\Gamma=1}\otimes_{S_{i_j}}R'
\] 
by Corollary \ref{cor:base-change} because $M(R')$ maps to $X_{P(k)}$. On the other hand, by Proposition \ref{prop:fs-subdomain}, we have $M(S_{i_j}(k,\alpha))=M(S_{i_j})$, yielding that 
\[
(K\otimes_{K_0}\mathrm{D}^{\dag}_{\rig}(V_{R'}))^{\varphi^f=g^*(\alpha),\Gamma=1}\ra (\mathrm{D}_{\dif}^{+,fn}(V_{R'})/(t^k))^{\Gamma}
\]
is surjective. Furthermore, it is injective by Proposition \ref{prop:cris-dif-injective}; so it is an isomorphism. Hence 
\[
(K\otimes_{K_0}\mathscr{D}^{\dag}_{\rig}(V_{R_j}))^{\varphi^f=g^*(\alpha),\Gamma=1}\ra (\mathscr{D}_{\dif}^{+,fn}(V_{R_j})/(t^k))^{\Gamma}
\]
is an isomorphism. This yields the theorem.
\end{proof}

\begin{remark}\label{rem:fs-same}
Our finite slope subspace $X_{fs}$ coincides with Nakamura's generalization of Kisin's finite slope subspace \cite{N11}. In fact, as noted in Remark \ref{rem:fs-artinian}, to characterize our finite slope subspaces, it suffices to test only finite $\Q$-algebras $R$ in Definition \ref{def:fs-space}(2). By the argument in \cite[(5.8)]{Ki03}, the same thing holds for Nakamura's finite slope subspaces as well. For such $R$, we have the comparisons 
\[
(\D_{\rig}^\dag(V_R))^{\Gamma}=D^+_{\mathrm{crys}}(V_R) \hspace{5mm}\text{and}\hspace{5mm} (\D^+_{\dif}(V_R))^{\Gamma}=D^+_{\mathrm{dR}}(V_R)
\]
by \cite[Th\'eor\`{e}me 3.6]{LB02} and \cite[Th\'eor\`{e}me 3.9]{F04} respectively. Thus our Definition \ref{def:fs-space}(2) coincides with the counterpart of Nakamura's in this case; hence the claim.
\end{remark}

\section{Global triangulation of refined families}

In this section, we prove the main results of this paper.

\subsection{Weakly refined families}
From now on, let $X$ be a reduced rigid analytic space over $E$, and let $V_X$ be a family of weakly refined $p$-adic representations of $G_K$ of dimension $d$ on $X$ as in \S0.3. We further suppose $\kappa_1=0$. Therefore the Sen polynomial for $V_X$ is of the form $TQ(T)$ with $Q(T)\in K\otimes_{\Q}\OO(X)[T]$. As in \S3, we put $P(i)=\prod_{j=0}^{i-1}Q(-j)$ for $i\geq1$.
\begin{convention}
For $c\in\mathbb{R}$, $1\leq i\leq d$, $x\in X$ and $?\in\{>,<,\leq,\geq\}$, we say $\kappa_i(x)?c$ if $\kappa_i(x)_\tau?c$ for any $\tau\in\H_K$.
\end{convention}

The first goal of this subsection is to show that the finite slope subspace of $X$ with respect to the pair $(F, V_X)$ is $X$ itself (Theorem \ref{thm:weak-refined-family}). We start by collecting some basic properties about the $p$-adic representations $V_z$ for $z\in Z$. 
\begin{lemma}\label{lem:crys-Frob-eigenspace}
The following are true.
\begin{enumerate}
\item[(1)]If $x\in X_{P(k)}$ for some $k\geq1$, then 
\[
\dim_{k(x)}(\D_\dif^{+,n}(V_x)/(t^k))_\tau^{\Gamma}\leq 1
\] 
for any $\tau\in \H_K$.
\item[(2)]
For any $z\in Z$ and $\sigma\in\mathrm{Gal}(K_0/\Q)$,
$\dim_{k(z)}(\D_{\rig}^{\dag}(V_z))_{\sigma}^{\varphi^f=F(z),\Gamma=1}\geq 1$.
\item[(3)]
For any $z\in Z$, if 
$v_K(F(z))<-\kappa_i(z)$
for all $i\geq 2$, then $\D_{\rig}^{\dag}(V_z)_{\sigma}^{\varphi^f=F(z),\Gamma=1}$ has dimension 1 over $k(z)$.
Furthermore, for any $k\geq1$ satisfying 
$v_K(F(z))<k\leq-\kappa_i(z)$
for all $i\geq 2$, the natural map
\begin{equation}\label{eq:crys-Frob-eigenspace}
(K\otimes_{K_0}\D_{\rig}^{\dag}(V_z)_{\sigma})^{\varphi^f=F(z),\Gamma=1}\ra\oplus_{\tau\in \H_\sigma}(\D_\dif^{+,fn}(V_z)/(t^k))_\tau^\Gamma
\end{equation}
is an isomorphism.
\end{enumerate}
\end{lemma}

\begin{proof}
By Corollary \ref{cor:t-descent},  the map 
\[
(\D_{\dif}^{+,n}(V_x)/(t^k))^\Gamma\ra(\D^n_{\Sen}(V_x))^\Gamma
\] 
is an isomorphism. On the other hand, since $x\in X_{P(k)}$, we see that $\kappa_1(x)=0$ is a multiplicity-one root of the Sen polynomial for $V_x$. This implies 
\[
\dim_{k(x)}(\D^n_{\Sen}(V_x))_\tau^\Gamma\leq 1
\] 
for any $\tau\in\mathrm{H}_K$. Hence $\dim_{k(x)}(\D_{\dif}^{+,n}(V_x)/(t^k))_\tau^\Gamma\leq 1$, yielding $(1)$.  For (2), note that the Hodge-Tate weights of $V_z$ are all nonpositive. Hence by Berger's dictionary (\cite[Th\'eor\`{e}me 3.6]{LB02}), 
\[
\dim_{k(z)}(\D_{\rig}^{\dag}(V_z))_\sigma^{\varphi^f=F(z),\Gamma=1}=\dim_{k(z)}(D_{\mathrm{crys}}(V_z))_\sigma^{\varphi^f=F(z)}
\geq \dim_{k(z)}(D_{F(z)})_\sigma^{\varphi^f=F(z)}=1,
\]
where the inequality follows from Definition \ref{def:weak-refined-family}(d). For (3), since 
$k>v_K(F(z))$,
we get the injectivity of (\ref{eq:crys-Frob-eigenspace}) by Proposition \ref{prop:cris-dif-injective}. On the other hand,  note that $z\in X_{P(k)}$ since $k\leq -k_i(z)$ for $2\leq i\leq d$. It follows by (1) that $\dim(\D_\dif^{+,fn}(V_z)/(t^k))_{\tau}^\Gamma\leq 1$. Hence the dimension of right hand side of (\ref{eq:crys-Frob-eigenspace}) over $k(x)$ is at most $[K:\Q]/f$. On the other hand, by (2), the left hand side of (\ref{eq:crys-Frob-eigenspace}) has $k(x)$-dimension at least $[K:\Q]/f$. Putting everything together, we conclude (\ref{eq:crys-Frob-eigenspace}) is an isomorphism, and 
\[
\dim_{k(z)}(\D_{\rig}^{\dag}(V_z))_\sigma^{\varphi^f=F(z),\Gamma=1}=1.
\]
\end{proof}

\begin{theorem}\label{thm:weak-refined-family}
The finite slope subspace of $X$ with respect to $(F, V_X)$ is $X$ itself.
\end{theorem}
\begin{proof}
Since $Z$ is Zariski dense in $X$, it suffices to prove $Z\subset X_{fs}$. Now let $z\in Z$, and let $M(S)$ be an affinoid subdomain containing $z$. Let $k$ be an integer such that 
\[
k>\log_{|\pi_K^{-1}|}|F^{-1}|,
\] 
where $|\cdot|$ means the sup norm on $M(S)$, i.e. the spectral norm of $S$. It follows that 
\[
k>\log_{|\pi_K^{-1}|}|F(x)^{-1}|
\] 
for any $x\in M(S)$. Now for $z'\in Z_k\cap M(S)$, we first have 
\[
-k_i(z')>k> v_K(F(z'))=\log_{|\pi_K^{-1}|}|F(z')^{-1}|
\] 
for $i\geq 2$ by the definition of $Z_k$.  We then deduce that the natural map 
\[
(K\otimes_{K_0}\D_\rig^{\dag}(V_{z'}))^{\varphi^f=F(z'),\Gamma=1}\ra(\D_{\dif}^{+,fn}(V_{z'})/(t^k))^\Gamma
\] 
is an isomorphism by Lemma \ref{lem:crys-Frob-eigenspace}(3). Hence $z'\in M(S(k,F))$ by Corollary \ref{cor:solution-locus-1}. Since $Z_k\cap M(S)$ is Zariski dense in $M(S)$ by Definition \ref{def:weak-refined-family}(e) and $M(S(k,F))$ is Zariski closed by its construction, we conclude $S(k,F)=S$ for all $k>\log_{|\pi_K^{-1}|}|F^{-1}|$. Furthermore, for any $i\geq1$, since $Z_i\cap M(S)$ is Zariski dense in $M(S)$, we deduce that 
$M(S)_{P(i)}\supset Z_i\cap M(S)$
is also Zariski dense in $M(S)$. We therefore conclude $M(S)_{fs}=M(S)$ following the construction of finite slope subspace. Hence $z\in M(S)_{fs}\subset X_{fs}$.
\end{proof}

The following theorem follows immediately from Theorem \ref{thm:weak-refined-family} and Theorem \ref{thm:fs-sheaf}.
\begin{theorem}\label{thm:weakly-refined}
Let $M(S)$ be an affinoid subdomain of $X$. Then for any $k>\log_{|\pi_K^{-1}|}|F^{-1}|$,
 \[
    (K\otimes_{K_0}\mathscr{D}_\rig^\dag(V_S))^{\varphi^f=F,\Gamma=1}
    \ra(\mathscr{D}^{+,fn}_{\dif}(V_S)/(t^k))^{\Gamma}
    \]
is an isomorphism. As a consequence,  $(\mathscr{D}_{\rig}^{\dag}(V_X))^{\varphi^f=F,\Gamma=1}$ is a coherent sheaf on $X$.
\end{theorem}

Our next goal is to show that the saturated locus $X_s$ of $V_X$ (see Definition \ref{def:intro-good}) is Zariski open and dense. To do this, we need to investigate the specialization maps 
\[
\mathscr{D}_\rig^\dag(V_X)^{\varphi^f=F,\Gamma=1}\otimes_{\OO_X}k(x)\ra \D_\rig^\dag(V_x)^{\varphi^f=F(x),\Gamma=1}
\] 
for $x\in X$. 

\begin{prop}\label{prop:weakly-refined}
Let $M(S)$ be an affinoid subdomain of $X$, and let $k$ be a positive integer satisfying $k>\log_{|\pi_K^{-1}|}|F^{-1}|$. Then the following are true.
\begin{enumerate}
\item[(1)]
For any $x\in M(S)_{P(k)}$, the natural map
\begin{equation}\label{eq:weakly-refined-1}
(K\otimes_{K_0}\D_\rig^\dag(V_x))^{\varphi^f=F(x),\Gamma=1}\ra (\D_{\dif}^{+,fn}(V_x)/(t^k))^{\Gamma}
\end{equation}
is an isomorphism.
\item[(2)]
For any $x\in M(S)_{P(k)}$, the natural map
\begin{equation}\label{eq:weakly-refined}
\D_\rig^\dag(V_S)^{\varphi^f=F,\Gamma=1}\otimes_Sk(x)\ra \D_\rig^\dag(V_x)^{\varphi^f=F(x),\Gamma=1}
\end{equation}
is an isomorphism.
\item[(3)]
For any $x\in M(S)$, $\dim_{k(x)}(\D_\rig^\dag(V_S)_\sigma^{\varphi^f=F,\Gamma=1}\otimes_Sk(x))\geq1$.
\item[(4)]
For any $x\in M(S)_{P(k)}$, $\dim_{k(x)}(\D_\rig^\dag(V_x))_\sigma^{\varphi^f=F(x),\Gamma=1}=\dim_{k(x)}(\D_{\dif}^{+,fn}(V_x)/(t^k))_\tau^{\Gamma}=1$.
\end{enumerate}
\end{prop}
\begin{proof}
Let $x\in M(S)_{P(k)}$.
Consider the following commutative diagram
 \[
\xymatrix{(K\otimes_{K_0}\D^{\dag}_{\rig}(V_S))^{\varphi^f=F,\Gamma=1}\otimes_Sk(x)\ar[d]\ar[r]&
(\D^{+,fn}_{\dif}(V_{S})/(t^k))^\Gamma\otimes_Sk(x)\ar[d]\\
(K\otimes_{K_0}\D^{\dag}_{\rig}(V_x))^{\varphi^f=F(x),\Gamma=1}\ar[r]&
(\D^{+,fn}_{\dif}(V_{x})/(t^k))^\Gamma.}
\]
The upper horizontal map is an isomorphism by Theorem \ref{thm:weakly-refined}.  The right vertical map is an isomorphism by Corollary \ref{cor:base-change}. The lower horizontal map is injective by Proposition \ref{prop:cris-dif-injective}. We thus deduce that the lower horizontal map and left vertical map are all isomorphisms. This yields (1) and (2).

We first prove (3) for $x\in Z$. In fact, by (2) and Lemma \ref{lem:crys-Frob-eigenspace}(3), we have 
\[
\dim_{k(z)}(\D_\rig^\dag(V_S)_\sigma^{\varphi^f=F,\Gamma=1}\otimes_Sk(z))=1
\] 
for any $z\in Z_k\cap M(S)$. Since $Z_k\cap M(S)$ is Zariski dense in $M(S)$ and $\D_\rig^\dag(V_S)_\sigma^{\varphi^f=F,\Gamma=1}$ is a finitely generated $S$-module, we therefore deduce that 
\[
\dim_{k(x)}(\D_\rig^\dag(V_S)_\sigma^{\varphi^f=F,\Gamma=1}\otimes_Sk(x))\geq 1
\] 
for any $x\in M(S)$. 
For (4), on one hand, we have 
\[
\dim_{k(x)}(\D_{\dif}^{+,fn}(V_x)/(t^k))_\tau^{\Gamma}\leq1
\] 
by Lemma \ref{lem:crys-Frob-eigenspace}(1). On the other hand, we have 
\[
\dim_{k(x)}(\D^{\dag}_{\rig}(V_x))_\sigma^{\varphi^f=F(x),\Gamma=1}\geq1
\] 
by (2) and (3). We then deduce (4) from (1).
\end{proof}

\begin{prop}\label{prop:good}
The subset of saturated points $X_s$ is  Zariski open in $X$.
\end{prop}
\begin{proof}
For each $\tau\in \H_K$, let $Y_\tau$ be the set of $x\in X$ such that the image of the composite 
\[
\mathscr{D}_\rig^\dag(V_X)^{\varphi^f=F,\Gamma=1}\ra \D_\rig^\dag(V_x)^{\varphi^f=F(x),\Gamma=1}\ra
\D^{fn}_\Sen(V_x)
\ra\D^{fn}_\Sen(V_x)_\tau
\] 
is zero for some (hence all sufficiently large) $n$. It is clear that each $Y_\tau$ is a Zariski closed subset of $X$. By Proposition \ref{prop:weakly-refined}(3),  the condition (1) of Definition \ref{def:intro-good} cuts out a Zariski open subset $X'$ of $X$. For $x\in X'$, by Lemma \ref{lem:saturated-submodule} (this is not circular!) we see that $x$ satisfies Definition \ref{def:intro-good}(2) if and only if $x\notin Y_\tau$ for any $\tau\in \H_K$. Therefore, we conclude
\[
X_s=X'\setminus \cup_{\tau\in \H_K}Y_\tau
\] 
is a Zariski open subspace of $X$.
.

\end{proof}

\begin{prop}\label{prop:weakly-refined-1}
For $x\in X$ and $k>v_K(F(x))$, if $x\in X_{P(k)}$, then $x\in X_s$ and
\[
\dim_{k(x)}(\D_\rig^{\dag}(V_x))_\sigma^{\varphi^f=F(x),\Gamma=1}=1.
\]
\end{prop}
\begin{proof}
Since $k>v_K(F(x))$, we may choose an affinoid neighborhood $M(S)$ of $x$ such that $k>\log_{|\pi_K^{-1}|}|F^{-1}|$ in $S$. By Proposition \ref{prop:weakly-refined} (2) and (4), we first deduce that $x$ satisfies Definition \ref{def:intro-good}(1), and
\[
\dim_{k(x)}(\D_\rig^{\dag}(V_x))_\sigma^{\varphi^f=F(x),\Gamma=1}=1.
\] 
Note that
\[
(\D_{\dif}^{+,fn}(V_x)/(t^k))^{\Gamma}\ra(\D_{\Sen}^{fn}(V_x))^\Gamma
\] 
is an isomorphism by Corollary \ref{cor:base-change}. Thus by Proposition \ref{prop:weakly-refined} (1) and (2), we deduce that 
\[
(K\otimes_{K_0}\D_\rig^\dag(V_S))^{\varphi^f=F,\Gamma=1}\otimes_Sk(x)\ra (\D_\Sen^{fn}(V_x))^\Gamma
\] 
is an isomorphism. Hence $x$ satisfies Definition \ref{def:intro-good}(2).
\end{proof}

\begin{cor}\label{cor:good-dense}
The subset of saturated points $X_s$ is a Zariski open and dense subspace of $X$.
\end{cor}
\begin{proof}
By Proposition \ref{prop:good}, $X_s$ is Zariski open in $X$. It remains to show that it is Zariski dense. Now for any affinoid subdomain $M(S)$ of $X$, it follows from Proposition \ref{prop:weakly-refined-1} that $M(S)_{P(k)}\subset X_s$ once $k$ is sufficiently large. Since $M(S)_{P(k)}$ is Zariski dense in $M(S)$, the corollary follows.
\end{proof}

In the rest of this subsection, we will determine a large class of points $x\in X$ which is contained in the saturated locus $X_s$ (Proposition \ref{prop:multiplicity-one}). To do this, we need to employ the following flatification result. Let
\[
\pi: Y'\ra Y
\] 
be a proper and birational morphism of separated and reduced rigid analytic spaces over $E$. Here birational means that for some coherent sheaf of ideals $H$, the complement $U$ of the closed subset $V(H)$, which is defined by $H$, is Zariski dense in $Y$, the restriction of $\pi$ to $\pi^{-1}(U)$ is an isomorphism, and $\pi^{-1}(U)$ is Zariski dense in $Y'$. Let $N$ be a coherent sheaf of $\OO_X$-modules. If $H'$ is the coherent sheaf of ideal defining the closed subset $\pi^{-1}(V(H))$ of $X'$, then the \emph{strict transform $N'$} of $N$ by $\pi$ is the quotient of $\pi^*N$ by its $H'^{\infty}$-torsion. In particular, for any morphism $\pi^*N\ra M$ of coherent sheaves over $X'$, if $M$ is torsion-free, then the morphism $\pi^*N\ra M$ factors through $N'$.  The following lemma follows from \cite[Lemma 3.4.2]{BC06}.

\begin{lemma}\label{lem:blow-up}
Let $Y$ be a separated and reduced rigid analytic space over $E$. If $M$ is a torsion-free coherent sheaf of modules over $Y$, then there exists a proper and birational morphism $Y'\ra Y$ of rigid analytic spaces with $Y'$ reduced such that the strict transform of $M$ by $\pi$ is a locally free coherent sheaf of modules $N$ over $Y'$. More precisely, we may choose $\pi$ to be the blow-up along a nowhere dense Zariski closed subspace of the normalization of $Y$.
\end{lemma}

In the rest of this subsection let $V_Y$ be a locally free coherent $\OO_Y$-module of rank $d$ equipped with a continuous $\OO_Y$-linear $G_K$-action. We denote by $d_{n}$ the rank of $\mathscr{D}_{\dif}^{+,n}(V_{Y})_\tau/(t^k)$ as a locally free $\OO_Y$-module for any $\tau\in \H_K$ (it is independent of $\tau$).

\begin{lemma}\label{lem:blow-up-free}
Let $V_Y$ be a locally free coherent $\OO_Y$-module of rank $d$ equipped with a continuous $\OO_Y$-linear $G_K$-action. Suppose for some $n>0$, the coherent sheaf $\mathscr{D}_\dif^{+,n}(V_Y)/(t^k)$ is well-defined.  Now let $\pi: Y'\ra Y$ be as in Lemma \ref{lem:blow-up}, and suppose for some $\tau\in\mathrm{H}_K$,  the strict transforms of 
 $(\mathscr{D}_\dif^{+,n}(V_Y)/(t^k))_\tau/(\mathscr{D}_\dif^{+,n}(V_Y)_\tau/(t^k))_\tau^{\Gamma}$ and $(\mathscr{D}_\dif^{+,n}(V_Y)/(t^k))_\tau^\Gamma$ by $\pi$ are locally free over $\OO_{Y'}$ of ranks $c$ and $d_n-c$ respectively. Then
$(\mathscr{D}_\dif^{+,n}(V_{Y'})/(t^k))_\tau/(\mathscr{D}_\dif^{+,n}(V_{Y'})/(t^k))_\tau^{\Gamma}$ and $(\mathscr{D}_\dif^{+,n}(V_{Y'})/(t^k))_\tau^\Gamma$
are locally free over $\OO_{Y'}$ of ranks $c$ and $d_n-c$ respectively as well.
\end{lemma}
\begin{proof}
We denote by $\mathscr{D}_1$ the strict transform of $(\mathscr{D}_{\dif}^{+,n}(V_Y)/(t^k))_\tau^{\Gamma}$, and by $\mathscr{D}_2$ the strict transform of $(\mathscr{D}_{\dif}^{+,n}(V_Y)/(t^k))_\tau/(\mathscr{D}_{\dif}^{+,n}(V_Y)/(t^k))_\tau^{\Gamma}$.
Since $Y'$ is normal by Lemma \ref{lem:blow-up}, it is a disjoint union of irreducible components.  By Proposition \ref{prop:torsion-free}, we see both $(\mathscr{D}_{\dif}^{+,n}(V_{Y'})/(t^k))_\tau^{\Gamma}$ and $(\mathscr{D}_{\dif}^{+,n}(V_{Y'})/(t^k))_\tau/(\mathscr{D}_{\dif}^{+,n}(V_{Y'})/(t^k))_\tau^{\Gamma}$ are torsion-free on each irreducible component of $Y'$. Hence the natural maps
\[
\pi^*((\mathscr{D}_{\dif}^{+,n}(V_Y)/(t^k))_\tau^{\Gamma})\ra(\mathscr{D}_{\dif}^{+,n}(V_{Y'})/(t^k))_\tau^{\Gamma}
\]
and
\[
\pi^*((\mathscr{D}_{\dif}^{+,n}(V_Y)/(t^k))_\tau/(\mathscr{D}_{\dif}^{+,n}(V_Y)/(t^k))_\tau^{\Gamma})\ra (\mathscr{D}_{\dif}^{+,n}(V_{Y'})/(t^k))_\tau/(\mathscr{D}_{\dif}^{+,n}(V_{Y'})/(t^k))_\tau^{\Gamma}
\]
factor through $\mathscr{D}_1$ and $\mathscr{D}_2$ respectively. Similarly, since $(\mathscr{D}_{\dif}^{+,n}(V_{Y'})/(t^k))_{\tau}$ is torsion-free,  the natural map
\[
\pi^*((\mathscr{D}_{\dif}^{+,n}(V_Y)/(t^k))_\tau^{\Gamma})\ra\pi^{*}(\mathscr{D}_{\dif}^{+,n}(V_{Y})/(t^k))_{\tau}\cong
(\mathscr{D}_{\dif}^{+,n}(V_{Y'})/(t^k))_{\tau}
\]
factors through $\mathscr{D}_1$. To conclude, consider the following commutative diagram
\[
\xymatrix{
& \mathscr{D}_1\ar[d] \ar[r]
& \pi^*(\mathscr{D}_{\dif}^{+,n}(V_Y)/(t^k))_{\tau} \ar[d]^{\simeq} \ar[r]
& \mathscr{D}_2 \ar[d]  \\
0\ar[r] & (\mathscr{D}_{\dif}^{+,n}(V_{Y'})/(t^k))_\tau^{\Gamma}\ar[r] &(\mathscr{D}_{\dif}^{+,n}(V_{Y'})/(t^k))_{\tau} \ar[r] &
(\mathscr{D}_{\dif}^{+,n}(V_{Y'})/(t^k))_\tau/(\mathscr{D}_{\dif}^{+,n}(V_{Y'})/(t^k))_\tau^{\Gamma}\ar[r] & 0}
\]
where the top sequence satisfies that the second map is surjective and the composite map is zero. By diagram chasing, we see that the right vertical map is surjective, and its kernel is isomorphic to the cokernel of the left vertical map. Since $\pi$ is an isomorphism on $\pi^{-1}(U)$, the map
$\mathscr{D}_1\ra (\mathscr{D}_{\dif}^{+,n}(V_{Y'})/(t^k))_\tau^{\Gamma}$
is an isomorphism on $\pi^{-1}(U)$ by Proposition \ref{prop:dif-torsion-sheaf}.  It follows that the kernel of 
\[
\mathscr{D}_2\ra (\mathscr{D}_{\dif}^{+,n}(V_{Y'})/(t^k))_\tau/(\mathscr{D}_{\dif}^{+,n}(V_{Y'})/(t^k))_\tau^{\Gamma}
\]
is supported on $Y'\setminus \pi^{-1}(U)$, which is a nowhere dense Zariski closed subspace of $Y'$. Therefore the kernel is zero because $\mathscr{D}_2$ is locally free and $Y'$ is reduced. Hence the right vertical map is an isomorphism.  Thus  $(\mathscr{D}_{\dif}^{+,n}(V_{Y'})/(t^k))_\tau/(\mathscr{D}_{\dif}^{+,n}(V_{Y'})/(t^k))_\tau^{\Gamma}$ is locally free of rank $d_{n}-c$. This implies that $(\mathscr{D}_{\dif}^{+,n}(V_{Y'})/(t^k))_\tau^{\Gamma}$ is locally free of rank $c$.
\end{proof}

\begin{lemma}\label{lem:fs-birational}
Keep notations as above. Suppose that the Sen polynomial for $V_Y$ has no constant term, and that $Y_{fs}=Y$ with respect to the pair $(V_Y,\alpha)$ for some $\alpha\in\OO(Y)^\times$. If $\pi: Y'\ra Y$ is a proper birational morphism with $Y'$ reduced, the finite slope subspace of $Y'$ with respect to $(\pi^{*}\alpha, V_{Y'})$ is $Y'$ itself.
\end{lemma}
\begin{proof}
It is clear that $Y'$ satisfies Definition \ref{def:fs-space}(2). Furthermore, since $\pi$ is birational and $Y'$ is reduced, we also deduce that $Y'$ satisfies Definition \ref{def:fs-space}(1).
\end{proof}

\begin{lemma}\label{lem:non-vanishing-birational}
Keep assumptions as in Lemma \ref{lem:fs-birational}. Moreover, suppose there exists an integer $k$ satisfying
\[
k>\log_{|\pi_K^{-1}|}|\alpha^{-1}|. 
\]  
If $Y'\ra Y$ is a proper and birational morphism as in Lemma \ref{lem:blow-up} such that for any $\tau\in \H_K$, the strict transforms of $(\mathscr{D}_{\dif}^{+,fn}(V_Y)/(t^k))_\tau^{\Gamma}$ and $(\mathscr{D}_{\dif}^{+,fn}(V_Y)/(t^k))_\tau/(\mathscr{D}_{\dif}^{+,fn}(V_Y)/(t^k))_\tau^{\Gamma}$ by $\pi$ are locally free of rank $c$ and $d_n-c$ respectively, then $(\mathscr{D}_\rig^\dag(V_{Y'}))_\sigma^{\varphi^f=\pi^*\alpha,\Gamma=1}$ is locally free of rank $c$ and 
\[
\mathscr{D}_\rig^\dag(V_{Y'})_\sigma^{\varphi^f=\pi^*\alpha,\Gamma=1}\otimes k(y)\ra \D_\rig^\dag(V_y)_\sigma
\]
is injective for any $\sigma\in\mathrm{Gal}(K_0/\Q)$ and $y\in Y'$.
\end{lemma}
\begin{proof}
By Lemma \ref{lem:fs-birational}, $Y'_{fs}=Y'$. Since $k>\log_{|\pi_K^{-1}|}|\alpha^{-1}|$,  the natural map
\[
(K\otimes_{K_0}\mathscr{D}^\dagger_{\rig}(V_{Y'})_\sigma)^{\varphi^f=\pi^*\alpha,\Gamma=1}\ra \oplus_{\tau\in \H_\sigma}(\mathscr{D}_{\dif}^{+,fn}(V_{Y'})/(t^k))_\tau^{\Gamma}
\]
is an isomorphism by Theorem \ref{thm:fs-sheaf}. By Lemma \ref{lem:blow-up-free}, $(\mathscr{D}_{\dif}^{+,fn}(V_{Y'})/(t^k))_\tau^{\Gamma}$ is locally free of rank $c$, so $(\mathscr{D}^\dagger_{\rig}(V_{Y'}))_\sigma^{\varphi^f=\pi^*\alpha,\Gamma=1}$ as well. Furthermore, since $(\mathscr{D}_{\dif}^{+,fn}(V_{Y'})/(t^k))_\tau/(\mathscr{D}_{\dif}^{+,fn}(V_{Y'})/(t^k))_\tau^{\Gamma}$ is also locally free by Lemma \ref{lem:blow-up-free}, we deduce that
\[
(\mathscr{D}_\dif^{+,fn}(V_{Y'})/(t^k))_\tau^{\Gamma}\otimes k(y)\ra (\mathscr{D}_{\dif}^{+,fn}(V_{Y'})/(t^k))_\tau\otimes k(y)\cong(\D_\dif^{+,fn}(V_y)/(t^k))_{\tau}
\]
is injective for any $y\in Y'$; the isomorphism follows from the fact that the functor $\D^{+,n}_{\dif}$ is compatible with base change (Proposition \ref{prop:base-change}). This implies that the map
\[
(K\otimes_{K_0}\mathscr{D}_\rig^\dag(V_{Y'}))_\sigma^{\varphi^f=\pi^*\alpha,\Gamma=1}\otimes k(y)\ra  \oplus_{\tau\in \H_\sigma}(\mathscr{D}_{\dif}^{+,fn}(V_{y})/(t^k))_\tau
\]
is injective. Note that this map factors through $(K\otimes_{K_0}\D_\rig^\dag(V_y)_\sigma)^{\varphi^f=\alpha(\pi(y)),\Gamma=1}$. The lemma follows. 
\end{proof}

\begin{prop}\label{prop:multiplicity-one}
For any $x\in X$ and $\sigma\in\mathrm{Gal}(K_0/\Q)$, if 
\[
\dim_{k(x)}\D_\rig^\dag(V^{\mathrm{ss}}_x)_\sigma^{\varphi^f=F(x),\Gamma=1}=1,
\] 
then $\mathscr{D}_\rig^\dag(V_X)_\sigma^{\varphi^f=F,\Gamma=1}$ is locally free of rank 1 around $x$, and
\[
\mathscr{D}_\rig^\dag(V_X)_\sigma^{\varphi^f=F,\Gamma=1}\otimes k(x)\ra \D_\rig^\dag(V_x)_\sigma^{\varphi^f=F(x),\Gamma=1}
\]
is an isomorphism
\end{prop}
\begin{proof}
We may suppose $X=M(S)$ is an affinoid space. Let $\widetilde{X}$ be the normalization of $X$. It follows that  $\widetilde{X}$ is the disjoint union of finitely many irreducible components. By Proposition \ref{prop:torsion-free}, both $(\mathscr{D}_{\dif}^{+,n}(V_{\widetilde{X}})/(t^k))_\tau^{\Gamma}$ and $(\mathscr{D}_{\dif}^{+,n}(V_{\widetilde{X}})/(t^k))_\tau/(\mathscr{D}_{\dif}^{+,fn}(V_{\widetilde{X}})/(t^k))_\tau^{\Gamma}$ are torsion-free on each irreducible component of $\tilde{X}$. Using Lemmas \ref{lem:blow-up},  \ref{lem:blow-up-free} and \ref{lem:non-vanishing-birational}, there exists a proper birational map $\widetilde{\pi}:X'\ra \widetilde{X}$ such that $(\mathscr{D}_\rig^\dag(V_{X'}))_\sigma^{\varphi^f=\pi^*F,\Gamma=1}$
is locally free of rank 1, and 
\[
(\mathscr{D}_\rig^\dag(V_{X'}))_\sigma^{\varphi^f=\pi^*F,\Gamma=1}\otimes k(x')\ra \D_\rig^\dag(V_{x'})_\sigma
\]
is injective for any $x'\in X'$. In particular, the map is nonzero. Therefore for any ideal $I$ of cofinite length of $\OO_{X',x'}$, the composite
\begin{equation*}
\begin{split}
(\mathscr{D}_\rig^\dag(V_{X'}))_\sigma^{\varphi^f=\pi^*F,\Gamma=1}\otimes_{\OO_{X'}}(\OO_{X', x'}/I)\ra (\D_\rig^\dag(V_{X'}&\otimes_{\OO_{X'}}(\OO_{X',x'}/I)))^{\varphi^f=\pi^*F,\Gamma=1}_\sigma\\
&\ra
(\D_\rig^\dag(V_{x'}))^{\varphi^f=(\pi^*F)(x'),\Gamma=1}_\sigma
\end{split}
\end{equation*}
is nonzero. 

Now let $\pi$ be the composite $X'\ra\widetilde{X}\ra X$, which is  also birational. If 
\[
x'\in\pi^{-1}(x),
\] 
applying \cite[Lemma 3.3.9]{BC06} to the functor $D^+_{\mathrm{crys}}(\cdot)^{\varphi^f=\pi^*F}_\sigma$, we deduce
\[
D^+_{\mathrm{crys}}(V_{X'}\otimes_{\OO_{X'}}(\OO_{X',x'}/I))^{\varphi^f=\pi^*F}_\sigma
\]
is free of rank 1 over $\OO_{X',x'}/I$.
Therefore, by \cite[Proposition 3.2.3]{BC06}, for all ideals $I$ of cofinite length of $\OO_{X,x}$, $D^+_{\mathrm{crys}}(V_{X}\otimes_{\OO_{X}}(\OO_{X,x}/I))_\sigma^{\varphi^f=F}$ is free of rank 1 over $\OO_{X, x}/I$. 

Furthermore, we claim that if $I'\supset I$ is another ideal of $\OO_x$, the natural map
\begin{equation}\label{eq:ideal-surjective}
D^+_{\mathrm{crys}}(V_{X}\otimes_{\OO_{X}}(\OO_{X,x}/I))_\sigma^{\varphi^f=F}\ra D^+_{\mathrm{crys}}(V_{X}\otimes_{\OO_{X}}(\OO_{X,x}/I'))_\sigma^{\varphi^f=F}
\end{equation}
is surjective. In fact, since $D^+_{\mathrm{crys}}$ is left exact, we first have the following exact sequence
\[
0\ra D^+_{\mathrm{crys}}(V_X\otimes_{\OO_X}(I'/I))_\sigma^{\varphi^f=F}\ra D^+_{\mathrm{crys}}(V_{X}\otimes_{\OO_{X}}(\OO_{X,x}/I))_\sigma^{\varphi^f=F}\ra D^+_{\mathrm{crys}}(V_{X}\otimes_{\OO_{X}}(\OO_{X,x}/I'))_\sigma^{\varphi^f=F}.
\]
This implies that 
\begin{equation*}
\begin{split}
l(D^+_{\mathrm{crys}}(V_X\otimes_{\OO_X}(I'/I))_\sigma^{\varphi^f=F})&\geq l(D^+_{\mathrm{crys}}(V_{X}\otimes_{\OO_{X}}(\OO_{X,x}/I))_\sigma^{\varphi^f=F})-l(D^+_{\mathrm{crys}}(V_{X}\otimes_{\OO_{X}}(\OO_{X,x}/I'))_\sigma^{\varphi^f=F})\\
&=l(\OO_{X,x}/I)-l(\OO_{X,x}/I')\\
&=l(I'/I).
\end{split}
\end{equation*}
On the other hand, we deduce from the assumption that
\begin{equation*}
l(D^+_{\mathrm{crys}}(V_X\otimes_{\OO_X}(I'/I))_\sigma^{\varphi^f=F})\leq l(I'/I).
\end{equation*}
This forces $l(D^+_{\mathrm{crys}}(V_X\otimes_{\OO_X}(I'/I))_\sigma^{\varphi^f=F})=l(I'/I)$. Hence (\ref{eq:ideal-surjective}) is surjective. 

Now choose a positive integer $k>\log_{|\pi_K^{-1}|}|F^{-1}|$. By Theorem \ref{thm:weakly-refined}, the map
\[
(K\otimes_{K_0}\D_\rig^{\dag}(V_S)_\sigma)^{\varphi^f=F,\Gamma=1}\ra \oplus_{\tau\in \H_\sigma}(\D_{\dif}^{+,n}(V_S)/(t^k))_\tau^{\Gamma}
\] 
is an isomorphism. Since $\widehat{\OO}_{X,x}$ is flat over $S$, we deduce from Lemma \ref{lem:flat-invariant} that
\[
(\D_{\dif}^{+,fn}(V_S)/(t^k))_\tau^\Gamma\otimes_S\widehat{\OO}_{X,x}\cong
((\D_{\dif}^{+,fn}(V_S)/(t^k))_{\tau}\otimes_S\widehat{\OO}_{X,x})^\Gamma.
\]
Since $(\D_{\dif}^{+,fn}(V_S)/(t^k))_{\tau}$ is finite locally free over $S$, we get
\[
(\D_{\dif}^{+,fn}(V_S)/(t^k))_{\tau}\otimes_S\widehat{\OO}_{X,x}\cong\limproj_l
(\D_{\dif}^{+,fn}(V_S)/(t^k))_{\tau}\otimes_SS/\mathfrak{m}^l_{x})\cong\limproj_l
(\D_{\dif}^{+,fn}(V_{S}/\mathfrak{m}^l_xV_S)/(t^k))_{\tau},
\]
where the last isomorphism follows from the base change property of the functor $\D_{\mathrm{dif}}^{+,fn}$.
Hence
\[
((\D_{\dif}^{+,fn}(V_S)/(t^k))_{\tau}\otimes_S\widehat{\OO}_{X,x})^\Gamma\cong\limproj_l
((\D_{\dif}^{+,fn}(V_S)/(t^k))_{\tau}\otimes_SS/\mathfrak{m}^l_x)^{\Gamma}\cong\limproj_l
(\D_{\dif}^{+,fn}(V_{S}/\mathfrak{m}^l_xV_S)/(t^k))_{\tau}^\Gamma.
\]
Now consider the following commutative diagram
\[
\xymatrix{(K\otimes_{K_0}\D^{\dagger}_{\rig}(V_S)_\sigma)^{\varphi^f=F,\Gamma=1}\otimes_S\widehat{\OO}_{X,x}\ar[d]\ar[r]&
\oplus_{\tau\in \H_\sigma}(\D_{\dif}^{+,fn}(V_S)_\tau/(t^k))^\Gamma\otimes_S\widehat{\OO}_{X,x}\ar[d]\\
\displaystyle{\limproj_l
(K\otimes_{K_0}\D_{\rig}^{\dagger}(V_{S}/\mathfrak{m}^l_xV_S)_\sigma)^{\varphi^f=F,\Gamma=1}}\ar[r]&
\displaystyle{\limproj_l
\oplus_{\tau\in \H_\sigma}(\D_{\dif}^{+,fn}(V_{S}/\mathfrak{m}^l_xV_S)_\tau/(t^k))^{\Gamma}}.}
\]
By the previous paragraph we see that $\displaystyle{\limproj_l
\D_{\rig}^{\dagger}(V_{S}/\mathfrak{m}^l_xV_S)_\sigma^{\varphi^f=F,\Gamma=1}}$
is a free $\widehat{\OO}_{X,x}$-module of rank 1. Since both the top horizontal and right vertical maps are isomorphisms, we deduce that the left vertical map embeds $(K\otimes_{K_0}\D^{\dagger}_{\rig}(V_S)_\sigma)^{\varphi^f=F,\Gamma=1}\otimes_S\widehat{\OO}_{X,x}$ as a direct summand of 
\[
\displaystyle{\limproj_l(K\otimes_{K_0}
\D_{\rig}^{\dagger}(V_{S}/\mathfrak{m}^l_xV_S)_\sigma)^{\varphi^f=F,\Gamma=1}}
=K\otimes_{K_0}\displaystyle{\limproj_l
\D_{\rig}^{\dagger}(V_{S}/\mathfrak{m}^l_xV_S)_\sigma^{\varphi^f=F,\Gamma=1}}.
\]
It follows that the map
\[
(K\otimes_{K_0}\D^{\dagger}_{\rig}(V_S))_{\sigma}^{\varphi^f=F,\Gamma=1}\otimes_Sk(x)\ra
(K\otimes_{K_0}\D^{\dagger}_{\rig}(V_x))_{\sigma}^{\varphi^f=F(x),\Gamma=1}
\]
is injective.  On the other hand, we have 
\[
\dim(\D^{\dagger}_{\rig}(V_S)_\sigma^{\varphi^f=F,\Gamma=1}\otimes_Sk(x))\geq1
\] 
by Proposition \ref{prop:weakly-refined}(3). Thus the left hand side is at least $[K:\Q]/f$-dimensional whereas the right hand side is exactly $[K:\Q]/f$-dimensional. Hence 
\[
\dim_{k(x)}(\D^{\dagger}_{\rig}(V_S)_\sigma^{\varphi^f=F,\Gamma=1}\otimes_Sk(x))=1
\]
and \[
\D^{\dagger}_{\rig}(V_S)_\sigma^{\varphi^f=F,\Gamma=1}\otimes_Sk(x)\ra
\D^{\dagger}_{\rig}(V_x)_\sigma^{\varphi^f=F(x),\Gamma=1}
\]
is an isomorphism.
\end{proof}

\subsection{Vector bundles and $\m$-modules}
In this subsection, we will recall some basic notions and properties of the theory of families $\m$-modules. Recall that we denote by $K'_0$ the maximal unramified extension of $\Q$ contained in $K_\infty$. Let $S$ be an affinoid algebra over $\Q$. 
\begin{defn}
Let $I$ be a subinterval of $(0,\infty)$. By a \emph{vector bundle} over $\r^I_{K_0'}\widehat{\otimes}_{\Q}S$ of rank $d$
we mean a locally free coherent sheaf $M^I_S$ of rank $d$ over the product
of the annulus $v_p(T)\in I$ within the affine $T$-line over $K_0'$ with $M(S)$
in the category of rigid analytic spaces over $\Q$.  We call $M_S^I$ \emph{free} if it is freely generated by its global sections.
By a \emph{vector bundle} $M_S$ over $\r_{K_0'}\widehat{\otimes}_{\Q}S$ we mean an object in the direct limit
as $r \to 0$ of the categories of vector bundles over
$\r_{K_0'}^r\widehat{\otimes}_{\Q}S$.

For a subinterval $I'$ of $I$, denote by
$M_S^{I'}$ the base change of $M_S^I$ to $\r^{I'}_{K_0'}\widehat{\otimes}_{\Q}S$. If $S\ra R$ is a map of affinoid algebras over $\Q$, we set $M_R^I$ and $M_R$ as the base changes of $M_S^I$ and $M_S$ to $\r^{I}_{K_0'}\widehat{\otimes}_{\Q}R$ and $\r_{K_0'}\widehat{\otimes}_{\Q}R$ respectively. For any $x\in M(S)$, we denote $M^I_{k(x)}$ and $M_{k(x)}$ by $M^I_x$ and $M_x$ respectively instead. 
\end{defn}

\begin{remark}
By Lemma \ref{lem:projection}, $M^r_x=M^r_S\otimes_{S}k(x)$. Hence the map $M^r_S\ra M^r_x$ is surjective.
\end{remark}

\begin{remark}
A locally free $\r^I_{K'_0}\widehat{\otimes}_{\Q}S$-module of rank $d$ naturally gives rise to a vector bundle of rank $d$ over $\r^{I}_{K_0'}\widehat{\otimes}_{\Q}S$. The converse is also true when $I$ is a closed interval.
\end{remark}

We need the following result, which is originally due to L\"utkebohmert \cite{L}, in \S4.3.

\begin{lemma}\label{lem:closed-interval}
Let $M_S^I$ be a vector bundle over $\r_{K_0'}^I\widehat{\otimes}_{\Q}S$. If $I$ is closed, then there exists a finite covering of $M(S)$ by affinoid subdomains $M(S_1),\dots, M(S_i)$ such that $M^I_{S_1},\dots, M^I_{S_i}$ are all free.
\end{lemma}

Recall that
there exists an isomorphism 
$\mathbf{B}^\dag_{\rig,K}\cong\r_{K_0'}$,
which identifies $\mathbf{B}^{\dag,\rho(r)}_{\rig,K}$ with $\r^r_{K_0'}$ for all sufficiently small $r$. 
We henceforth identify $\mathbf{B}_{\rig,K}^{\dagger}\widehat{\otimes}_{\Q}S$ with $\r_{K_0'}\widehat{\otimes}_{\Q}S$, and equip the latter with the induced $\varphi$- and $\Gamma$-actions.

\begin{defn}\label{def:phi-gamma}
By a \emph{$\m$-module} over $\r_{K_0'}\widehat{\otimes}_{\Q}S$ of rank $d$ we mean a vector bundle $D_S$ over $\r_{K_0'}\widehat{\otimes}_{\Q}S$ of rank $d$
equipped with commuting semilinear $\varphi$- and $\Gamma$-actions such that the induced map $\varphi^*D_S\ra D_S$ is an
isomorphism as vector bundles over $\r_{K_0'}\widehat{\otimes}_{\Q}S$.
We say $D_S$ \emph{free} if the underlying vector bundle is free. The morphisms of $\m$-modules over $\r_{K_0'}\widehat{\otimes}_{\Q}S$ are morphisms of the underlying vector bundles which respect $\varphi$- and $\Gamma$-actions.
\end{defn}

\begin{defn}
Let $D_S$ be a $\m$-module over $\r_{K_0'}\widehat{\otimes}_{\Q}S$. It is clear from Definition \ref{def:phi-gamma} that for $r$ sufficiently small, $D_S$ is represented by a vector bundle $D_S^r\subset D_S$ over $\r^r_{K_0'}\widehat{\otimes}_{\Q}S$ such that $\varphi$ maps $D_S^r$ to $D_S^{r/p}$, and the induced map $\varphi^*(D_S^r)\ra D_S^{r/p}$
is an isomorphism as vector bundles over  $\r^{r/p}_{K_0'}\widehat{\otimes}_{\Q}S$. We call such $D_S^r$ \emph{representative vector bundles} of $D_S$.
\end{defn}
\begin{remark}
Our definition of $\m$-modules over $\r_{K_0'}\widehat{\otimes}_{\Q}S$ is the same as the notion of \emph{families of $\m$-modules} over $\r_{S_{K'}}$ defined in \cite{KL}.
\end{remark}

\begin{remark}\label{rem:Drig-phi-Gamma}
If $V_S$ is a locally free $S$-linear representation of rank $d$ of $G_{K}$, $\D_\rig^\dag(V_S)$ is naturally a $\m$-module of rank $d$ over
$\r_{{K_0'}}\widehat{\otimes}_{\Q}S$ with representative vector bundles $\D_\rig^{\dag,\rho(r)}(V_S)$ for $r$ sufficiently small.
\end{remark}

\begin{remark}
If $S$ is a finite extension of $\Q$, then $D_S$ is free over $\r_{K'_0}\widehat{\otimes}_{\Q}S=\r_{K'_0}\otimes_{\Q}S$ by the B\'ezout property of $\r_{K'_0}\otimes_{\Q}S$. Thus our definition of $\m$-modules is compatible with the definition of classical $\m$-modules.
\end{remark}

\begin{remark}
In fact, one can also define a $\m$-module over $\r_{K_0'}\widehat{\otimes}_{\Q}S$ of rank $d$ to be a finite presented projective module over $\r_{K_0'}\widehat{\otimes}_{\Q}S$ of rank $d$ equipped with commuting semilinear $\varphi$- and $\Gamma$-actions such the induce map 
$\varphi^*(D_S)\ra D_S$ is an isomorphism. The equivalence between this definition and ours are proved in \cite{KPX} and
\cite{Bel13} independently. 
\end{remark}


\begin{lemma}\label{lem:representative-unique}
Let $L$ be a finite extension of $\Q$, and put $L'=L\otimes_{\Q}K_0'$. Let $D$ be a $\m$-module over $\r_{K_0'}\otimes_{\Q}L=\r_{L'}$ of rank $n$, and let $E$ be a $\m$-submodule of $D$ of rank $m$. Then there exists an $r_0>0$ such that if $D^r$ and $E^r$ are representative vector bundles of $D$ and $E$ for some $r\leq r_0$, then $E^r\subset D^r$. As a consequence, $D$ has at most one representative vector bundle over $\r_{L'}^r$ when $r$ is sufficiently small.
\end{lemma}
\begin{proof}
Fix some $r_0>0$ such that for any $a\in\r_{K'_0}$, if $\varphi(a)\in \r^r_{K_0'}$ for 
some $0<r\leq r_0$, then $a\in\r^{pr}_{K'}$. Now let $\mathrm{d}=(d_1,\dots,d_n)$ and $\mathrm{e}=(e_1,\dots,e_m)$ be $\r_{L'}^r$-bases of $D^r$ and $E^r$ respectively. Since $D^r$ and $E^r$ are representative vector bundles, there exist invertible matrices $A$ and $B$ defined over $\r_{L'}^{r/p}$ such that
$\varphi(\mathrm{d})=\mathrm{d}A$ and $\varphi(\mathrm{e})=\mathrm{e}B$. Write $\mathrm{e}=\mathrm{d}C$ for some
$n\times m$ matrix $C$ defined over $\r_{L'}$. It follows 
\[
\mathrm{d}CB=\mathrm{e}B=\varphi(\mathrm{e})=
\varphi(\mathrm{d})\varphi(C)=\mathrm{d}A\varphi(C),
\] 
yielding $CB=A\varphi(C)$. Hence $\varphi(C)=A^{-1}CB$. Now
suppose $C$ is defined over $\r_{L'}^s$ for some $s>0$. If $s<r/p$, then $\varphi(C)=A^{-1}CB$ is over $\r_{L'}^s$, yielding that
$C$ is defined over $\r_{L'}^{ps}$. Iterating this argument, we conclude that $C$ is defined over $\r_{L'}^{r/p}$. Thus $\varphi(C)$ is defined over
$\r_{L'}^{r/p}$, yielding $C$ is defined over $\r_{L'}^r$. This implies $E^r\subset D^r$.
\end{proof}

\begin{lemma}\label{lem:representative-saturated}
Keep notations as in the Lemma \ref{lem:representative-unique}.
Then $E$ is saturated in $D$ if and only if $E^r$ is saturated in $D^r$. Furthermore, in this case,  we have
\[
E^r=D^r\cap E,
\] 
and $D^r/E^r$ is the representative vector bundle of $D/E$ over $\r_{L'}^r$.
\end{lemma}
\begin{proof}
It is obvious that if $E^r$ is saturated in $D^r$, then $E$ is saturated in $D$. Now suppose $E$ is saturated in 
$D$. First note that $D^r/(E\cap D^r)$ is a submodule of $D/E$. Hence $D^r/(E\cap D^r)$ is finitely generated and torsion-free over $\r_{L'}^r$. This yields that it is finite free over $\r_{L'}^r$ by the B\'ezout property of $\r_{L'}^r$.  Furthermore, since it generates $D/E$ over $\r_{L'}$, we deduce
\[
\rank_{\r_{L'}^r}(D^r/(E\cap D^r))\geq \rank_{\r_{L'}}(D/E)=d-s.
\] 
On the other hand, $E\cap D^r$ is a closed $\r_{L'}^r$-submodule of $D^r$. Hence it is also finite free over 
$\r_{L'}^r$. Since $E^r\subseteq E\cap D^r$, we deduce
\[
\rank_{\r_{L'}^r}(E\cap D^r)\geq \rank_{\r^r_{L'}}E^r=s.
\] 
Since
\[
\rank(E\cap D^r)+\rank(D^r/(E\cap D^r))=\rank D^r=d,
\] 
we deduce 
\[
\rank(E\cap D^r)=s\quad \text{and}\quad \rank(D^r/(E\cap D^r))=d-s. 
\]

We claim that $E\cap D^r$ and $D^r/(E\cap D^r)$ are representative vector bundles of $E$ and $D/E$ respectively. First note that the natural map
\[
(D^r/(E\cap D^r))\otimes_{\r_{L'}^r}\r_{L'}\ra D/E
\] 
is an isomorphism because it is surjective, and both sides are finite free over $\r_{L'}$ of the same rank. 
It follows that 
\[
(E\cap D^r)\otimes_{\r_{L'}^r}\r_{L'}\ra E
\] 
is also an isomorphism. Now consider the following commutative diagram
\[
\xymatrix{
&0\ar[r] & \varphi^*(E\cap D^r)\ar[d] \ar[r]
& \varphi^*(D^r)\ar[d] \ar[r]
& \varphi^*(D^r/(E\cap D^r))\ar[d] \ar[r] &0  \\
& 0 \ar[r] & (E\cap D^{r})\otimes_{\r_{L'}^{r}}\r^{r/p}_{L'} \ar[r] &
D^{r/p} \ar[r] & D^{r/p}/((E\cap D^{r})\otimes_{\r_{L'}^{r}}\r^{r/p}_{L'})\ar[r]&0.}
\]
The middle vertical map is an isomorphism as $D^r$ is a representative vector bundle. Thus the right vertical map is surjective. Hence it is an isomorphism because both the source and target are finite free of the same rank over 
$\r^{r/p}_{L'}$. This yields that the left vertical map is also an isomorphism. The claim now follows, and we deduce the lemma from Lemma \ref{lem:representative-unique}.
\end{proof}

\begin{prop}\label{prop:quotient}
Keep notations as above. Let $D_S$ be a $\m$-module over $\r_{{K_0'}}\widehat{\otimes}_{\Q}S$ of rank $d$, and let $E_S$ be a $\m$-submodule of $D_S$ of rank $s$. Suppose $E_S^r\subset D_S^r$ are representative vector bundles of $E_S$ and $D_S$ respectively. If $E_x$ is a saturated $\m$-submodule of $D_x$ for every $x\in M(S)$, $D^r_S/E^r_S$ is a vector bundle over $\r_{{K_0'}}^r\widehat{\otimes}_{\Q}S$ of rank $d-s$. As a consequence, $D_{S}/E_{S}$ is a $\m$-module over $\r_{{K_0'}}\widehat{\otimes}_{\Q} S$ of rank $d-s$. 
\end{prop}
\begin{proof}
It suffices to show that $E_x^r$ is saturated
in $D_x^r$ for every $x\in M(S)$. The latter follows from Lemma
\ref{lem:representative-saturated}.
\end{proof}


The following lemma has been used in the proof of Proposition \ref{prop:good}.

\begin{lemma}\label{lem:saturated-submodule}
Keep notations as above. Suppose $L$ is a finite extension of $E$. Let $V$ be a $L$-linear representation of $G_K$ of dimension $d$.  Let $D_1$ be a rank 1 $\m$-submodule over $\r_{L'}$ of $D=\D_\rig^\dag(V)$. Fix an integer $n$ so that the degree of field extensions $[K_m:\Q(\epsilon_m)]$ are constant for all $m\geq n$,  and $D^{\rho(r_n)}_1\subset D^{\rho(r_n)}$. Then $D_1$ is saturated in $D$ if and only if $D^{\rho(r_n)}_1$ has nonzero image in $D_\Sen^n(V)_\tau$ via the composite
\[
\iota_n: D^{\rho(r_n)}=\D_\rig^{\dag,r_n}(V)\ra \D^{+,n}_{\dif}(V)\ra \D^n_{\Sen}(V)\ra\D^n_{\Sen}(V)_\tau
\]
for any $\tau\in \H_K$.
\end{lemma}
\begin{proof}
The ``only if'' part is obvious. It remains to prove the ``if" part. To do this, we apply induction for $\m$-modules defined in \cite{L07}.  Using the set up of \cite{L07}, $\Ind^{\Gamma_{\Q}}_{\Gamma_K}D_1$ and $\Ind^{\Gamma_{\Q}}_{\Gamma_K}D$ are $(\varphi,\Gamma_{\Q})$-modules over $\r_{\Q}\otimes_{\Q}L=\r_L$ of ranks $[K:\Q]$ and $d[K:\Q]$ respectively. Furthermore,  
\[
\Ind^{\Gamma_{\Q}}_{\Gamma_K}D=\D_{\rig,\Q}^\dag(\Ind_{G_{\Q}}^{G_K}V)
\] 
since inductions for $p$-adic representations are compatible with inductions for the associated $\m$-modules \cite[Proposition 2.1]{L07}. Then it suffices to show that $\Ind^{\Gamma_{\Q}}_{\Gamma_K}D_1$ is a saturated $(\varphi,\Gamma_{\Q})$-submodule of rank $h=[K:\Q]$ of $\Ind^{\Gamma_{\Q}}_{\Gamma_K}D$.

Suppose the contrary is true. Using \cite[Proposition 3.1]{L07}, we first deduce that as an $L\otimes_{\Q}\Q(\epsilon_n)$-module, the image of $\Ind^{\Gamma_{\Q}}_{\Gamma_K}D^{\rho(r_n)}_1$ in $\D_\Sen^n(\Ind^{\Gamma_{\Q}}_{\Gamma_K}V)$ can be generated by $h-1$ elements. This implies that the image has $L$-dimension $\leq(h-1)[\Q(\epsilon_n):\Q]$.
On the other hand, since $\Gamma_K$ acts transitively on the set of components $(L\otimes_{K}K_n)_\tau$, the image of $D^{\rho(r_n)}_1$ in $\D_\Sen^n(V)_\tau$ has $L$-dimension at least $[K_n:K]$. Now by the assumption on $n$, we have 
\[
\D_\Sen^n(\Ind^{\Gamma_{\Q}}_{\Gamma_K}V)=\Ind^{\Gamma_{\Q}}_{\Gamma_K}(\D_\Sen^n(V))=
\oplus_{\tau\in\H_K}\Ind^{\Gamma_{\Q}}_{\Gamma_K}(\D_\Sen^n(V)_{\tau}).
\] 
It follows that the image of $\Ind^{\Gamma_{\Q}}_{\Gamma_K}D^{\rho(r_n)}_1$ in $\D_{\Sen}^n(\Ind^{\Gamma_{\Q}}_{\Gamma_K}V)$ has $L$-dimension  at least 
\begin{equation*}
\begin{split}
h[\Gamma_{\Q}:\Gamma_K][K_n:K]=h(h/[K_n:\Q(\epsilon_n)])&[K_n:K]\\
&=h[K_n:\Q]/[K_n:\Q(\epsilon_n)]=h[\Q(\epsilon_n):\Q].
\end{split}
\end{equation*}
This yields a contradiction!
\end{proof}

\begin{remark}
In the case when $S=L$ is a finite extension of $\Q$, for any $\delta\in\widehat{\mathscr{T}}(S)$, Nakamura constructs a rank 1 $B$-pair $W(\delta)$ \cite{N09}. A short computation shows that $\r_L(\delta)$ is isomorphic to the $\m$-module corresponding to $W(\delta)$. Therefore, for an $L$-linear representation of $G_K$, being trianguline with parameters $(\delta_i)_{1\leq i\leq d}$ in the sense of Definition \ref{def:trianguline} is the same as being split trianguline in the sense of Nakamura with the same set of parameters.
\end{remark}

\subsection{Refined families}
\label{def:refined-representaiton}
The main goal of this subsection is prove the main result of this paper. That is, a family of refined $p$-adic representations of $G_K$ admits a global triangulation on a Zariski open and dense subspace of the base that contains all regular non-critical points. In what follows, we first give the definition of regular non-critical refined $p$-adic representations. 

\begin{defn}Let $L$ be a finite extension of $E$, and let $V$ be a $d$-dimensional crystalline $L$-linear representation of $G_K$ such that $\varphi^f$ acting on $D_{\mathrm{crys}}(V)$ has all its eigenvalues in $L^\times$.
\begin{enumerate}
\item[(1)]By a \emph{refinement} of $V$ we mean a $\varphi$-stable $K_0\otimes_{\Q}L$-filtration $\mathcal{F}=(\mathcal{F}_i)_{1\leq i\leq d}$ of $D_{\mathrm{crys}}(V)$:
\[
0=\mathcal{F}_0\subsetneq\mathcal{F}_1\cdots\subsetneq\mathcal{F}_d=D_{\mathrm{crys}}(V).
\]
In particular, $\dim\mathcal{F}_i=i$. 
\item[(2)]
For $\tau\in\H_K$, suppose the Hodge-Tate weights of $D_{\mathrm{dR}}(V)_\tau$ are
\[
k_{1,\tau}>k_{2,\tau}\cdots>k_{d,\tau}.
\]
We say the refinement $\mathcal{F}$ is \emph{$\tau$-non-critical} if
\begin{equation}\label{eq:non-critical}
D_{\mathrm{dR}}(V)_\tau=(K\otimes_{K_0}\mathcal{F}_i)_\tau\oplus \mathrm{Fil}^{k_{i+1,\tau}}(D_{\mathrm{dR}}(V)_{\tau})
\end{equation}
for all $1\leq i\leq d$. The refinement $\mathcal{F}$ is said to be \emph{non-critical} if it is $\tau$-non-critical for every $\tau\in\H_K$.
\item[(3)]
We denote by $\varphi_i$ the eigenvalue of $\varphi^f$ on $\mathcal{F}_i/\mathcal{F}_{i-1}$. We say the refinement $\mathcal{F}$ is \emph{regular} if for any $1\leq i\leq d$, $\varphi_1\cdots\varphi_i$ is an eigenvalue of $\varphi^f$ on $D_{\mathrm{crys}}(\wedge^iV)$ of multiplicity one.
\end{enumerate}
\end{defn}
The refinement $\mathcal{F}$ gives rise to an ordering $(\varphi_1,\dots,\varphi_d)$ of the  $\varphi^f$-eigenvalues on $D_{\mathrm{crys}}(V)$. If all these eigenvalues are distinct, any ordering of them uniquely gives rise to a refinement. For any $\tau\in\mathrm{H}_K$, the refinement $\mathcal{F}$ also gives rise to an ordering $(s_{1,\tau},\dots,s_{d,\tau})$ of $\{k_{1,\tau},\dots,k_{d,\tau}\}$, defined by the property that the jumps of the Hodge filtration of $D_{\mathrm{dR}}(V)_\tau$ induced on $(K\otimes_{K_0}\mathcal{F}_i)_\tau$ are $(s_{1,\tau},\dots,s_{i,\tau})$. It is straightforward to see that $\mathcal{F}$ is $\tau$-non-critical if and only if the associated ordering of the Hodge-Tate weights is $(k_{1,\tau},\dots,k_{d,\tau})$.




From now on, let $X$ be a reduced rigid analytic space over $E$, and let $E$ be a family of refined $p$-adic representations of $G_K$ of dimension $d$ over $X$ as in \S0.3. In the following, we retain the notations 
in \S0.3 and \S0.4. 

\begin{remark}\label{rmk:wedge}
Recall that in \S0.3 we define $\alpha_i=\prod_{j=1}^iF_j$ and $\eta_i=\prod_{j=1}^i\chi_j$ for $1\leq i\leq d$. If $V_X$ is a refined family of rank $d$, then for each $1\leq i\leq d$, the $i$-th exterior product $\wedge^iV_X$ is a weakly refined family with $F=\alpha_i$, the generalized Hodge-Tate weights 
\[
\{\kappa_I=\sum_{j\in I}\kappa_j\}_{|I|=i},
\]  
the biggest Hodge-Tate weight $\kappa_1+\cdots+\kappa_i$ and the same Zariski dense subset $Z$.
Hence $(\wedge^iV_X)(\eta_i^{-1})$ is a weakly refined family with generalized Hodge-Tate weights $\{\kappa_I-\kappa_{\{1,\dots,i\}}\}_{|I|=i}$ and $F=\alpha_i$. In particular, its biggest Hodge-Tate weight is 0.
\end{remark}


By Remark \ref{rmk:wedge}, we may apply Theorem \ref{thm:weakly-refined} to $(\wedge^iV_X)(\eta_i^{-1})$ to get the following:
\begin{prop}
For each $1\leq i\leq d$, the presheaf $\mathscr{D}_\rig^\dag((\wedge^iV_X)(\eta_i^{-1}))^{\varphi^f=\alpha_i, \Gamma=1}$ is a coherent sheaf on $X$.
\end{prop}


For each $1\leq i\leq d-1$, let $TQ_i(T)$ be the Sen polynomial for $(\wedge^iV_X)(\eta^{-1}_i)$, and let 
\[
P_i(k)=\prod_{j=0}^{k-1}Q_i(-j)
\] 
for $k\geq1$. 
The following proposition follows immediately from Proposition \ref{prop:weakly-refined-1}.

\begin{prop}\label{prop:good-criterion}
For $x\in X$, if there exist positive integers $k_i>v_K(\alpha_i(x))$ for each $1\leq i\leq d-1$ satisfying
\begin{equation}\label{eq:good}
(P_1(k_1)\cdots P_{d-1}(k_{d-1}))(x)\neq0,
\end{equation}
then $x\in X_s$ and $\dim\D_\rig^{\dag}((\wedge^iV_x)(\eta_i(x)^{-1}))_\sigma^{\varphi^f=\alpha_i(x),\Gamma=1}$ for each $1\leq i\leq d$ and $\sigma\in\mathrm{Gal}(K_0/\Q)$.
\end{prop}

\begin{prop}\label{prop:regular-local}
For $x\in X$, $1\leq i\leq d$ and $\sigma\in\mathrm{Gal}(K_0/\Q)$, if 
\[
\dim_{k(x)}\D_\rig^\dag((\wedge^iV_x^{\mathrm{ss}})(\eta_i(x)^{-1}))_\sigma^{\varphi^f=\alpha_i(x),\Gamma=1}=1,
\]
the coherent sheaf $\mathscr{D}_\rig^\dag((\wedge^iV_X)(\eta_i^{-1}))_\sigma^{\varphi^f=\alpha_i, \Gamma=1}$
is locally free of rank 1 around $x$, and 
\[
\mathscr{D}_\rig^\dag((\wedge^iV_X)(\eta_i^{-1}))^{\varphi^f=\alpha_i, \Gamma=1}_\sigma\otimes k(x)\ra\D_\rig^\dag((\wedge^iV_x)(\eta_i(x)^{-1})))_\sigma^{\varphi^f=\alpha_i(x),\Gamma=1}
\]
is an isomorphism.
\end{prop}
\begin{proof}
We conclude the proposition by applying Proposition \ref{prop:multiplicity-one} to the weakly refined family $(\wedge^iV_X)(\eta^{-1}_i)$.
\end{proof}

As already noted in \S0.4, the saturated locus $X_s$ of $V_X$ is a Zariski open and dense subspace of $X$. Therefore it reduces to show that $V_{X_s}$ admits a global triangulation on $X_s$ and $X_s$ contains all regular non-critical points. To this end, we will show that the triangulation locus of $V_{X_s}$ forms a reduced Zariski closed subspace of $X_s$.  The upshot is to note that for a sequence of crystalline periods of 
the successive exterior products of $V_x$ obtained by the previous results, the condition that it gives rise to a triangulation of $V_x$ is purely algebraic.  To make this statement precise, we introduce the following notions.

\begin{defn}
Let $A$ be a commutative ring with identity, and let $M$ be a free $A$-module of rank $d$.
\begin{enumerate}
\item[(1)] We call a free $A$-submodule $N\subseteq M$ of rank $c$ \emph{cofree} if $M/N$ is a free $A$-module of rank $d-c$. We call $m\in M$ \emph{cofree} if $Am$ is cofree. 

\item[(2)]Let $m\in M$ be cofree, and let $n\in \wedge^iM$ for some $1\leq i\leq d$. Suppose 
\[
m\wedge n=0
\] 
in $\wedge^{i+1}M$. Then there exists a unique $\overline{n}\in \wedge^{i-1}(M/Am)$ such that the wedge product of any lift of $\overline{n}$ in $\wedge^{i-1}M$ with $m$ is equal to $n$; we call $\overline{n}$ the \emph{quotient} of $n$ by $m$. Let $N$ be a free rank 1 $A$-submodule of $M$, and let $P$ be a free rank 1 $A$-submodule of 
$\wedge^i M$. If $N\wedge P=0$ in $\wedge^{i+1}M$, we define the quotient of $P$ by $N$ to be the $A$-submodule of $\wedge^{i-1}M$ generated by the quotient of any generator of $P$ by any generator of $N$.

\item[(3)]
For each $1\leq i\leq d$, let $N_i$ be a free rank 1 $A$-submodule of $\wedge^i M$. We say the sequence $N_1,\dots, N_{d}$ forms a \emph{chain} in $M$ if there exists an $A$-basis $e_1,\dots,e_d$ of $M$ such that 
\[
N_i=Ae_1\wedge\cdots\wedge e_i
\] 
for all $1\leq i\leq d$. In this case, the filtration
\[
\mathrm{Fil}_i(M)=\text{Span of $\{e_j\}_{0\leq j\leq i}$},\quad 1\leq i\leq d-1, 
\]
which is independent of the choice of the basis $\{e_1,\dots,e_d\}$, is called the \emph{associated filtration} of the chain $N_1,\dots, N_{d}$.

Let $m_i\in\wedge^iM$ for $1\leq i\leq d$. We say the sequence $m_1,\dots, m_{d}$ forms a \emph{chain} in $M$ if the sequence $Am_1,\dots, Am_d$ forms a chain. In this case, we call the associated filtration
of $Am_1,\dots, Am_d$ the \emph{associated filtration} of the chain $m_1,\dots, m_{d}$.
\end{enumerate}

\end{defn}

The following lemma is a simple exercise in linear algebra. 
\begin{lemma}\label{lem:chain-reduction}
The sequence $m_1,\dots,m_d$ forms a chain in $M$ if and only if the following hold.
\begin{enumerate}
\item[(1)]
$m_1$ is cofree.
\item[(2)]
$m_1\wedge m_i=0$ for $2\leq i\leq d$.
\item[(3)]
The sequence of quotients of $m_2,\dots,m_{d}$ by $m_1$ forms a chain in $M/Am_1$.
\end{enumerate}
\end{lemma}

\begin{lemma}\label{lem:chain-bezout}
Suppose $A$ is a B\'ezout domain, and let $M$ be a free $A$-module of rank $d$. For each $1\leq i\leq d$, let $m_i\in\wedge^iM$ be cofree. Now suppose $A\ra B$ is an injective map of commutative rings. Then the sequence $m_1,\dots,m_d$ forms a chain in $M$ if and only if it forms a chain in $M\otimes_AB$.
\end{lemma}
\begin{proof}
To show the ``if'' part of the lemma, we proceed by induction on $d$. The initial case is trivial. Suppose it is true for $d=k-1$ for some $k\geq2$. Now suppose $\rank M=k$ and the sequence $m_1,\dots,m_k$ forms a chain in $M\otimes_AB$. Then $m_1\wedge m_i=0$ in $\wedge^{i+1}(M\otimes_AB)$. Hence $m_1\wedge m_i=0$ in $\wedge^{i+1}M$ since the natural map 
\[
\wedge^{i+1}M\ra \wedge^{i+1}(M\otimes_AB)
\]
is injective. Furthermore, since $m_i$ is cofree in $\wedge^iM$, its quotient by $m_1$ is cofree in $\wedge^{i-1}(M/Am_1)$ by the B\'ezout property of $A$. We therefore conclude the lemma from Lemma \ref{lem:chain-reduction} and the inductive assumption.
\end{proof}

\begin{lemma}\label{lem:chain-point}
Let $L$ be a finite extension of $\Q$, and let $D$ be a $\m$-module over $\r_{L'}$ of rank $d$ (recall that $L'=L\otimes_{\Q}K_0'$). Then the following are true.
\begin{enumerate}
\item[(1)]Let $D_1$ be a rank 1 $\m$-submodule of $D$.
Then $D_1$ is cofree in $D$ if and only if $D_1^r$ is cofree in $D^r$ for some (hence all) sufficiently small $r$.
\item[(2)]
For $1\leq i\leq d$, let $D_i$ be a rank 1 $\m$-submodule of $\wedge^iD$.
Then the sequence $D_1,\dots, D_d$ forms a chain in $D$ if and only if the sequence $D_1^r,\dots, D_d^r$ forms a chain in $D^r$ for some (hence all) sufficiently small $r$.
\end{enumerate}
\end{lemma}
\begin{proof}
We deduce (1) from Lemma \ref{lem:representative-saturated}. We deduce (2) from (1) and Lemma \ref{lem:chain-bezout}.
\end{proof}

Now let $S$ be an affinoid algebra over $\Qp$.

\begin{lemma}\label{lem:chain-vector-bundle}
Let $I$ be a closed subinterval of $(0,\infty)$, and let $M^I_S$ be a vector bundle over $\r^I_{K_0'}\widehat{\otimes}_{\Q}S$ of rank $d$. For $1\leq i\leq d$, let $a_i$ be a global section of $\wedge^iM^I_S$ such that its image in $\wedge^iM^I_x$ is cofree for any $x\in M(S)$. Then the set of $x\in M(S)$ where the image of the sequence $a_1,\dots, a_{d}$ forms a chain in $M^I_x$ forms a reduced Zariski closed subspace of $M(S)$.
\end{lemma}
\begin{proof}
We proceed by induction on $d$. The case $d=1$ is trivial. Now suppose that the lemma is true for $d=k-1$ for some $k\geq2$, and that $M_S^I$ has rank $k$. By assumption, the image of $a_1$ in $M_x^I$ is cofree for any $x\in M(S)$. Hence $M_S^I/(\r^I_{K_0'}\widehat{\otimes}_{\Q}S)a_1$ is a vector bundle of rank $k-1$ over $\r^I_{K_0'}\widehat{\otimes}_{\Q}S$.  

Since $I$ is a closed interval, using Lemma \ref{lem:closed-interval}, we may suppose that both $M_S^I$ and $M_S^I/(\r^I_{K_0'}\widehat{\otimes}_{\Q}S)a_1$ are free over $\r^I_{K_0'}\widehat{\otimes}_{\Q}S$ by restricting on a finite covering of $M(S)$ by affinoid subdomains. Thus $a_1$ is cofree in $M_S^I$. Extend $\{a_1\}$ to a basis of $M_S^I$ over $\r_S^I$. We may expand $a_i$ using this basis. It is then straightforward to see that for each $2\leq i\leq k$, the set of $x\in M(S)$, where the image of $a_1\wedge a_i$ in $\wedge^{i+1}M_x^I$ is zero, forms a reduced Zariski closed subspace of $M(S_i)$ of $M(S)$. Furthermore, it follows that
\[
a_1\wedge a_i=0
\] 
in $\wedge^{i+1}M_{S_i}^I$ for $2\leq i\leq d$. 

Now let $M(S')$ be the intersection of all $M(S_i)$, and let 
$b_i\in\wedge^{i-1}(M_{S'}^I/(\r_{K_0'}^I\widehat{\otimes}_{\Q}S')a_1)$
be the quotient of $a_i$ by $a_1$. For $x\in M(S')$, since the image of $a_i$ is cofree in $\wedge^iM^I_x$, the image of $b_i$ in $\wedge^{i-1}(M_{x}^I/(\r_{K_0'}^I\otimes_{\Q}k(x))a_1(x))$ is cofree by the B\'ezout property of $\r_{K_0'}^I$. By Lemma \ref{lem:chain-reduction}, the desired subset of $M(S)$ is then the set of $x$ where the image of the sequence $b_2,\dots,b_k$ forms a chain in $(M^I_{S'}/(\r_{K_0'}^I\widehat{\otimes}_{\Q}S')a_1)_x$. We therefore conclude the case $d=k$ by the inductive assumption.
\end{proof}

\begin{lemma}\label{lem:chain-Zariski}
Let $D_S$ be a $\m$-module over $\r_{K_0'}\widehat{\otimes}_{\Q}S$ of rank $d$. For $1\leq i\leq d$, let $D_i\subset \wedge^iD_S$ be a rank 1 $\m$-submodule over $\r_{K_0'}\widehat{\otimes}_{\Q}S$.
If $D_i$ specializes to a rank 1 cofree $\r_{K_0'}\otimes_{\Q}k(x)$-submodule of $\wedge^iD_x$ for any $x\in M(S)$ and $1\leq i\leq d$, the set of $x\in M(S)$ where the image of the sequence $D_1,\dots, D_d$ forms a chain in $D_x$ forms a reduced Zariski closed subspace of $M(S)$.
\end{lemma}
\begin{proof}
By Lemma \ref{lem:chain-point}(2), the sequence $D_1,\dots, D_d$ forms a chain in $D_x$ if and only if $D_1^r,\dots,D_d^r$ forms a chain in $D_x^r$ for all sufficiently small $r$. By Lemma \ref{lem:chain-bezout}, the latter holds if and only if $D_1^{[r,r]},\dots,D_d^{[r,r]}$ forms a chain in $D_x^{[r,r]}$. We then deduce the lemma by Lemma \ref{lem:chain-vector-bundle}.
\end{proof}


In the following, for $1\leq i\leq d$, set
\[
N_{i,X}=\mathscr{D}_\rig^\dag((\wedge^iV_X)(\eta_i^{-1}))^{\varphi^f=\alpha_i,\Gamma=1}\otimes_{K_0\otimes_{\Qp}\OO_X}
\mathscr{D}_\rig^\dag(\eta_i),
\] 
which is a rank 1 $\m$-submodule of $\mathscr{D}_\rig^\dag(\wedge^iV_X)$ of type $\eta_i$.  Recall that in \S0.3 for $1\leq i\leq d$, we define the character $\delta_i:K^\times\ra\OO(X)^\times$ by setting $\delta_i|_{\OO_K^\times}=\eta_i$ and $\delta_i(\pi_K)=\alpha_i$. 

\begin{prop}\label{prop:chain-Zariski}
The triangulation locus of $V_X$ forms a reduced Zariski closed subspace of $X_s$. Furthermore, the sequence $N_1,\dots,N_d$ gives rise to a global triangulation of $V_X$ on the triangulation locus.
That is, for any affinoid subdomain $M(S)$ of the triangulation locus,  the sequence $N_1,\dots,N_d$ forms a chain in $\D_\rig^\dag(V_S)$ whose associated filtration is a triangulation of $\D_\rig^\dag(V_S)$ with parameters $(\delta_i/\delta_{i-1})_{1\leq i\leq d}$. 
\end{prop}
\begin{proof}
Note that by its definition, the triangulation locus of $V_X$ is exactly the set of $x\in X_s$ where the image of the sequence $N_1,\dots, N_d$ forms a chain in $\D_\rig^\dag(V_x)$. 
We then deduce the first statement from Lemma \ref{lem:chain-Zariski}.

Now let $M(S)$ be an affinoid subdomain of the triangulation locus of $V_X$, and suppose $\D_\rig^{\dag,s}(V_S)$ is defined for some $s>0$. Set $N_{i,X}^{(0)}=N_{i,X}$.
Since $N^{(0)}_{1,x}$ is a rank 1
saturated $\m$-submodule of $\D_\rig^\dag(V_x)$ for $x\in M(S)$, by Proposition \ref{prop:quotient},
\[
D^{(1)}_{S}=\D_\rig^\dag(V_S)/N^{(0)}_{1,S}
\] 
is a $\m$-module of rank $d-1$ over $\r_{K_0'}\widehat{\otimes}_{\Q}S$
with a representative vector bundle 
\[
D_{S}^{(1),\rho(s)}=\D_\rig^{\dag,s}(V_S)/N_{1,S}^{(0),\rho(s)}.
\]
 By Lemma \ref{lem:closed-interval}, we choose a finite covering $\{M(S_j)\}_{j\in J}$ of $M(S)$ by affinoids such that all the vector bundles $N^{(0),[\rho(s)/p^f,\rho(s)]}_{i,S_j}$ and $D_{S_j}^{(1),[\rho(s)/p^f,\rho(s)]}$ are free. Since $M(S)$ is contained in the triangulation locus, it follows that
 \[
 N^{(0),[\rho(s)/p^f,\rho(s)]}_{1,S_j}\wedge N^{(0),[\rho(s)/p^f,\rho(s)]}_{i,S_j}=0
 \]
 for $2\leq i\leq d$ and $j\in J$.
Taking the quotient of $N^{(0),[\rho(s)/p^f,\rho(s)]}_{i,S_j}$ by $N^{(0),[\rho(s)/p^f,\rho(s)]}_{1,S_j}$ for each $j$ and gluing these quotients, we obtain a vector bundle $N_{i,S}^{(1),[\rho(s)/p^f,\rho(s)]}$ over $\r_{K_0'}\widehat{\otimes}_{\Q}S$. Furthermore, note that each $N_{i,S_j}^{(1),[\rho(s)/p^f,\rho(s)]}$ admits a basis $e$ satisfying 
\[
\varphi^f(e)=(\alpha_i/\alpha_1)(e).
\] 
Therefore, we can extend $N_{i,S}^{(1),[\rho(s)/p^f,\rho(s)]}$ to a rank 1 $\m$-submodule, which is of type $\delta_i/\delta_1$, of $D_{S}^{(1)}$.

We may iterate the above procedure as follows. Suppose after the $k$-th step, we have a $\m$-module $D^{(k)}_{S}$ over $\r_{K_0'}\widehat{\otimes}_{\Q}S$ of rank $d-k$ and a rank 1 $\m$-submodule $N_{i,S}^{(k)}$ of $\wedge^{i-k}D_i$ of type ${\delta_i/\delta_k}$, which specializes to a saturated $\m$-submodule of $\wedge^{i-k}D_x^{i}$ for any $x\in X$, for each $k+1\leq i\leq d$.  Now let 
\[
D_{S}^{(k+1)}=D^{(k)}_S/N^{(k)}_{k+1,S}.
\]
It is then a $\m$-module over $\r_{K_0'}\widehat{\otimes}_{\Q}S$ of rank $d-k-1$ by Proposition \ref{prop:quotient}. Then by the same argument as above, for each $k+2\leq i\leq d$, we get a rank 1 $\m$-submodule $N_{i,S}^{(k+1)}$ of $\wedge^{i-k-1}D^{(k+1)}_S$ of type ${\delta_i/\delta_{k+1}}$, which specializes to saturated $\m$-submodule of $\wedge^{i-k-1}D^{(k+1)}_x$ for $x\in M(S)$. 

Now let $\mathrm{Fil}_i(\D_\rig^\dag(V_S))=\ker(\D_\rig^\dag(V_S)\ra D^{(i)}_{S})$ for $1\leq i\leq d$. It  follows from the above procedure that $(\mathrm{Fil}_i(\D_\rig^\dag(V_S)))_{1\leq i\leq d}$ is a triangulation of $\D_\rig^\dag(V_S)$ with successive quotients
\[
\mathrm{Fil}_{i+1}(\D_\rig^\dag(V_S))/\mathrm{Fil}_i(\D_\rig^\dag(V_S))\cong N^{(i)}_{i+1,S}
\]
for $0\leq i\leq d-1$. The yields the second statement of the theorem.
\end{proof}

It remains to show that the triangulation locus contains all regular non-critical points. 

\begin{prop}\label{prop:large-class}
For $x\in X$, if $V_x$ satisfies 
\[
\dim_{k(x)}\D_\rig^\dag((\wedge^iV_x^{\mathrm{ss}})(\eta_i(x)^{-1}))_\sigma^{\varphi^f=\alpha_i(x),\Gamma=1}=1
\]
for all $1\leq i\leq d-1$ and
$\sigma\in\mathrm{Gal}(K_0/\Q)$, and $\D_\rig^\dag(V_x)$ admits a triangulation $(\mathrm{Fil}_i(\D_\rig^\dag(V_x)))_{1\leq i\leq d}$ with parameters
$(\delta_{i}/\delta_{i-1})(x)_{1\leq i\leq d}$, then the sequence 
\[
(\mathscr{D}_\rig^\dag((\wedge^iV_X)(\eta_i^{-1}))^{\varphi^f=\alpha_i,\Gamma=1}\otimes_{K_0\otimes_{\Q}\OO_X}\D_\rig^\dag(\eta_i(x)))_{1\leq i\leq d}
\]
forms a chain in $\D_\rig^\dag(V_x)$ and its associated filtration is just $(\mathrm{Fil}_i)_{1\leq i\leq d}$.

\end{prop}
\begin{proof}
By Proposition \ref{prop:multiplicity-one}, $(\mathscr{D}_\rig^\dag((\wedge^iV_X)(\eta_i^{-1}))^{\varphi^f=\alpha_i,\Gamma=1}_\sigma$ is locally free of rank 1 around $x$, and 
\[
(\mathscr{D}_\rig^\dag((\wedge^iV_X)(\eta_i^{-1}))^{\varphi^f=\alpha_i,\Gamma=1}_\sigma\otimes k(x)
\ra \D_\rig^\dag((\wedge^iV_x)(\eta_i(x)^{-1}))^{\varphi^f=\alpha_i(x),\Gamma=1}_\sigma
\]
is an isomorphism for all $i$ and $\sigma$. Thus
\[
(\mathscr{D}_\rig^\dag((\wedge^iV_X)(\eta_i^{-1}))^{\varphi^f=\alpha_i,\Gamma=1}\otimes k(x)
\ra \D_\rig^\dag((\wedge^iV_x)(\eta_i(x)^{-1}))^{\varphi^f=\alpha_i(x),\Gamma=1}
\]
is an isomorphism. 
By assumption, $\D_\rig^\dag(V_x)$ admits a triangulation $(\mathrm{Fil}_i)_{1\leq i\leq d}$ with parameters $(\delta_i/\delta_{i-1})_{1\leq i\leq d}$. In particular, $\D_\rig^\dag(V_x)$ contains a rank 1 $\m$-submodule 
\[
D=\mathrm{Fil}_1(\D_\rig^\dag(V_x))\cong\r_{k(x)}(\delta_1).
\] 
Recall that $\r_{k(x)}(\delta_1)$ is defined to be
\[
D_{\alpha_1(x)}\otimes_{K_0\otimes_{\Q}k(x)}\D_{\rig}^\dag(\eta_1(x)). 
\]
We then deduce
\[
\dim_{k(x)}(D(\eta_1^{-1}(x)))_\sigma^{\varphi^f=\alpha_1(x),\Gamma=1}\geq1.
\] 
This forces
\[
(D(\eta_1^{-1}(x)))_\sigma^{\varphi^f=\alpha_1(x),\Gamma=1}=\D_\rig^\dag (V_x(\eta^{-1}_1(x)))^{\varphi^f=\alpha_1(x),\Gamma=1}_\sigma
\] 
for all $\sigma$. Hence
\[
(D(\eta_1^{-1}(x)))^{\varphi^f=\alpha_1(x),\Gamma=1}=\D_\rig^\dag (V_x(\eta^{-1}_1(x)))^{\varphi^f=\alpha_1(x),\Gamma=1}.
\] 
It follows that the image of the map
\[
\mathscr{D}_\rig^\dag((\wedge^iV_X)(\eta_1^{-1}))^{\varphi^f=\alpha_1,\Gamma=1}\otimes_{K_0\otimes_{\Q}k(x)}\D_\rig^\dag(\eta_1(x))\ra \D_\rig^\dag(V_x)
\] 
is exactly $\mathrm{Fil}_1(\D_\rig^\dag(V_x))$. By a similar argument, we deduce that the image of the map
\[
\mathscr{D}_\rig^\dag((\wedge^iV_X)(\eta_i^{-1}))^{\varphi^f=\alpha_i,\Gamma=1}\otimes_{K_0\otimes_{\Qp}k(x)}
\D_\rig^\dag(\eta_i(x))\ra \wedge^i\D_\rig^\dag(V_x)
\] 
is exactly $\wedge^i\mathrm{Fil}_i(\D_\rig^\dag(V_x))$. This yields the desired result.
\end{proof}

\begin{theorem}\label{thm:good-triangulation}
The triangulation locus of $V_X$ contains all the points which satisfy the assumption of Proposition \ref{prop:large-class}.  In particular, the triangulation locus of $V_X$ contains all regular non-critical points. As a consequence, the triangulation locus of $V_X$ coincides with the saturated locus $X_s$, which is a Zariski open and dense subspace of $X$, and the $\m$-modules
\[
\D_\rig^\dag((\wedge^iV_S)(\eta_i^{-1}))^{\varphi^f=\alpha_i,\Gamma=1}\otimes_{K_0\otimes_{\Qp}S}
\D_\rig^\dag(\eta_i)
\]
for $1\leq i\leq d$ give rise to a triangulation of $\D_\rig^\dag(V_S)$ with parameters $(\delta_i/\delta_{i-1})_{1\leq i\leq d}$ on any affinoid subdomain $M(S)$ of $X_s$.
\end{theorem}

\begin{proof}
The first assertion is an immediate consequence of Proposition \ref{prop:large-class}. Furthermore, it is clear that regular non-critical points satisfy the assumption of Proposition \ref{prop:large-class}. Thus they belong to the triangulation locus of $V_X$. On the other hand, note that $X_s$ is the intersections of the saturated loci of the weakly refined families $\wedge^iV_X$ for all $1\leq i\leq d$. Hence it is Zariski open by Propositions \ref{prop:good}. Since the set of regular non-critical points is Zariski dense in $X$, it follows that it is Zariski dense in $X_s$, and $X_s$ is Zariski dense in $X$. We then conclude the rest of the theorem by Proposition \ref{prop:chain-Zariski}.
\end{proof}

As mentioned in the introduction, it is expected that all non-critical points belong to the locus of global triangulation. Regarding this point, we make the following conjecture.
\begin{conj}
For $x\in X$, if $\D_\rig^\dag(V_x)$ admits a triangulation $(\mathrm{Fil}_i(\D_\rig^\dag(V_x)))_{1\leq i\leq d}$ with parameters $((\delta_{i}/\delta_{i-1})(x))_{1\leq i\leq d}$ such that 
\[
\dim_{k(x)}(\D_\rig^\dag(V_x)/\mathrm{Fil}_{i-1}(\D_\rig^\dag(V_x)))_\sigma^{\varphi^f=\alpha_i/\alpha_{i-1},\Gamma=\eta_i/\eta_{i-1}}=1
\] 
for all $1\leq i\leq d$ and $\sigma\in\mathrm{Gal}(K_0/\Q)$, then $x$ belongs to the triangulation locus. 
\end{conj}

\subsection{Specializations of refined families}

\begin{lemma}\label{lem:phi-gamma-submodule}
Let $L$ be a finite extension of $E$, and let $D$ be a $\m$-module over $\r_{L'}$ ($L'=L\otimes_{\Q}K_0'$). Let $D_1$ be a $\m$-submodule of $D$, and let $D_1'$ be its saturation in $D$. Then there exists a positive integer $k$ such that $t^kD_1'\subset D_1$.
\end{lemma}
\begin{proof}
Note that $\Ind^{\Gamma_{\Q}}_{\Gamma_K}D_1'$ is the saturation of $\Ind^{\Gamma_{\Q}}_{\Gamma_K}D_1$ in $\Ind^{\Gamma_{\Q}}_{\Gamma_K}D$. By \cite[Proposition 3.1]{L07}, there exists a positive integer $k$ such that $t^k\Ind^{\Gamma_{\Q}}_{\Gamma_K}D_1'\subset\Ind^{\Gamma_{\Q}}_{\Gamma_K}D_1$. This yields the lemma.
\end{proof}

\begin{theorem}\label{thm:trianguline}
For $x\in X$, the $p$-adic representation $V_x$ is trianguline.
\end{theorem}
\begin{proof}
Let $M(S)$  be an affinoid neighborhood of $x$. 
By the same argument as in the proof of Proposition \ref{prop:multiplicity-one}, there exists a proper birational map $\pi:X'\ra M(S)$ such that for all $1\leq i\leq d$ and $\sigma\in\mathrm{Gal}(K_0/\Q)$, the coherent sheaves
\[
\mathscr{D}_\rig^\dag((\wedge^iV_{X'})(\pi^{*}(\eta_i^{-1})))^{\varphi^f=\pi^*(\alpha_i),\Gamma=1}_\sigma
\] 
are locally free of rank 1 and the natural map
\[
\mathscr{D}_\rig^\dag((\wedge^iV_{X'})(\pi^{*}(\eta_i^{-1})))^{\varphi^f=\pi^*(\alpha_i),\Gamma=1}_\sigma\otimes k(x')
\ra\D_\rig^\dag((\wedge^iV_{x'})(\pi^{*}(\eta_i^{-1})(x')))_\sigma
\]
is injective for any $x'\in X'$. By the previous lemma,  
\[
\mathscr{D}_\rig^\dag((\wedge^iV_{X'})(\pi^{*}(\eta_i^{-1})))^{\varphi^f=\pi^*(\alpha_i),\Gamma=1}\otimes_{K_0\otimes_{\Q}\OO_{X'}}\mathscr{D}_\rig^\dag(\pi^*(\eta_i))[1/t]
\] 
specializes to a rank 1 saturated $(\r_{K_0'}\otimes_{\Q}k(x))[1/t]$-submodule in $\D_\rig^\dag(\wedge^iV_{x'})[1/t]$.

Pick some $s\geq s(V_S)$ such that $s\neq r_n$ for any $n\in\mathbb{N}$, and put $r=\rho(s)$. It follows that $t$ has no roots in the annulus $v_p(T)=r$. Thus $t$ is invertible in $\r_{K_0'}^{[r,r]}$. Hence the natural map
\[
\r^r_{K_0'}[1/t]\ra\r_{K'}^{[r,r]}
\] 
is injective. Now let $M(S')$ be an affinoid subdomain of $X'$. Set $D_{S'}=\D_\rig^\dag(V_{S'})$ and $D_{i,S'}=\D_\rig^\dag(\pi^*(\eta_i)|_{M(S')})$ for $1\leq i\leq d$. We claim that the sequence 
\[
(\D_\rig^\dag((\wedge^iV_{S'})(\pi^{*}(\eta_i^{-1})))^{\varphi^f=\pi^*(\alpha_i),\Gamma=1}\otimes_{K_0\otimes_{\Q}S'}D^{[r,r]}_{i,x})_{1\leq i\leq d} 
\] 
forms a chain in $D_{x'}^{[r,r]}$ for any $x'\in M(S')$. In fact, by Lemma \ref{lem:chain-vector-bundle}, the set of such $x'$ forms a reduced Zariski closed subspace of $M(S')$. On the other hand, suppose $\pi$ is an isomorphism on a Zariski dense and open subset $U$ of $X'$. By Theorem \ref{thm:good-triangulation}, after shrinking $U$, we may further suppose that $\pi(U)$ is contained in the triangulation locus of $M(S)$. It is then clear that the sequence 
\[
(\D_\rig^\dag((\wedge^iV_{S'})(\pi^{*}(\eta_i^{-1})))^{\varphi^f=\pi^*(\alpha_i),\Gamma=1}\otimes_{K_0\otimes_{\Q}S'}D^{[r,r]}_{i,x})_{1\leq i\leq d}
\] 
forms a chain in $D_{x'}^{[r,r]}$ for any $x'\in U\cap M(S')$. Since $U\cap M(S')$ is Zariski dense in $M(S')$, the claim follows. 

The claim and Lemma \ref{lem:chain-point} then imply that the image of the sequence 
\[
(\D_\rig^\dag((\wedge^iV_{S'})(\pi^{*}(\eta_i^{-1})))^{\varphi^f=\pi^*(\alpha_i),\Gamma=1}\otimes_{K_0\otimes_{\Q}S'}D^{r}_{i,S'}[1/t])_{1\leq i\leq d}
\]  
forms a chain in $\D_{\rig}^{\dag,s}(V_{x'})[1/t]$ for any $x'\in M(S')$.  Hence $\D_\rig^\dag(V_{x'})[1/t]$ is triangulable (in the obvious sense). This yields that $\D_\rig^\dag(V_{x'})$ is triangulable. It remains to show that $\D_\rig^\dag(V_{x})$ is triangulable. Without loss of generality we may assume that $k(x')$ is Galois over $k(x)$ for some $x'$ in the preimage of $x$. In this case, define a $G=\mathrm{Gal}(k(x')/k(x))$-action on
\[
\D_\rig^\dag(V_{x'})=\D_\rig^\dag(V_{x})\otimes_{k(x)}k(x')
\]
by setting $g(a\otimes b)=a\otimes g(b)$. It then follows that the triangulation of $\D_\rig^\dag(V_{x'})$ maps onto a triangulation of $\D_\rig^\dag(V_{x})$ via the projection 
\[
\frac{1}{|G|}\sum_{g\in G}g: \D_\rig^\dag(V_{x'})\ra\D_\rig^\dag(V_{x}).
\]

\end{proof}

\subsection{Application to the eigencurve}
Fix a positive integer $N$ which is prime to $p$.  Let $S$ be the set of places of $\mathbb{Q}$
consisting of the infinite place and the places dividing $pN$. Let $\overline{V}$ be a two dimensional $G_{\mathbb{Q},S}$-representation over a finite field of characteristic $p$, which is $p$-modular in the sense of \cite{CM}. Let $R_{\overline{V}}$ be the universal deformation ring of the pseudo representation associated to $\overline{V}$. Let $X_{\overline{V}}$ be the generic fiber of $\mathrm{Spf}(R_{\overline{V}})$, which is a rigid analytic space over $\Q$. By the works of Coleman-Mazur \cite{CM} and Buzzard \cite{Bu}, there is a $\Q$-rigid analytic curve $\mathcal{C}\subset X_{\overline{V}}\times \mathbb{G}_m$ whose $\mathbb{C}_p$-valued points correspond bijectively to overconvergent eigenforms of tame level $N$, which are of finite slope, and whose residual Galois representation have the same semi-simplification as $\overline{V}$. We further assume that $\mathcal{C}$ belongs to the cuspidal part of the eigencurve. That is, the overconvergent modular forms parametrized by $\mathcal{C}$ are all cuspidal. 

Let 
\[
T:G_{\mathbb{Q},S}\ra \OO(\mathcal{C})
\]
be the pseudo representation obtained by pulling back the universal pseudo representation of 
$G_{\mathbb{Q},S}$ on $X_{\overline{V}}$ via the composite
\begin{equation}\label{eq:eigencurve}
\mathcal{C}\ra X_{\overline{V}}\times\mathbb{G}_m\ra X_{\overline{V}}.
\end{equation}
Let $\alpha\in\OO(\mathcal{C})^\times$ denote the function of $U_p$-eigenvalue. 
Let $\kappa: \mathcal{C}\ra\mathcal{W}$ be the weight  map. We normalize $\kappa$ in such a way that if $x\in\mathcal{C}$ is a classical eigenform of weight $k$, then $\kappa(x)=k-1$. 

Let $\widetilde{\mathcal{C}}$ denote the normalization of $\mathcal{C}$. By \cite{CM}, there exists a family of $p$-adic representations of $G_{\mathbb{Q},S}$ of dimension 2 over $\mathcal{C}$ whose associated pseudo representation is isomorphic to the pullback of $T$ via $\widetilde{\mathcal{C}}\ra\mathcal{C}$.  Let $V_{\widetilde{C}}$ be the dual of this family of $p$-adic representations. Let $\widetilde{\alpha}\in\OO(\widetilde{\mathcal{C}})^\times$ denote the pullback of $\alpha$ via $\widetilde{\mathcal{C}}\ra\mathcal{C}$. Let $\widetilde{\kappa}:\widetilde{\mathcal{C}}\ra\mathcal{C}$ be the composite of $\kappa$ with $\widetilde{\mathcal{C}}\ra\mathcal{C}$. Let $\widetilde{Z}$ be the set of classical points $z\in\mathcal{C}$ such that $V_z$ is crystalline with distinct crystalline  Frobenius eigenvalues.  By Coleman's classicality theorem, it is straightforward to see that  $V_{\widetilde{\mathcal{C}}}$ is a family of $2$-dimensional weakly refined $p$-adic representations together with $\kappa_1=0,\kappa_2=-\widetilde{\kappa}, F=\widetilde{\alpha}, Z=\widetilde{Z}$.

\begin{prop}
The coherent sheaf $\mathscr{D}_\rig^\dag(V_{\mathcal{\widetilde{C}}})^{\varphi=F,\Gamma=1}$ is invertible, and its image in $\D_\rig^\dag(V_x)$ is nonzero for any $x\in\widetilde{\mathcal{C}}$. As a consequence, $V_x$ is trianguline for any $x\in \widetilde{\mathcal{C}}$.
\end{prop}
\begin{proof}
Let $M(S)$ be an affinoid subdomain of $\tilde{\mathcal{C}}$. Let $k$ be a positive integer such that 
\[
k>\log_p|F^{-1}|
\] 
in $S$. By Theorems \ref{thm:weak-refined-family} and \ref{thm:fs-sheaf}, the map
 \[
 \D_\rig^\dag(V_S)^{\varphi=F,\Gamma=1}\ra
(\D_\dif^{+,n}(V_S)/(t^k))^\Gamma
\] 
is an isomorphism. Note that $(\D_\dif^{+,n}(V_S)/(t^k))^\Gamma$ is a finite torsion-free $S$-module by Proposition \ref{prop:torsion-free}.  Hence it is a locally free $S$-module because $S$ is smooth and 1-dimensional. Thus $ \D_\rig^\dag(V_S)^{\varphi=F,\Gamma=1}$ is a locally free $S$-module. Furthermore, by Proposition \ref{prop:good}, it is locally free of rank 1 on a Zariski open and dense subspace of $M(S)$. Hence it is locally free of rank 1 on $M(S)$, yielding the first statement of the theorem. 

For the second statement, by Proposition \ref{prop:torsion-free}, $(\D_\dif^{+,n}(V_S)/(t^k))/(\D_\dif^{+,n}(V_S)/(t^k))^\Gamma$ is finite and torsion-free over $S$ as well. This implies that for any $x\in M(S)$, the natural map
\[
(\D_\dif^{+,n}(V_S)/(t^k))^\Gamma\otimes k(x)\ra \D_\dif^{+,n}(V_x)/(t^k)
\] 
is injective. It follows that $\D_\rig^\dag(V_S)^{\varphi=F,\Gamma=1}\otimes k(x)\ra \D_\rig^\dag(V_x)$ is injective as well.
\end{proof}

\begin{prop}\label{prop:eigencurve}
For $x\in \tilde{\mathcal{C}}$, $x$ is not saturated if and only if $V_x$ satisfies one of the following two disjoint conditions:
\begin{enumerate}
\item[(1)]
The weight $\kappa(x)$ is a positive integer and $v_p(F(x))>\kappa(x)$. As a consequence, $V_x$ belongs to $\mathscr{S}^{\mathrm{ng}}_*\cap\mathscr{S}_*^{\mathrm{HT}}$ in the sense of \cite{C12}; hence $V_x$ is irreducible, Hodge-Tate and non-de Rham. Furthermore, the image of $t^{-\kappa(x)}\mathscr{D}_\rig^\dag(V_{\widetilde{\mathcal{C}}})^{\varphi=F,\Gamma=1}$ generates a rank 1 saturated $\m$-submodule in $\D_\rig^\dagger(V_x)$.
\item[(2)]The weight $\kappa(x)$ is a positive integer and $v_p(F(x))=\kappa(x)$, and $V_x$ has a rank 1 subrepresentation $V_x'$ which is crystalline with Hodge-Tate weight $-\kappa(x)$. Furthermore, in this case, the image of $\mathscr{D}_\rig^\dag(V_{\widetilde{\mathcal{C}}})^{\varphi=F,\Gamma=1}$ in $\D_\rig^\dag(V_x)$ is $k(x)\cdot t^{\kappa(x)}e'$ where $e'$ is a canonical basis of $\D_\rig^\dag(V_x')$.
\end{enumerate}
In case (2), if $x\in Z$, then it is critical. Hence $V_x$ is split. Suppose $V_x=V_1\oplus V_2$ where $V_1$ has Hodge-Tate weight $0$ and $V_2$ has Hodge-Tate weight $-\kappa(x)$. Then the image of $\mathscr{D}_\rig^\dag(V_{\widetilde{\mathcal{C}}})^{\varphi=F,\Gamma=1}$ in $\D_\rig^\dag(V_x)$ is $k(x)\cdot t^{\kappa(x)}e_2$ where $e_2$ is a canonical basis of $\D_\rig^\dag(V_2)$.
\end{prop}

\begin{proof}
Suppose that $x$ is not saturated. Let $D$ be the saturation of the rank 1 $\m$-submodule of $\D_\rig^\dagger(V_x)$ generated by
$\mathscr{D}_\rig^\dag(V_{\mathcal{\widetilde{C}}}))^{\varphi=F,\Gamma=1}\otimes k(x)$.
Suppose 
\[
\mathscr{D}_\rig^\dag(V_{\mathcal{\widetilde{C}}})^{\varphi=F,\Gamma=1}\otimes k(x)=k(x)\cdot t^ke
\] 
for some positive integer $k$ and canonical basis $e$ of $D$. Thus the Hodge-Tate weight of $D$ is $-k$, yielding that $\kappa(x)=k$ is a positive integer. By Kedlaya's slope theory, $D$ has nonnegative slope, yielding that $v_p(\alpha(x))\geq \kappa(x)$. If the inequality is strict, then $V_x$ satisfies the condition $(1)$. If $v_p(\alpha(x))= \kappa(x)$, it is straightforward to see that $V_x$ satisfies the condition (2). Furthermore, if $x\in Z$, it is clear that $x$ is critical.

For the converse, suppose $V_x$ satisfies (1). If it is saturated, then it follows from Colmez's classification of $2$-dimensional irreducible trianguline representations of $G_{\Q}$ \cite[\S3.3]{C07} that $\D_\rig^\dag(V_x)$  belongs to $\mathscr{S}_{+}^{\mathrm{ncl}}$. However, by \cite[Proposition 3.5]{C07}, we know that all $2$-dimensional triangulable $\m$-modules belonging to $\mathscr{S}_{+}^{\mathrm{ncl}}$ are non-\'etale. This makes a contradiction. Now suppose $V_x$ satisfies (2). Note that $V_x/V_x'$ has Hodge-Tate weight $0$. Thus if the image of $\mathscr{D}_\rig^\dag(V_{\mathcal{\tilde{C}}})^{\varphi=F,\Gamma=1}\otimes k(x)$ in $\D_\rig^\dagger(V_x/V_x')$ is nonzero, it generates a rank 1 $\m$-submodule which is of Hodge-Tate weight $0$ and positive slope, yielding a contradiction. Therefore  $\mathscr{D}_\rig^\dag(V_{\mathcal{\tilde{C}}})^{\varphi=\alpha,\Gamma=1}\otimes k(x)$ maps into $\D_\rig^\dagger(V'_x)$. It then follows that the image is of the given form.
\end{proof}

In the case when $\overline{V}$ is an absolutely irreducible $G_{\mathbb{Q},S}$-representation, $R_{\overline{V}}$ coincides with the universal deformation ring of $\overline{V}$\footnote{In this case, $\mathcal{C}$ automatically belongs to the cuspidal part of the eigencurve.}. Let $V_{\mathcal{C}}$ be the dual of the pullback of the universal representation of $G_{\mathbb{Q},S}$ on $R_{\overline{V}}$ via (\ref{eq:eigencurve}). Let $Z$ be the set of classical points $z\in\mathcal{C}$ such that $V_z$ is crystalline with distinct crystalline Frobenius eigenvalues.  Then $V_{\mathcal{C}}$ is a family of $2$-dimensional weakly refined $p$-adic representations over $\mathcal{C}$ together with $\kappa_1=0, \kappa_2=-\kappa, F=\alpha$ and $Z$. Similarly, we have the following result.

\begin{theorem}\label{thm:eigencurve}
For any $x\in\mathcal{C}$, $\mathscr{D}_\rig^\dag(V_{\mathcal{C}})^{\varphi=F,\Gamma=1}$ is locally free of rank 1 around $x$ unless $\kappa(x)=0$, and $V^{\mathrm{ss}}_x$ is crystalline and satisfies $\dim D_{\mathrm{crys}}(V^{\mathrm{ss}}_x)^{\varphi=F(x)}=2$. If $x$ is not of this form, it is not saturated if and only if it satisfies one of the following two disjoint conditions:
\begin{enumerate}
\item[(1)]The weight $\kappa(x)$ is a positive integer and $v_p(F(x))>\kappa(x)$. As a consequence, $V_x$ belongs to $\mathscr{S}^{\mathrm{ng}}_*\cap\mathscr{S}_*^{\mathrm{HT}}$ in the sense of \cite{C12}; hence $V_x$ is irreducible, Hodge-Tate and non-de Rham. Furthermore, in this case $t^{-\kappa(x)}\mathscr{D}_\rig^\dag(V_{\mathcal{C}})^{\varphi=F,\Gamma=1}$ generates a rank 1 saturated $\m$-submodule in $\D_\rig^\dagger(V_x)$.
\item[(2)]The weight $\kappa(x)$ is a positive integer and $v_p(F(x))=\kappa(x)$, and $V_x$ has a rank 1 subrepresentation $V_x'$ which is crystalline with Hodge-Tate weight $-\kappa(x)$. Furthermore, in this case, the image of $\mathscr{D}_\rig^\dag(V_{\mathcal{C}})^{\varphi=F,\Gamma=1}$ in $\D_\rig^\dag(V_x)$ is $k(x)\cdot t^{\kappa(x)}e'$ where $e'$ is a canonical basis of $\D_\rig^\dag(V_x')$.

\item[(2')]In case (2), if $x\in Z$, then it is critical. Furthermore, suppose that $V_x=V_1\oplus V_2$ where $V_1$ has Hodge-Tate weight $0$ and $V_2$ has Hodge-Tate weight $-\kappa(x)$. Then the image of $\mathscr{D}_\rig^\dag(V_{\mathcal{C}})^{\varphi=F,\Gamma=1}$ in $\D_\rig^\dag(V_x)$ is $k(x)\cdot t^{\kappa(x)}e_2$ where $e_2$ is a canonical basis of $\D_\rig^\dag(V_2)$.
\end{enumerate}
\end{theorem}
\begin{proof}
If $\dim_{k(x)}D_{\mathrm{crys}}(V^{\mathrm{ss}}_x)^{\varphi=F(x)}\leq 1$, $\mathscr{D}_\rig^\dag(V_{\mathcal{C}})^{\varphi=F,\Gamma=1}$ is locally free of rank 1 around $x$ by Proposition \ref{prop:multiplicity-one}. Thus if $(\mathscr{D}_\rig^\dag(V_{\mathcal{C}}))^{\varphi=F,\Gamma=1}$ is not locally free of rank 1 around $x$,  then $D_{\mathrm{crys}}(V^{\mathrm{ss}}_x)^{\varphi=F(x)}$ is of dimension 2.  Furthermore, in this case, $\kappa(x)=0$ by the weak admissibility of $D_{\mathrm{crys}}(V^{\mathrm{ss}}_x)$. We deduce the rest of the theorem by the same argument as in the proof of Proposition \ref{prop:eigencurve}.
\end{proof}

\begin{remark}
By Theorem \ref{thm:eigencurve}, if $\mathscr{D}_\rig^\dag(V_{\mathcal{C}})^{\varphi=F,\Gamma=1}$ is not locally free around $x$, then the 
weak admissibility of $D_{\mathrm{crys}}(V^{\mathrm{ss}}_x)$ implies
$v_p(F(x))=0$; hence $x$ is ordinary. Furthermore, it follows that the weight character of $x$ is crystalline (hence unramified) of Hodge-Tate weight $0$. By the spectral theory of $U_p$, we know that the set of those $x$ is finite. 

We conjecture that there is no such $x$, i.e. $\mathscr{D}_\rig^\dag(V_{\mathcal{C}})^{\varphi=F,\Gamma=1}$ is everywhere locally free of rank 1 over the eigencurve. In fact, for $p\geq 5$ and $N=1$, by the virtue of a classical result of Mazur-Wiles \cite[\S8, Proposition 2]{MW}, we see that if the weight character of $x$ is non-trivial on the torsion subgroup of $\mathbb{Z}_p$, $V_\mathcal{C}$ is an extension of a ramified infinite order character by an unramified character around $x$. It is then straightforward to see that around $x$, $\mathscr{D}_\rig^\dag(V_{\mathcal{C}})^{\varphi=F,\Gamma=1}$ is locally free of rank 1  around $x$ and gives rise to the desired  global triangulation. 
We expect that an analogue of the result of Mazur-Wiles holds for general $p$ and $N$; this would confirm most of our conjecture. 
\end{remark}

\end{document}